\documentclass[11pt, a4paper]{amsart}

\usepackage[text={14.325cm, 24.19481cm}, centering]{geometry}

\usepackage{amsmath}
\usepackage{amssymb}
\usepackage{verbatim}
\usepackage[linktocpage]{hyperref}
\usepackage{mathdots}
\usepackage{comment}
\usepackage{graphicx}
\usepackage{color}
\definecolor{NoteColor}{rgb}{1,0,0}

\usepackage{enumerate}

\usepackage{microtype}

\usepackage{lineno}

\theoremstyle{plain}
\newtheorem{thm}{Theorem}[section]
\newtheorem{prop}[thm]{Proposition}
\newtheorem{lem}[thm]{Lemma}
\newtheorem{cor}[thm]{Corollary}
\newtheorem{conjecture}[thm]{Conjecture}

\theoremstyle{definition}
\newtheorem{df}[thm]{Definition}

\newtheorem{question}[thm]{Question}

\theoremstyle{remark}
\newtheorem{rem}[thm]{Remark}

\newcommand{\K}{K\"{a}hler }

\newcommand{\Oo}{{\mathcal O}}

\def    \HH      {{\mathbb H}}

\def    \C      {{\mathbb C}}
\def    \R      {{\mathbb R}}
\def    \Z      {{\mathbb Z}}
\def    \K      {{\mathbb K}}

\def    \P      {{\mathbb P}}
\def    \CP     {{\mathbb C}{\mathbb P}}
\def    \RP     {{\mathbb R}{\mathbb P}}
\def    \SS     {{\mathbb S}}
\def    \ra     {{\rightarrow}}
\def    \lra     {{\longrightarrow}}

\def    \tr     {\operatorname{tr}}

\def    \r      {\rangle}

\newcommand{\midwd}{\,\,\,\middle\vert\,\,\,}

\copyrightinfo{2021}{Daniele Alessandrini, Colin Davalo and Qiongling Li}

\author{Daniele Alessandrini}
\address{Department of Mathematics, Columbia University, 2990 Broadway, New York, NY, 10027, USA}
\email{daniele.alessandrini@gmail.com}

\author{Colin Davalo}
\address{Mathematisches Institut, Ruprecht-Karls Universität Heidelberg, Im Neuenheimer Feld 205, 69120 Heidelberg, Germany}
\email{cdavalo@mathi.uni-heidelberg.de}

\author{Qiongling Li}
\address{Chern Institute of Mathematics and LPMC, Nankai University, Weijin Road No.94, Tianjin, China}
\email{qiongling.li@nankai.edu.cn}

\title{Projective Structures with (Quasi-)Hitchin Holonomy}
\date{\today}

\begin{document}

\begin{abstract}
In this paper we investigate the properties of the real and complex projective structures associated to Hitchin and quasi-Hitchin representations that were originally constructed using Guichard-Wienhard's theory of domains of discontinuity. We determine the topology of the underlying manifolds and we prove that some of these geometric structures are fibered in a special standard way. In order to prove these results, we give two new ways to construct these geometric structures: we construct them  using gauge theory, flat bundles and Higgs bundles, and we also give a new geometric way to construct them.  
\end{abstract}

\maketitle
\tableofcontents

\setcounter{tocdepth}{2}

\section{Introduction}

\subsection{Higher Teichm\"uller theory}

Higher Teichm\"uller theory is a branch of geometry studying the analogs of Teichm\"uller space for higher rank Lie groups. For an introduction see Wienhard's survey paper \cite{WienhardInvitation}. The starting point is the observation that Teichm\"uller theory is deeply related with the rank one groups $PSL(2,\R) = \mathrm{Isom}^+(\HH^2)$ and $PSL(2,\C) = \mathrm{Isom}^+(\HH^3)$. For a closed orientable surface $S$ of genus $g\geq 2$, Teichm\"uller space can be identified with a connected component of the character variety $X(\pi_1(S),PSL(2,\R))$, whose elements are discrete and faithful representations called Fuchsian representations. Similarly, the quasi-Fuchsian space, i.e. the space of quasi-Fuchsian representations, is an open subset of $X(\pi_1(S),PSL(2,\C))$. 

The theory can be extended by replacing $PSL(2,\R)$ and $PSL(2,\C)$ by higher rank Lie groups, such as for example $PSL(m,\R)$ and $PSL(m,\C)$, groups of rank $m-1$. It is possible to find copies of Teichm\"uller space inside $X(\pi_1(S),PSL(m,\R))$ and $X(\pi_1(S),PSL(m,\C))$, this is called the Fuchsian locus. We can then deform the Fuchsian locus to obtain interesting open subsets of these character varieties. 

Hitchin \cite{HitchinLieGroups} proved that the Fuchsian locus is contained in connected components of $X(\pi_1(S),PSL(m,\R))$ that are homeomorphic to $\R^{(m^2-1)(2g-2)}$. They are now called the Hitchin components, and their elements are the Hitchin representations. There are two such components when $m$ is even and one component when $m$ is odd. Labourie \cite{Labourie} proved that they are Anosov representations and that they share many properties with the Fuchsian representations. Hitchin components are the first example of higher Teichm\"uller spaces (see \cite{WienhardInvitation}). 

Similarly, we can find connected open subsets of $X(\pi_1(S),PSL(m,\C))$ that contain the Fuchsian locus and are the higher rank analog of the quasi-Fuchsian space. We call these subsets the quasi-Hitchin spaces, and their elements the quasi-Hitchin representations. For a definition, see Section \ref{sec:quasi-hitchin}, where we will mainly cover the case of even $m$.

\subsection{Geometric structures}

Teichm\"uller space and the quasi-Fuchsian space were introduced because they serve as deformation spaces of geometric structures on the surface: Teichm\"uller space parametrizes the hyperbolic structures on $S$, and the quasi-Fuchsian space parametrizes the quasi-Fuchsian complex projective structures on $S$. 
The initial motivation for this work comes from the following question:
\begin{question} \label{question:parameter space}
Can we see the Hitchin components and the quasi-Hitchin spaces as deformation spaces of geometric structures on a manifold? 
\end{question}

This question has a satisfying answer in two cases: when $m=3$, the Hitchin component of $PSL(3,\R)$ is the deformation space of convex $\RP^2$-structures on $S$, see Choi-Goldman \cite{CG}; when $m=4$, the Hitchin component of $PSL(4,\R)$ is the deformation space of convex foliated $\RP^3$-structures on the unit tangent bundle of $S$, see Guichard-Wienhard \cite{GW08}. 

When $m>4$, there is no complete answer yet. This is a good moment to remark that the theory is very different for odd or even $m$. For odd $m$, a possible way to approach this question is given in Danciger-Gu\'eritaud-Kassel \cite{DGK} and another one in Stecker-Treib \cite{SteckerTreib}. For the rest of this paper, we will assume that $m = 2n$ is even: this case is easier to understand because of the existence of half-dimensional spaces in $\R^{2n}$.

In this case, a partial answer is given by Guichard-Wienhard \cite{GW}. There, they introduced the theory of domains of discontinuity for A\-no\-sov representations, and using this tool they proved that every Hitchin representation in $PSL(2n,\R)$ is the holonomy of an $\RP^{2n-1}$-structure on a closed manifold $M_\R$ that does not depend on the representation. In this way, they show that a Hitchin component is homeomorphic to a connected component of the deformation space of $\RP^{2n-1}$-structures on $M_\R$. This result has a twofold interest: on the one hand, they can describe the topology of a connected component of the aforementioned deformation space; on the other hand, they can see the Hitchin component as a deformation space of geometric structures, making it even more similar to Teichm\"uller space.   

Similarly, they proved that every quasi-Hitchin representation in $PSL(2n,\C)$ is the holonomy of a $\CP^{2n-1}$-structure on a closed   manifold $M_\C$ that does not depend on the representation. This time the deformation space of $\CP^{2n-1}$-structures on $M_\C$ contains an open subset that is a covering of the quasi-Hitchin space. In this way they compute the dimension of an open subset of the aforementioned deformation space. Moreover, there is hope in the future to describe the topology of the quasi-Hitchin space: for example the topology of the quasi-Fuchsian space is well understood. Again, they can see the quasi-Hitchin space as a parameter space of geometric structures, making it even more similar to the quasi-Fuchsian space. 

The theorem in \cite{GW} gets close to answer Question \ref{question:parameter space}, but it still leaves two questions open: we will discuss them as Quentions \ref{question:topology} and \ref{question:characterization}.

\begin{question} \label{question:topology}
What are the manifolds $M_\R$ and $M_\C$? 
\end{question}
The method used in \cite{GW} does not give any information about the topology of the manifolds $M_\R$ and $M_\C$, because they are constructed in an indirect way. We only know it explicitly in the case of  $M_\R$ for $n=2$: it has $2$ connected components, one of whom is the unit tangent bundle of $S$ \cite{GW08}. Guichard-Wienhard \cite{GW} announced that $M_\R$ is always a fiber bundle over $S$, with fiber a quotient of a Stiefel manifold, 
this result will appear in \cite{GWprep}. The topology of the manifold $M_\C$ was studied by Dumas-Sanders \cite{DumasSanders}. They computed the homology ring of $M_\C$ (and of other related manifolds) and this result inspired the following conjecture:
\begin{conjecture}[Dumas-Sanders {\cite[Conj. 1.1]{DumasSanders}}] \label{conj:dumassanders}
The manifold $M_\C$ (and other manifolds obtained by the Kapovich-Leeb-Porti construction \cite{KLP13} from a quasi-Hitchin representation in a complex semisimple group) is homeomorphic to a continuous fiber bundle over $S$.
\end{conjecture}

In this paper, we will determine the topology of $M_\R$ and $M_\C$, see Theorem \ref{thm:topology of M}.
This will answer Question \ref{question:topology} and prove Conjecture \ref{conj:dumassanders} in the special case of $M_\C$, adding more evidence for the general conjecture.  

\begin{question} \label{question:characterization}
How can we characterize the geometric structures corresponding to Hitchin or quasi-Hitchin representations?
\end{question}

Here the point is that the map from the Hitchin component or quasi-Hitchin space to the deformation space of geometric structures is usually not onto. The question asks for a characterization of the image. In the cases where the problem is well understood, geometric characterizations exist: for $PSL(2,\R)$ the image is all hyperbolic structures on $S$, for $PSL(2,\C)$ the image is all the quasi-Fuchsian $\CP^1$-structures on $S$, for $PSL(3,\R)$ the image is all the convex $\RP^2$-structures on $S$ \cite{CG}, for $PSL(4,\R)$, the image is all the convex foliated $\RP^3$-structures on $T^1 S$ \cite{GW08}. 

In his thesis, Baraglia \cite{BaragliaThesis} showed that Hitchin representations in $PSp(4,\R)$ correspond to convex-foliated contact projective structures on the unite tangent bundle of the surface, and he proved that for such structures the image by the developing map of each fiber is a projective line. Collier, Tholozan and Toulisse \cite{CTT} reinterpreted this in light of the isomorphism $PSp(4,\R) \simeq SO_0(2,3)$, that transform contact projective structures into photon structures, generalized Baraglia's result to all maximal representations, showing that also in this case the image by the developing map of each fiber is a special circle, and proved the converse result: all photon structures that are fibered in the same special way have holonomy in the space of maximal representations. They also have a similar result for maximal representations in $SO_0(2,n)$, $n\geq 3$.

Here we don't have a complete answer to Question \ref{question:characterization} for the Hitchin components, but we prove that if a Hitchin or quasi-Hitchin representation is close enough to the Fuchsian locus, it is the holonomy of a fibered projective structure, i.e. a projective structure on a bundle over the surface such that the image by the developing map of each fiber is a standard embedding of a Stiefel manifold, see Theorem \ref{thm:dev image intro}.

A preliminary version of the results obtained here was announced in the survey paper by Alessandrini \cite{AleSIGMA}.

\subsection{Our results}

We will first discuss our answer to Question \ref{question:topology}. We fix an even $m=2n$, and we denote by $\K$ the field $\R$ or $\C$. We denote by $S$ a closed oriented surface of genus $g$. Using their theory of domains of discontinuity, Guichard-Wienhard \cite{GW} constructed manifolds $M_\K$, such that every Hitchin representation in $PSL(2n,\R)$ is the holonomy of an $\R\P^{2n-1}$-structure on $M_\R$, and every quasi-Hitchin representation is the holonomy of a $\C\P^{2n-1}$-structure on $M_\C$, see Section \ref{sec:dod} for more details. 

In order to describe the topology of $M_\K$, we consider the space $F_\K$, defined in the following way. The space $F_\R = T^1 \R\P^{n-1}$ is the unit tangent bundle of the projective space. It carries a natural $SO(2)$-action, the geodesic flow for the round metric on $\R\P^{n-1}$. For the space $F_\C$, we consider the sphere $\SS^{2n-1}$. It carries a $U(1)$-action called the Hopf action, whose quotient is $\C\P^{n-1}$. The space $F_\C$ is the quotient $F_\C = (T^1 \SS^{2n-1})/U(1)$. The space $F_\C$ also carries an $SO(2)$-action: the geodesic flow of the round metric on $\SS^{2n-1}$ commutes with the Hopf action, hence it descends to an $SO(2)$-action on $F_\C$ that we will still call the geodesic flow. For more details about $F_\K$ and the relevant actions, see Section \ref{sec:stiefel}.

\begin{thm} \label{thm:topology of M}
For $n\geq 3$ and $\K=\R$, or $n\geq 2$ and $\K=\C$, the manifold $M_\K$ is homeomorphic to a fiber bundle over the surface $S$ with fiber $F_\K$, and structure group $SO(2)$ acting on $F_\K$ via the geodesic flow. This bundle has Euler class $2g-2$, and this invariant completely characterizes the bundle. 
\end{thm}

In the text, this theorem is proved in Theorems \ref{thm:comparison} and \ref{thm:Fibration} and Proposition \ref{prop:diagonal-Hitchin}. This result answers Question \ref{question:topology} and proves Conjecture \ref{conj:dumassanders} in the special case of $M_\C$. We add more details on the topology of $M_\K$ in Section \ref{DescriptionOfFibers}, for example, we prove that $M_\R$ is a circle bundle over $S \times \mathrm{Gr}^+(2,\R^n)$, where $\mathrm{Gr}^+(2,\R^n)$ is the oriented Grassmannian of $2$-planes in $\R^n$ (Proposition \ref{prop:circle bundle over a product}). A similar description for $M_\C$ is also given there, but it is more complicated to state.

In order to prove the theorem, we give a new construction of projective structures whose holonomy is close enough to a Fuchsian representation. In this new construction, we start with an explicit manifold that is defined as a suitable fiber bundle $p:U_\K\to S$ over the surface $S$. We prove the following:

\begin{thm}
Let $\K=\R$ or $\C$, and $n \geq 2$. Let $\rho$ be a Fuchsian representation of $\pi_1(S)$ into $PSL(2n,\K)$. There exists a projective structure on $U_\K$ whose holonomy is $\rho\circ p_*$.

\medskip

The projective structure constructed on $U_\K$ is projectively isomorphic to the one constructed by Guichard-Wienhard on $M_\K$. In particular, the fibre bundle $U_\K$ is diffeomorphic to $M_\K$. 
\end{thm}

In the text, this is proved in Theorems \ref{Independence}, \ref{thm:comparison}, \ref{thm:Fibration}. This theorem is the key to the proof of Theorem \ref{thm:topology of M}. These projective structures are constructed by taking a transverse section of the flat bundle associated to $\rho$. We use Higgs bundles for Fuchsian representations to prove the transversality of the section. If we consider a representation $\rho$ that is close enough to a Fuchsian representation, the section will remain transverse, and hence one can obtain also a geometric structure associated to $\rho$. 

In this paper, we cannot give a complete answer to Question \ref{question:characterization}, but we prove a result in that direction. We prove that, for the structure constructed with our method, the developing image of the fibers of the bundle agrees with a standard model. The following theorem is proved in Proposition \ref{prop:nearby repr} and Theorem \ref{thm:developing of the fiber}.

\begin{thm} \label{thm:dev image intro}
There exists a neighborhood $\mathcal{O}_\K$ of Fuchsian representations into $PSL(2n,\K)$ inside the space of Hitchin representations if $\K=\R$ and quasi-Hitchin representations if $\K=\C$ and there exists a map $\mathcal{G}_\K$ that associates to any representation $\rho \in \mathcal{O}_\K$ a projective structure on $U_\K$ whose holonomy is $\rho\circ p_*$.

\medskip

The image of the fibers $p^{-1}(x)$ of $U_\K$ at a point $x\in S$ by the developing map associated to the projective structure is a standard projective embedding of the Stiefel manifold (see Definition \ref{df:standard Stiefel}).
\end{thm}  

Now consider the case when $\rho$ is a Fuchsian representation in $PSL(2,\R)$, diagonally embedded into $PSL(2n,\K)$. There also exists a fiber bundle $p:W_\K\to S$ over $S$ with a projective  structure whose holonomy is $\rho\circ p_*$. The developing image of the fibers of these fibrations are also standard projective embeddings of the Stiefel manifold. In Proposition \ref{prop:diagonal-Hitchin} we compare these bundles with the bundles $U_\K$, and show that they are often isomorphic.

\begin{prop}
As topological bundles over $S$, we have the following isomorphisms.
\begin{enumerate}
\item For $n\geq 3$, $U_\R\cong W_\R$;
\item For $n\geq 2$, $U_\C\cong W_\C$;
\end{enumerate}
\end{prop}

In the main body of the paper, we will work not only with projective structures but also with spherical structures associated with Hitchin and quasi-Hitchin representations, see Section \ref{sec:geometric structures} for more details. For the spherical structures, we have results very similar to the results about the projective structures mentioned here.

\subsection{Our methods}

In order to prove our results, we give an independent construction of the projective structures associated with a Hitchin or quasi-Hitchin representation. Our method works for representations that are close enough to the Fuchsian locus. 

We are using gauge theory: given a representation, we cosider the associated flat vector bundle on the surface $E_\K$ and its projectivization $\P(E_\K)$. We construct the manifold $U_\K$ as a subbundle of $\P(E_\K)$. The pull back of $\P(E_\K)$ to $U_\K$ is a bundle on $U_\K$ with a canonical section. This section induces an equivariant map from the universal covering of $U_\K$ to the fiber $\K\P^{2n-1}$. In order for this equivariant map to be the developing map of a projective structure, we have to check a transversality condition of the canonical section with reference to the flat connection. This construction is called the graph of a projective structure on $U_\K$, see Section \ref{sec:geometric structures}.

When the representation is Fuchsian, we use Higgs bundles to describe the bundle $E_\K$ and its flat connection. These objects are very concrete in this case, and we can use this description to define the subbundle $U_\K$ and to prove that the canonical section is transverse, see Section \ref{section:ConstructionOfProjectiveStructures}.

The use of gauge theory let us extend the construction of these geometric structures on a fiber bundle over $S$ for representations that are close enough to a Fuchsian representation, using compactness and a transversality argument.
However it is not clear whether our method can allow to define geometric structures for every Hitchin representation, as the transversality of the canonical section is difficult to prove in general.

In Sections \ref{dods} and \ref{sec:Geometricdescription}, we show that the the geometric structures we constructed are isomorphic to the ones constructed by Guichard and Wienhard. In particular the topology of the manifold on which these structures are defined are the same, and hence we can answer Question \ref{question:topology}. In order to do this, we study the image of the developing map, and we prove that this image is contained in the domain of discontinuity defined by Guichard and Wienhard. A topological argument will then prove that the projective structures are isomorphic.

In Section \ref{sec:Geometricdescription} we give a more geometric point of view on our construction of geometric structures for the Fuchsian locus. We give a way to construct the same projective structures that does not use Higgs bundles. This construction depends on the choice of a parameter $\lambda$, and we choose $\lambda$ in a way that is related to the harmonic metric associated to the Higgs bundle, and will yield a projective structure with the same developing map as the one constructed with the Higgs bundle method.

Finally, in Section \ref{sec:ComparisonDiagonal}, we compare the topology of $U_\K$ with the topology of the manifold $W_\K$ associated to a diagonal representation. We use the second Stiefel-Whitney class to show that $U_\K$ and $W_\K$, are isomorphic as $SO(n)\times SO(2)$-principal bundles in many cases.

\subsection{Organization of the paper}

In Section \ref{sec:quasi-hitchin} we present the definition of Hitchin and quasi-Hitchin representations, and their Anosov properties.

\medskip

In Section \ref{sec:geometric structures} we present the definition of geometries  in the sense of Klein and Thurston and their associated structures on manifolds. We then recall two ways of constructing geometric structures on manifolds: the construction of a domain of discontinuity, and the construction of the graph of a geometric structure. We also prove a version of Thurston's holonomy principle.

\medskip

In Section \ref{sec:dod} we recall some results of Guichard and Wienhard \cite{GW}: the construction of a domain of discontinuity for Anosov representations, its explicit description when the representation is Fuchsian, and finally the correspondence they obtain between some geometric structures on some manifolds $M_\K$ for $\K=\R$ or $\C$ and Hitchin or quasi-Hitchin representations.

\medskip

In Section \ref{sec:manifolds} we give a description of Stiefel manifolds, which will arise as the fibers of the manifolds $U_\K$ on which we will construct some geometric structures.

\medskip

In Section \ref{section:ConstructionOfProjectiveStructures} we construct geometric structures on some fiber bundles $U_\K$ over the surface $S$. For this we recall the description of Higgs bundles associated to a Fuchsian representation. We construct a section of the flat bundles $\mathbb{P}(E_\C)$ and $\mathbb{S}(E_\C)$ and show that this section is transverse. Finally we extend this construction on a neighborhood of the space of Fuchsian representations.

\medskip

In Section \ref{dods} we compare our construction with the construction of Guichard and Wienhard, and prove using Lemma \ref{NonIntersect} that the two constructions induce the same geometric structure, on the same manifold.

\medskip

In Section \ref{sec:Geometricdescription} we give a more geometric way to construct these geometric structures associated to Fuchsian representations. In this section we prove Lemma \ref{NonIntersect}. 

\medskip

In Section \ref{sec:ComparisonDiagonal} we describe the geometric structures associated to diagonal representations constructed by Guichard--Wienhard, and we compare the topology of the underlying manifolds associated to Fuchsian representations and diagonal representations.

\medskip

In Section \ref{DescriptionOfFibers} we give more details on the topology of the manifolds $M_\K$.

\subsection{Acknowledgements}

The authors thank David Dumas, Olivier Guichard, Andrew Sanders and Anna Wienhard for enlightening conversations. C.D thanks in particular Olivier Guichard for the supervision of his master thesis, during which he worked on this problem. The authors acknowledge support from U.S. National Science Foundation grants DMS 1107452, 1107263, 1107367 ``RNMS: GEometric structures And Representation varieties'' (the GEAR Network). D.A. was supported by the DFG grant AL 1978/1-1 within the Priority Programme SPP 2026 ``Geometry at Infinity''. C.D acknowledges support from the ENS, Paris and was partially funded through the DFG Emmy Noether project 427903332 of B. Pozzetti. Q.L. acknowledges support from Nankai Zhide Foundation.

\section{(Quasi-)Hitchin representations}   \label{sec:quasi-hitchin}

In this paper, $S$ will denote a closed orientable surface of genus $g \geq 2$. We will denote by $X(\pi_1(S),G)$ the \emph{character variety} of a reductive Lie group $G$, i.e. the parameter space $\mathrm{Hom}^*(\pi_1(S),G)/G$ of conjugacy classes of reductive representations of $\pi_1(S)$ into $G$.

Here we will consider representations $\rho:\pi_1(S) \ra PSL(2n,\K)$, with $\K = \R$ or $\C$. We are mainly interested in Anosov representations, a special type of representations introduced by Labourie \cite{Labourie}, by Burger-Iozzi-Labourie-Wienhard \cite{BILW}, and, in full generality, by Guichard-Wienhard \cite{GW}. 

The notion of Anosov representation depends on the choice of a parabolic subgroup. Let $V_n \subset \K^{2n}$ be a fixed half-dimensional vector subspace and $[V_n] \subset \K\P^{2n-1}$ be the corresponding projective subspace, we consider the parabolic subgroup
\[  Q_n = \{ g\in PSL(2n,\K) \mid g([V_n]) \subset [V_n] \}\,.  \]
We will not recall here the complete definition of $Q_n$-Anosov representations, see \cite{GW,GGKW,KLP13,KLP14}. We will only recall that every $Q_n$-Anosov representation $\rho$ is discrete and faithful, and that it admits a continuous
 $\rho$-equivariant map:
\begin{equation}   \label{eq:curve xi}
 \xi:\partial_\infty \pi_1(S)\ \ra\ \mathrm{Gr}_n(\K^{2n})\,.  
\end{equation}

We will denote by $\mathrm{Anosov}_{2n}(S,\K) \subset X(\pi_1(S), PSL(2n,\K))$ the open subset of the character variety  
consisting of $Q_n$-Anosov representations. The spaces $\mathrm{Anosov}_{2n}(S,\K)$ are important for their geometric and dynamical properties. 

For example, when $n=1$, the spaces $\mathrm{Anosov}_2(S,\K)$ are fundamental in Teichm\"uller theory and hyperbolic geometry. The space $\mathrm{Anosov}_2(S,\R) = \mathrm{Fuch}_{2}(S)$ is the space of discrete and faithful representations in $PSL(2,\R)$, also called \emph{Fuchsian representations}. This space is the union of two connected components of $X(\pi_1(S),PSL(2,\R))$, we call them the \emph{Fuchsian components}, and each of them is a copy of \emph{Teichm\"uller space}. Similarly, the space $\mathrm{Anosov}_2(S,\C) = \mathrm{QFuch}_{2}(S)$ is the \emph{quasi-Fuchsian space}, the space of the representations in $PSL(2,\C)$ whose action on $\CP^1$ is topologically conjugate to the action of a Fuchsian representation. These representations are called the \emph{quasi-Fuchsian representations}.

When $n>1$, we can identify some connected components of $\mathrm{Anosov}_{2n}(S,\K)$ that are the higher rank analogs of the Fuchsian components and the quasi-Fuchsian space. Let $\iota:PSL(2,\R) \ra PSL(2n,\R)$ be the irreducible representation, unique up to conjugation. For every Fuchsian representation $\rho_0:\pi_1(S)\ra PSL(2,\R)$, the composition $\iota\circ\rho_0:\pi_1(S)\ra PSL(2n,\R)$ is called a \emph{Fuchsian representation} in $PSL(2n,\R)$. Their space, denoted by $\mathrm{Fuch}_{2n}(S)$ and called the \emph{Fuchsian locus}, is the union of two copies of Teichm\"uller space inside $X(\pi_1(S),PSL(2n,\R))$. The two connected components of $X(\pi_1(S), PSL(2n,\R))$ containing $\mathrm{Fuch}_{2n}(S)$ are called the \emph{Hitchin components}, and their union is denoted by $\mathrm{Hit}_{2n}(S)$.
Its elements are called the \emph{Hitchin representations}. The Hitchin components were introduced by Hitchin \cite{HitchinLieGroups}, who proved that each component is homeomorphic to $\R^{(4n^2-1)(2g-2)}$. Labourie \cite{Labourie} proved that all Hitchin representations are $Q_n$-Anosov, hence
\[ \mathrm{Hit}_{2n}(S) \subset \mathrm{Anosov}_{2n}(S,\R)\,. \] 
This result implies that every Hitchin representation is discrete and faithful, in particular every component of $\mathrm{Hit}_{2n}(S)$ is a \emph{higher Teichm\"uller space}, see Wienhard's survey paper \cite{WienhardInvitation}.

When $\mathrm{Hit}_{2n}(S)$ is mapped to $X(\pi_1(S), PSL(2n,\C))$ using the inclusion of groups $PSL(2n,\R) \ra PSL(2n,\C)$, the image is connected, because the two components are conjugate by an element of  $PSL(2n,\C)$. We denote by $\mathrm{QHit}_{2n}(S)$, the \emph{quasi-Hitchin space}, the connected component
 of the space of irreducible elements in $\mathrm{Anosov}_{2n}(S,\C)$ that contains the image of $\mathrm{Hit}_{2n}(S)$. This is an open subset of $X(\pi_1(S),PSL(2n,\C))$ whose elements will be called $Q_n$-\emph{quasi-Hitchin representations} in $PSL(2n,\C)$, or shortly just \emph{quasi-Hitchin representations}. The idea behind this name is that they are the higher rank analogs of quasi-Fuchsian representations. The quasi-Hitchin space was studied (with a different name and in greater generality) by Dumas-Sanders \cite{DumasSanders}. A similar space, with the same name is also studied in a paper in preparation by Alessandrini-Maloni-Wienhard \cite{AMW}.

\section{Geometric structures}   \label{sec:geometric structures}

One of the motivations for this work is to understand parameter spaces of geometric structures on a fixed manifold. A geometric structure is a way to model a manifold locally on a homogeneous geometry, in the sense of Klein and Thurston (see \cite{ThurstonBook, AleSIGMA}). 

\subsection{Geometries}
More precisely, a \emph{geometry} is a pair $(X,G)$, where $X$ is a manifold and $G$ is a Lie group acting on $X$ transitively and effectively. We now discuss the four geometries we will consider in this paper.

Let $\K=\R$ or $\C$. Given a vector space $\K^{m}$, we will denote the associated projective space by $\P(\K^m) = \K\P^{m-1}$ and the associated sphere by   
\[ \SS(\K^{m}) = (\K^{m}\setminus \{0\})/\sim\,,\]
where the equivalence relation is
\[v \sim w \Leftrightarrow \exists r\in \R_{>0}: v = r w\,. \]
Here we prefer to see the spheres as quotients instead of seeing them as subsets of $\K^{m}$. In this way it is clear that the group $SL(m,\K)$ acts naturally on $\SS(\K^{m})$. When we don't need to consider the group actions, we will also identify the spheres with subsets of $\K^{m}$ consisting of vectors of unit norm with reference to a standard scalar or Hermitian product. 

We denote the natural $SL(m,\K)$-equivariant projections by
\begin{center}
\vspace{0.15cm}
\begin{tabular}{ccccc}
$\K^{m}\setminus \{0\}$ & $\longrightarrow$ & $\SS(\K^{m})$ & $\longrightarrow$ & $\K\P^{m-1}$\\
           $v$           & $\longmapsto$ &    $[v]_\SS$   & $\longmapsto$ & $[v]_\P\,.$  
\end{tabular}
\vspace{0.15cm}
\end{center}

Note that for $\K=\R$, $\SS(\R^{m}) = \SS^{m-1}$, 
and the projection $\SS(\R^{m}) \ra \R\P^{m-1}$ is $2:1$, the universal covering of $\R\P^{m-1}$.
When $\K=\C$, $\SS(\C^{m}) = \SS^{2m-1}$, and the projection $\SS(\C^{m}) \ra \C\P^{m-1}$ is a circle bundle, the Hopf fibration.

We will consider the following geometries, for $\K = \R$ or $\C$:
\begin{itemize}
\item the (real or complex) \emph{projective geometry}:\ \ \ $\K\P^{2n-1} := (\K\P^{2n-1}, PSL(2n,\K))$.
\item the (real or complex) \emph{spherical geometry}:\ \ \ $\SS(\K^{2n}) := (\SS(\K^{2n}), SL(2n,\K))$.
\end{itemize}

The two spherical geometries are closely related to the real projective geometry. The $SL(2n,\R)$-equivariant projection $\SS(\R^{2n}) \ra \RP^{2n-1}$ shows that the geometry $\SS(\R^{2n})$ is locally isomorphic to $\RP^{2n-1}$. The group $SL(2n,\C)$ can be identified with a subgroup of $SL(4n,\R)$ and the identity map of $\SS(\C^{2n}) \ra \SS(\R^{4n})$ is equivariant: we can hence say that $\SS(\C^{2n})$ is a subgeometry of $\SS(\R^{4n})$.

\subsection{Structures on manifolds} \label{subsec:structures on manifolds}
Now we will describe how to use geometries to define geometric structures on manifolds. Let $M$ be a fixed closed manifold having the same dimension as $X$. An $(X,G)$-\emph{structure} on $M$ is an atlas on $M$ with charts taking values in $X$, whose transition functions are locally restrictions of elements of $G$. 

An $(X,G)$-structure on $M$ determines a \emph{developing pair} $(h,D)$, where $h:\pi_1(M)\ra G$ is a representation called the \emph{holonomy representation} and $D: \widetilde{M} \rightarrow X$ is an $h$-equivariant local diffeomorphism called the \emph{developing map}. The developing pair is well defined up to a $G$-action, see \cite{AleSIGMA} for details. Vice versa, the developing map determines the $(X,G)$-structure. 

We denote by $\mathcal{D}_{(X,G)}(M)$ the \emph{deformation space} of all $(X,G)$-structures on $M$, up to the equivalence relation of being isomorphic with an isomorphism that is isotopic to the identity. The holonomy representation induces a map called the \emph{holonomy map}
\[\mathrm{hol}: \mathcal{D}_{(X,G)}(M)\ \ra\ \mathrm{Hom}(\pi_1(M),G)/G\,, \]
that associates a geometric structure to the conjugacy class of its holonomy representation.
This map is always open and it has discrete fibers (see Goldman \cite{GoldmanSurvey2}). Sometimes it fails to be a local homeomorphism, see the counterexamples in Kapovich \cite{Kapovich} and Baues \cite{Baues}, where branching appears. We will now show that, in a special case that includes all the geometric structures associated to Hitchin and quasi-Hitchin representations,  the map $\mathrm{hol}$ is a local homeomorphism. 

\begin{lem}   \label{lemma:Holonomy principle}
Let $\K=\R$ or $\C$, and let $G = SL(m,\K)$ or $PSL(m,\K)$. Consider the open subset $\mathrm{Hom}^{a.i.}(\pi_1(M),G)$ consisting of the absolutely irreducible representations (i.e. the representations that are irreducible over $\C$), and the corresponding open subset of the deformation space
\[\mathcal{D}_{(X,G)}^{\,a.i.}(M) = \mathrm{hol}^{-1}(\mathrm{Hom}^{a.i.}(\pi_1(M),G)/G)\,. \]
Then the restricted map
\[\mathrm{hol}|_{a.i.}:  \mathcal{D}_{(X,G)}^{\,a.i.}(M)\ \ra\ \mathrm{Hom}^{a.i.}(\pi_1(M),G)/G \]
is a local homeomorphism. 
\end{lem} 
\begin{proof}
A more general statement is given, without proof, by Goldman \cite{GoldmanSurvey2}. The proof makes use of Thurston's holonomy principle, see Baues \cite[Thm. 3.15]{Baues} for the statement and Canary-Epstein-Green \cite[Sec. 1.7]{CEG} for the proof of the holonomy principle. Notice that an absolutely irreducible representation has a closed orbit under conjugation. Moreover it has trivial centralizer by Schur's Lemma, in particular it is of principal orbit type under conjugation. As discussed in Kapovich \cite[Sec. 5]{Kapovich}, this implies the lemma.   
\end{proof}

The holonomy map gives a link between the deformation spaces of geometric structures and representations. In particular, the lemma says that we can often describe the local geometry of the deformation space using representations. In some special cases, it is possible to construct some interesting sections of the holonomy map over special open subsets of its image. 

In this paper, we will consider two different methods for producing $(X,G)$-structures on manifolds. The first is the method of domains of discontinuity: consider a discrete subgroup $\Gamma < G$, and assume that there exists an open subset $\Omega \subset X$ such that $\Gamma$ preserves $\Omega$ and acts on $\Omega$ freely and properly discontinuously. In this case, $\Omega$ is called a \emph{domain of discontinuity} for $\Gamma$. Then $M := \Omega / \Gamma$ is a manifold, and it inherits an $(X,G)$ structure from $\Omega$. In this case, the holonomy representation is a surjective homomorphism $h:\pi_1(M) \ra \Gamma$, and the developing map is a covering $D:\widetilde{M} \ra \Omega$. Notice that, in this way, the manifold $M$ is constructed as a quotient, and its topology is not always easy to determine. See Section \ref{sec:dod} for an account of how this method was applied to Hitchin and quasi-Hitchin representations. The second method is described in the next subsection.

\subsection{Graph of geometric structures}  \label{subsec:graph of structure}
We will now describe the graph of a geometric structure. This is a way to encode the developing pair $(h,D)$ of a geometric structure using the theory of bundles, see the survey paper \cite{AleSIGMA} for more details. 

Given a geometry $(X,G)$ and a manifold $M$, there is a natural bijection between conjugacy classes of representations of $\pi_1(M)$ in $G$ and gauge equivalence classes of flat structures on bundles over $M$ with fiber $X$ and structure group $G$. Under this bijection, an equivariant map $\widetilde{M} \ra X$ corresponds to a section of the bundle. Assuming that $\dim(X) = \dim(M)$, a section of a flat bundle on $M$ is said to be \emph{transverse} if the corresponding equivariant map is a local diffeomorphism. From this, we can see that giving an $(X,G)$-structure on $M$ is equivalent to giving a flat bundle over $M$ with fiber $X$ and structure group $G$ together with a transverse section. 

We now describe how this point of view becomes more explicit in the case of the spherical and projective geometries we are discussing in this paper. For the spherical geometries, the holonomy takes vaues in $G=SL(2n,\K)$. For the projective geometries, the holonomy takes vaues in $G=PSL(2n,\K)$, but the special projective structures we will consider have the extra property that their holonomy lifts to $SL(2n,\K)$. In both cases, a bundle with structure group $G$ is described by a vector bundle on $M$, and the flat structure is described by a flat connection on the vector bundle.   

Now let's fix a representation $\rho:\pi_1(M) \ra SL(2n,\K)$, and denote by $(E, \nabla)$ the corresponding vector bundle over $M$ and its flat connection. For every point $x\in M$, we denote by $E_{x}$ the fiber of $E$ over $x$. We denote by $\SS(E)$ the associated spherical bundle and by $\P(E)$ the associated projective bundle. A section $s$ of $E$ determines sections $[s]_\SS$ of $\SS(E)$ and $[s]_\P$ of $\P(E)$. These sections determine $\rho$-equivariant maps
\begin{align*}
D_s:           & \widetilde{M}\ \rightarrow\  \K^{2n}\,, \\
D_{[s]_\SS}:   & \widetilde{M}\ \rightarrow\  \SS(\K^{2n})\,, \\
D_{[s]_\P}:    & \widetilde{M}\ \rightarrow\  \K\P^{2n-1}\,.
\end{align*}

\begin{lem}   \label{lemma:transversality}
The  map $D_s$ is an immersion if and only if for every $x\in M$ and every $v \in T_x M$, $\nabla_{v} s \neq 0$. 

The map $D_{[s]_\SS}$ is an immersion if and only if for every $x\in M$ and $v \in T_x M$, the derivative $\nabla_{v} s$ is non-vanishing when projected to the quotient $E_x / \left<s_x\right>_\R$, where $\left<s_x\right>_\R$ is the real span of $s_x$.

The map $D_{[s]_\P}$ is an immersion if and only if for every $x\in M$ and $v \in T_x M$, the derivative $\nabla_{v} s$ is non-vanishing when projected to the quotient $E_x / \left<s_x\right>$, , where $\left<s_x\right>$ is the span of $s_x$.
\end{lem}
\begin{proof}
Let $x\in M$ and $v\in T_x M$. Let $\gamma$ be a differentiable curve in $M$ with initial point $x=\gamma(0)$ and initial tangent vector $v=\gamma'(0)$. Let $P(\gamma)^0_t: E_{\gamma(t)}\rightarrow E_{\gamma(0)}$ be the parallel transport operator of  $\nabla$. We let $\widetilde{M}$ be the universal covering space of $M$ corresponding to $x\in M$: a point of $\widetilde{M}$ is a relative homotopy class of paths $\gamma: [0,1]\rightarrow M$ such that $\gamma(0)=x$ and the projection to $M$ is the evaluation map $\gamma\mapsto \gamma(1)$.

The map $D_s$ is given by $D_s([\gamma]) = P(\gamma)^0_1 s_{\gamma(1)} \in E_x \simeq \K^{2n}.$  
The  path $\gamma$ lifts to a path $\widetilde{\gamma}: [0,1]\ni t\mapsto [\gamma|_{[0,t]}]\in \widetilde{M}$ based at the constant path $[x]\in \widetilde{M}$. The relation between the parallel transport operator and the connection is 
\[ \nabla_v s = \lim _{{h\to 0}}{\frac  {P (\gamma )_{h}^{0}s_{{\gamma (h)}}-s_{{\gamma (0)}}}{h}}=\left.{\frac  {d}{dt}}P (\gamma )_{t}^{0}s_{{\gamma (t)}}\right|_{{t=0}} = \left.{\frac  {d}{dt}}D_s(\tilde{\gamma}(t))\right|_{{t=0}}=(D_s)_*(\tilde v),  \]
where $\tilde v$ is the tangent vector of the lift path $\widetilde{\gamma}$ in $\widetilde{M}$ at the constant path $[x]$. Therefore, we have the developing map $D_s: \widetilde{M}\rightarrow \K^{2n}$ is an immersion at $x$ if and only if $\nabla_vs\neq 0$ for any tangent vector $v$ of $M$ at $x$.

The statements about $D_{[s]_\SS}$ and $D_{[s]_\P}$ now follow because these maps are the composition of $D_s$ with the projection from $\K^{2n}$ to $\SS(\K^{2n})$ and $\K\P^{2n-1}$ respectively. 
\end{proof}

\section{Domains of discontinuity}   \label{sec:dod}

\subsection{Construction of the domains of discontinuity.} \label{subsec:general_dod}
Let $\K = \R$ or $\C$. Given $\rho \in \mathrm{Anosov}_{2n}(S,\K)$, Guichard and Wienhard \cite{GW} used the $\rho$-equivariant curve $\xi$ from (\ref{eq:curve xi}) to construct a domain $\Omega_\K \subset \P(\K^{2n}) $ in the following way. They define the subset 
$$K_\K = \bigcup_{t \in \partial_\infty \pi_1(S)} [\xi(t)]_\P \subset \P(\K^{2n}),  $$ 
where $[\xi(t)]_\P$ is the projective $n-1$-plane corresponding to the linear $n$-plane $\xi(t)$. The set $K_\K$ is a $\rho$-invariant compact subset of $\P(\K^{2n})$, so its complement 
$$\Omega_\K = \P(\K^{2n}) \setminus K_\K$$ 
is a $\rho$-invariant open subset. Guichard-Wienhard \cite{GW} prove that the action of the image group $\rho(\pi_1(S))$ on $\Omega_\K$ is properly discontinuous, free and co-compact, so the quotient is a closed manifold:
\[M_\K =\rho(\pi_1(S))\setminus \Omega_\K\,.\]

\begin{lem}
If $\rho \in \mathrm{Hit}_{2n}(S)$, or $\rho \in \mathrm{QHit}_{2n}(S)$, then $\rho$ admits a lift  
\[\bar{\rho}:\pi_1(S)\ \rightarrow\ SL(2n,\K)\,.\]
\end{lem}
\begin{proof}
The irreducible representation $\iota:PSL(2,\R) \rightarrow PSL(2n,\K)$ admits a lift to a representation $\bar{\iota}:SL(2,\R) \rightarrow SL(2n,\K)$. Every Fuchsian representation $\rho_0:\pi_1(S) \ra PSL(2,\R)$ admits a lift $\bar{\rho}_0:\pi_1(S) \ra SL(2,\R)$, hence every Fuchsian representation $\iota\circ\rho_0$ in $PSL(2n,\R)$ admits a lift $\bar{\iota}\circ \bar{\rho}_0$. 

Now we can conclude that every representation in $\mathrm{Hit}_{2n}(S)$ or $\mathrm{QHit}_{2n}(S)$ admits a lift, because the existence of a lift is a topological property, hence it only depends on the connected component of the character variety where the representation lies. 
\end{proof}

Assume now that the $Q_n$-Anosov representation $\rho$ admits a lift $\bar{\rho}:\pi_1(S) \rightarrow SL(2n,\K)$. In this case, it is possible to consider the action of $SL(2n,\K)$ and $\bar{\rho}(\pi_1(S))$ on the spheres $\SS(\K^{2n})$.
We denote by $SK_\K$ and $S\Omega_\K$ the inverse image of $K_\K$ and $\Omega_\K$ by the projections. The action of $\bar{\rho}(\pi_1(S))$ on $S\Omega_\K$ is again properly discontinuous, free and co-compact, so the quotient is a closed manifold:
\[SM_\K =\bar{\rho}(\pi_1(S)) \setminus S\Omega_\K\,.\] 
For $\K = \R$, $SM_\R$ is a $2:1$ covering of $M_\R$, while for $\K=\C$, $SM_\C$ is a circle bundle over $M_\C$.

\subsection{The Fuchsian case.} \label{subsec:fuchsian_dod}

The construction presented in Section \ref{subsec:general_dod} becomes more explicit in the special case when $\rho = \iota \circ \rho_0$ is a Fuchsian representation in $PSL(2n,\R)$. There is an explicit model of the representations $\iota$ and $\bar{\iota}$ given by an action on homogeneous polynomials.

The standard action of $SL(2,\R)$ on $\K^2$ induces an action on the algebra $\K[X,Y]$ of polynomials in two variables, given by
\begin{equation}  \label{eq:action on polynomials}
g\cdot P(X,Y) = P(aX+bY,cX+dY)\,,
\end{equation}
where $P \in \mathbb{K}^{(2n-1)}[X,Y]$ and $g\in SL(2,\mathbb{R})$ with
\[g^{-1}=\begin{pmatrix} a & b \\ c & d \end{pmatrix}\,.\]
This formula defines an $SL(2,\R)$-action on $\K[X,Y]$ that leaves invariant the subspaces of homogeneous polynomials of every degree. We consider the vector subspace
\[ V_\K := \mathbb{K}^{(2n-1)}[X,Y] \simeq   \K^{2n}  \]
of homogeneous polynomials of odd degree $2n-1$ in two variables.

The action (\ref{eq:action on polynomials}) gives a representation $\epsilon:SL(2,\R) \ra SL(2n,\K)$. Then
we can define a representation $\bar{\iota}:SL(2,\R)\to SL(2n,\K)$ taking $g$ to $\epsilon((g^{-1})^t)$ and an induced projective representation $\iota:PSL(2,\R)\to PSL(2n,\K)$. These representations are irreducible, hence, by uniqueness up to conjugation, they are a model of the irreducible representation. There is a natural $\iota$-equivariant curve $v:\RP^1 \ra \mathrm{Gr}_n(\mathbb{K}^{(2n-1)}[X,Y])$ defined in the following way: for $[a:b]\in \RP^1$, we will define $v([a:b])$ as the linear subspace  
\[v([a:b]) := \{\ P \in \mathbb{K}^{(2n-1)}[X,Y]\ \mid\      (a X + b Y)^n | P  \ \}\,.\]
This is the $n$-dimensional linear subspace consisting of all polynomials containing the factor $(a X + b Y)$ with multiplicity at least $n$. 

Let $g\in SL(2,\R)$ be a hyperbolic element, and let $[a,b]$ and $[c,d]$ be its repelling and attracting fixed points in $\mathbb{RP}^1$. Then $\iota(g)$ is also diagonalisable, with distinct real eigenvalues. A basis of eigenvectors is given by the polynomials $(a X + b Y)^k (c X + d Y)^{2n-1-k}$, for $0\leq k< 2n$. When $g$ acts on $\mathrm{Gr}_n(\mathbb{K}^{(2n-1)}[X,Y])$, it has unique attracting and repelling fixed points at $v([a:b])$ and $v([c:d])$.

Let $\rho=\iota \circ\rho_0$ be a Fuchsian representation in $PSL(2n,\R)$, where $\rho_0$ is a Fuchsian representation. The representation $\rho_0$ is Anosov, hence it admits a continous equivariant map $\xi_0:\partial_\infty \pi_1(S) \to \mathbb{RP}^1$. The map $\xi_0$ induces the $\rho$-equivariant map $\xi = v \circ \xi_0$. 

When we apply the construction of Section \ref{subsec:general_dod} to a Fuchsian representation in $PSL(2n,\R)$, we see that the image of the boundary map $\xi$ associated to $\rho$ does not depend on $\rho_0$, hence also the induced sets $K_\K$, $SK_\K$, $\Omega_\K$ and $S\Omega_\K$ don't depend on $\rho_0$. Explicitly, $K_\K$ and $SK_\K$ consist of all the projective or spherical classes of polynomials having a root with multiplicity at least $n$, and $\Omega_\K$ and $S\Omega_\K$ consist of all the projective or spherical classes of polynomials all of whose roots have multiplicity smaller then $n$.

\subsection{Domains of discontinuity and geometric structures.}

The $\rho$-equivariant curve $\xi:\partial_\infty\pi_1(S) \ra \mathrm{Gr}_n(\K^{2n})$ varies continuously with the representation $\rho$ (see \cite{GW}), so do the sets $K_\K, SK_\K, \Omega_\K$ and $S\Omega_\K$. Since the group action on $\Omega_\K$ and $S\Omega_\K$ is co-compact, this implies that the topology of $M_\K$ and $SM_\K$ is constant on every connected component of $\mathrm{Anosov}_{2n}(S,\K)$ (see \cite[Thm. 9.12]{GW}). In the special case of Hitchin or quasi-Hitchin representations, these results give rise to the following statement:

\begin{thm}[Guichard-Wienhard \cite{GW}] \    \label{thm:Guichard-Wienhard-dod} 
\begin{enumerate}
\item For $n \geq 2$, there exist closed manifolds $M_\R$ and $SM_\R$ and maps
\[ \mathrm{Hit}_{2n}(S)\ \ra\ \mathcal{D}_{\RP^{2n-1}}(M_\R),  \hspace{1.5cm}  \mathrm{Hit}_{2n}(S)\ \ra\ \mathcal{D}_{\SS(\R^{2n})}(SM_\R)\,, \]
that are homeomorphisms between $\mathrm{Hit}_{2n}(S)$ and a connected component of their target spaces.

\item For $n \geq 1$, there exists closed manifolds $M_\C$ and $SM_\C$ and connected components $\mathcal{U}$ and $S\mathcal{U}$ of respectively $\mathcal{D}_{\CP^{2n-1}}(M_\C)$ and $\mathcal{D}_{\SS(\C^{2n})}(SM_\C)$ and maps 

\[ \mathcal{U},S\mathcal{U}\ \ra\ \mathrm{QHit}_{2n}(S)\]
that are covering maps whose deck transformations coincide with the action of some elements of the mapping class group of $M_\C$ or $SM_\C$, respectively. 
\end{enumerate}
In the four cases, the closed manifold (denote it by $M$) comes equipped with a homomorphism $p:\pi_1(M) \ra \pi_1(S)$, and the geometric structure associated by the map to a representation $\rho$ has holonomy $\rho \circ p$ or $\bar{\rho} \circ p$, where $\bar{\rho}$ is a lift of $\rho$ to $SL(2n,\R)$ or $SL(2n,\C)$.
\end{thm}

Part (1) of the previous theorem is in \cite[Thm 11.3]{GW}. Part (2) is not explicitly written there, since they don't mention quasi-Hitchin representations, but the proof of Part (2) is the same as the proof of \cite[Thm 11.6]{GW}. The reason why the statement is more complicated in Part (2) is because we don't know the topology of $\mathrm{QHit}_{2n}(S)$, hence this space might be non simply-connected.

In order to prove the theorem, Guichard and Wienhard used their theorem about domains of discontinuity, described in Section \ref{subsec:general_dod}. This method can be applied to all $Q_n$-Anosov representations (see Section \ref{sec:quasi-hitchin} for details), and all Hitchin and quasi-Hitchin representations are $Q_n$-Anosov. They also need to use some version of the Thurston's holonomy principle. They don't clarify which version they are using, since at the time the counterexamples to this principle were not well known, and it was believed to be true in general. In this paper, we wrote Lemma \ref{lemma:Holonomy principle} in order to clarify this part of the proof. This is also the reason why, in the definition of $\mathrm{QHit}_{2n}(S)$, we required these representations to be irreducible. 

When applied to the case of $\mathrm{Hit}_{2n}(S)$, Theorem \ref{thm:Guichard-Wienhard-dod} describes the topology of a connected component of $\mathcal{D}_{\RP^{2n-1}}(M_\R)$ and $\mathcal{D}_{\SS(\R^{2n})}(SM_\R)$. When applied to the case of $\mathrm{QHit}_{2n}(S)$, it computes the dimension of an open subset of $\mathcal{D}_{\CP^{2n-1}}(M_\C)$ and $\mathcal{D}_{\SS(\C^{2n})}(SM_\C)$. Moreover, there is hope in the future to describe the topology of $\mathrm{QHit}_{2n}(S)$: for example the topology of $\mathrm{QFuch}_{2}(S)$ is well understood. 

This theorem is of central importance in the study of the Hitchin components and the spaces of quasi-Hitchin representations because it allows us to see these spaces as parameter spaces of geometric structures on a closed manifold. This strengthens the analogy between these spaces and Teichm\"uller space as well as the space of quasi-Fuchsian representations. 

This method of producing geometric structures does not give any information about the topology of the constructed manifolds, and determining it directly is very hard. Even the question whether the manifold is a fiber bundle over the surface is hard to understand, except in the case of $\mathrm{Hit}_{4}(S)$, where it was proved by Guichard-Wienhard \cite{GW08}. 

Another point that is not well understood is how to characterize the image of the maps to the deformation spaces of geometric structures: the Hitchin components and the quasi-Hitchin space parametrize some subset of geometric structures, but which ones? A complete answer exists only for $\mathrm{Hit}_{4}(S)$, see Guichard-Wienhard \cite{GW08}.

In this work, we address both questions. Our main result is the description of the topology of the closed manifolds from Theorem \ref{thm:Guichard-Wienhard-dod}, see Theorem \ref{thm:topology of M} for details. We also have some partial results about the geometric structures in the image of the maps: the geometric structures are constructed on a fiber bundle over $S$, and each fiber has a particular geometry, see Theorem \ref{thm:dev image intro} for details.

\section{Stiefel manifolds}   \label{sec:manifolds}

\subsection{Projective embeddings}   \label{sec:stiefel}
Consider $(\C^n,\left<\cdot,\cdot\right>)$, where $\left<\cdot,\cdot\right>$ is the standard Hermitian product:
\[ \left<v,w\right> = \overline{v}^T \cdot w \,.\]
We denote the associated norm by $\left|v\right| = \sqrt{\left<v,v\right>}$. We consider $\R^n$ as a subset of $\C^n$, and the restriction of $\left<\cdot,\cdot\right>$ to $\R^n$ gives the standard scalar product. 

Let $\K=\R$ or $\C$. We will consider the following subsets of $\K^n \times \K^n = \K^{2n}$:
\begin{eqnarray*}
C_\C & = &  \left\{(v,w) \in \C^n\times\C^n \midwd   \mathrm{Re}(\left<v,w\right>)=0, \left|v\right|^2=\left|w\right|^2   \right\},\\
C_\R & = &  \left\{(v,w) \in \R^n\times\R^n \midwd  \left<v,w\right>=0, \left|v\right|^2=\left|w\right|^2 \right\}.
\end{eqnarray*}
These subsets are \emph{cones}, i.e they have the property that if $(v,w) \in C_\K$, then for every $\lambda\in\K$, we have $(\lambda v, \lambda w) \in C_\K$. As a consequence, these sets induce subsets of $\K\P^{2n-1}$ and $\SS(\K^{2n})$:
\begin{eqnarray*}
SF_\K & = & \left\{ [(v,w)]_\SS \in \SS(\K^{2n}) \midwd (v,w) \in C_\K \setminus \{0\} \right\}, \\
 F_\K & = & \left\{ [(v,w)]_\P \in \K\P^{2n-1}    \midwd (v,w) \in C_\K \setminus \{0\} \right\}.
\end{eqnarray*}

We see from the equations that $SF_\R$ parametrizes ordered pairs of orthogonal unit vectors in $\R^n$, hence it is a \emph{Stiefel manifold}. Similarly, we can see $\C^n$ as an $\R$-vector space of dimension $2n$, with the scalar product $\mathrm{Re}(\left<\cdot,\cdot\right>)$. In this way, from the equations, we see that $SF_\C$ parametrizes ordered pairs of orthogonal unit vectors in $\R^{2n}$ and is thus again a Stiefel manifold. Their topology is given by:
\[SF_\R = T^1 \SS^{n-1}, \hspace{1cm}  SF_\C = SF_\R^{2n} = T^1 \SS^{2n-1} \,,\]
where the map $SF_\R \ra \SS^{n-1}$ and $SF_\C \ra \SS^{2n-1}$ are the projections $[(v,w)]_\SS \mapsto [v]_\SS$.

By analogy, we will call the manifold $F_\K$ a \emph{projective Stiefel manifold}. The topology of $F_\R$ is given by
\[ F_\R = SF_\R/\Z_2  = T^1 \R\P^{n-1}\,,   \]
where the map $F_\R \ra \R\P^{n-1}$ is the projection $[(v,w)]_\P \ra [v]_\P$.

To describe $F_\C$ we need to consider again the complex structure of $\C^n$, that we see as an action of $U(1)$ by multiplication. 
This group acts on $\C^n$, on $\SS^{2n-1}$ (the \emph{Hopf action}), on $C_\C$ and on $SF_\C$. The space $F_\C$ is the quotient of $SF_\C$ by this action:
\[F_\C  = SF_\C/U(1) = (T^1 \SS^{2n-1})/U(1)\,,\]
where the action of $U(1)$ on $T^1 \SS^{2n-1}$ is the differential of the Hopf action on $\SS^{2n-1}$. The space $F_\C$ is closely related with the tangent bundle of $\C\P^{n-1}$, see Lemma \ref{lem:topology of FCn}.

We will give a more detailed description on $SF_\K, F_\K$ in Section \ref{DescriptionOfFibers}.

\begin{df} \label{df:standard Stiefel}
We will say that a submanifold $M\subset \SS(\K^{2n})$ is a \emph{standard spherical embedding} of the Stiefel manifold if there is a transformation in $SL(2n,\K)$ that sends $M$ to $SF_\K$.

Similarly, we will say that a submanifold $M\subset \K\P^{2n-1}$ is a \emph{standard projective embedding} of the projective Stiefel manifold if there is a transformation in $PSL(2n,\K)$ that sends $M$ to $F_\K$.
\end{df}

The submanifolds $SF_\K$ and $F_\K$ are real algebraic smooth subvarieties of $\SS(\K^{2n})$ and $\K\P^{2n-1}$ respectively, and they are both projective and affine real algebraic varieties, for $\K = \R$ or $\C$. Note even for $\K = \C$, they are not complex varieties, because the equations involve the complex conjugation.

\subsection{Group actions}

The natural action of $O(n)$ on $(\R^n,\left<\cdot,\cdot\right>)$ induces an action of $O(n)$ on $\R^n\times\R^n$, $C_\R$, $SF_\R$ and $F_\R$, which is transitive on $SF_\R$ and $F_\R$. Using the transitive action we can see the spaces $SF_\R$ and $F_\R$ as homogeneous spaces:
\[SF_\R = O(n)/O(n-2), \hspace{1cm} F_\R = O(n)/(O(n-2)\times \left<\pm{Id}\right>)\,.\]

Similarly, we have a natural action of $O(2n)$ on $(\C^n,\mathrm{Re}(\left<\cdot,\cdot\right>))$ by $\R$-linear maps which induces an action of $O(2n)$ on $\C^n\times\C^n$, $C_\C$, $SF_\C$, which is transitive on $SF_\C$, so we have:
\[SF_\C = O(2n)/O(2n-2)\,.\]

It is important to remark that the $O(2n)$-action on $SF_\C$ does not induce an $O(2n)$-action on $F_\C$. Recall that we see the complex structure of $\C^n$ as an action of $U(1)$, which acts as a   subgroup of $O(2n)$. 
The only elements of $O(2n)$ that have a well defined action on $F_\C$ are the ones that commute with $U(1)$. This is precisely the subgroup $U(n)$, the elements of $O(2n)$ whose action on $\C^n$ is $\C$-linear. The group $U(n)$ acts on $\C^n\times\C^n$, $C_\C$, $SF_\C$ and $F_\C$, but this action is never transitive because the function $f(v,w)=\frac{\mathrm{Im}(\left<v,w\right>)}{\left|v\right|\left|w\right|}$, well defined on $F_\C$ and $SF_\C$, is $U(n)$-invariant and takes all the values in $[-1,1]$. 

Every homogeneous space $G/H$ carries a left action of $G$ and a right action of $N_G(H)/H$. These two actions commute, so they give an action of $ G \times N_G(H)/H$. For $SF_\K$, we have $N_G(H)/H = O(2)$ and for $F_\R$, $N_G(H)/H = PO(2)$. The action of $O(n)\times O(2)$ on $O(n)/O(n-2)$ can also be described explicitly with matrix multiplication. We see an element $(v,w) \in C_\R$ as a $n\times 2$ matrix. Then the action becomes: 
\begin{equation}\label{O(n)O(2)action} 
(O(n)\times O(2)) \times C_\R \ni ((A,B),(v,w))\ \longmapsto\ A(v,w)B^{-1} \in C_\R\,. 
\end{equation}
It descends to an $O(n)\times O(2)$-action on $SF_\R$ and $F_\R$. The same formula describes the $U(n)\times O(2)$-action on $SF_\C$ and $F_\C$. The fact that the right $O(2)$-action is well defined on $F_\C$ should not be surprising, because the right $O(2)$-action on $SF_\C$ commutes with the $U(1)$-action given by the complex structure, inducing an $O(2)$-action on $F_\C$.

We consider the representation
\begin{equation}  \label{eq:geodesic flow}
\delta: U(1) \ni e^{i\theta} \ \longmapsto \  (\mathrm{Id}, R_\theta) \in G \times O(2)\,, 
\end{equation}
where $G$ is one of the groups $O(n), O(2n)$ or $U(n)$, and $R_\theta \in SO(2)$ is 
\begin{equation} \label{eq:R theta}
R_\theta =  
\left(\begin{matrix}
\cos\theta & -\sin\theta\\
\sin\theta &  \cos\theta
\end{matrix}\right)
\,.
\end{equation}
This $U(1)$-action can easily be described in words: the rotation $R_\theta$ acts on every $[(v,w)]_\SS \in SF_\R$ by rotating the vectors $v,w$ by an angle $\theta$ in the plane $\left<v,w\right>_{\R}$. When $SF_\R$ is seen as unit tangent bundle of a sphere, this action is the geodesic flow, which is periodic, hence it is a circle action.

\subsection{Alternative equations of the embeddings}\label{subsection: AnotherDesription}
We introduce another description of the Stiefel manifolds, which will be used in Section \ref{section:ConstructionOfProjectiveStructures}. Consider the $\C$-antilinear involution $\tau_0$ defined as 
\[\tau_0:\C^{2n} \ni (t_1, t_2, \dots, t_{2n})^T\ \longrightarrow\ (\overline{t}_{2n}, \overline{t}_{2n-1}, \dots, \overline{t}_{1})^T \in \C^{2n}\,.\] 
We will use $\tau_0$ as a real structure on $\C^{2n}$, this means that the real points with reference to $\tau_0$ are
\begin{equation}   \label{eq:real part}
\mathrm{Fix}(\tau_0) = \left\{v \in \C^{2n} \midwd v =  \tau_0(v)\right\}\,.
\end{equation}

Consider the sets 
\[C_\C' =  \left\{t=(t_1, \dots, t_{2n})^T \in \C^{2n} \midwd \textstyle\sum_{j=1}^{n} t_{2j-1} \overline{t}_{2j} = 0 \right\}\,,
\hspace{1cm} C_\R' = C_\C' \cap \mathrm{Fix}(\tau_0)\,. \]
These subsets are cones, hence they induce subsets of $\K\P^{2n-1}$ and $\SS(\K^{2n})$:
\[SF_\C' =\left\{[t]_{\SS}\in \SS(\C^{2n}) \midwd t\in C_\C' \setminus \{0\}\right\}\,, \hspace{1cm} SF_\R' = SF_\C' \cap \mathrm{Fix}(\tau_0)\,. \]
\[F_\C' = \left\{[t]_{\P} \in \C\P^{2n-1}\midwd t\in C_\C' \setminus \{0\}\right\}\,, \hspace{1cm} F_\R' = F_\C' \cap \mathrm{Fix}(\tau_0)\,. \]

In the rest of this subsection we will prove that $SF_\C', SF_\R', F_\C', F_\R'$ are standard spherical or projective embeddings of the Stiefel manifolds. 
Let $k$ be such that $n=2k$ if $n$ is even, $n=2k+1$ if $n$ is odd. We consider the representation 
\begin{equation}   \label{eq:phi}
\phi: U(1) \ni e^{i\theta} \ \longmapsto\ (L_\theta, R_\theta) \in O(n)\times O(2)\,,
\end{equation} 
where $R_\theta$ was defined in (\ref{eq:R theta}) and $L_\theta$ is 
\begin{enumerate}[(i)]
\item $L_\theta = \mathrm{diag}(L_1,L_2,\cdots,L_k),$ if $n=2k$ is even, 
\item $L_\theta = \mathrm{diag}(L_1,L_2,\cdots,L_k,1), $ if $n=2k+1$ is odd, where
\end{enumerate}
\[L_i=\begin{pmatrix}
 \cos((2n+2-4i)\theta) & -\sin((2n+2-4i)\theta) \\
 \sin((2n+2-4i)\theta) &  \cos((2n+2-4i)\theta)\end{pmatrix}\,.
\]

We also consider the representation
\begin{equation} \label{form:U1 action irr}
\phi': U(1) \ni e^{i\theta}\  \longmapsto\ (e^{i(2n-1)\theta},e^{i(2n-3)\theta},\cdots,e^{i(1-2n)\theta}) \in U(2n)\,.
\end{equation}

\begin{rem}   \label{rem:actions on the cones}
The representation $\phi$ leaves invariant the sets $C_\C$ and $C_\R$. By quotient, it also acts on $SF_\C$, $SF_\R$, $F_\C$, and $F_\R$. The representation $\phi$ can be thought of as a modification of the geodesic flow $\delta$ from (\ref{eq:geodesic flow}). Similarly, the representation $\phi'$ leaves invariant the sets $C_\C'$ and $C_\R'$, and acts, by quotient on $SF_\C'$, $SF_\R'$, $F_\C'$, and $F_\R'$. These $U(1)$-actions on $F_\C$, $F_\R$, $F_\C'$, and $F_\R'$ are not effective, since the subgroup $\{\pm 1\}$ acts trivially.
\end{rem}

\begin{lem} \label{lemma:C-linear transformation}
There is a $\C$-linear isomorphism $A: \C^{2n} \ni t \mapsto (v,w) \in \C^n\times \C^n$ with the following properties.
\begin{enumerate}
\item It maps $\mathrm{Fix}(\tau_0)$ to $\R^n \times \R^n$.
\item It maps $C_\K'$, $SF_\K'$, $F_\K'$ to $C_\K$, $SF_\K$, $F_\K$ respectively.
\item It conjugates the representation $\phi'$ to the representation $\phi$.
\end{enumerate}
In particular, $SF_\K'$ and $F_\K'$ are standard spherical or projective embeddings of the Stiefel manifolds.
\end{lem}
\begin{proof} 
Consider the $\C$-linear transformation $A$ defined by
\[ A: \C^{2n} \ni t\ \mapsto\ (v,w) \in \C^n\times \C^n\,,\]
where, for $2k \leq n$, we have
\begin{eqnarray*}
\begin{pmatrix}
v_{2k-1} \\w_{2k-1}\\v_{2k}\\w_{2k}\end{pmatrix} =\frac{1}{2}\begin{pmatrix}
 1 &  1 &  1 &  1\\
-i &  i & -i &  i\\
-i & -i &  i &  i\\
-1 &  1 &  1 & -1
\end{pmatrix}\begin{pmatrix}
t_k\\
t_{k+1}\\
t_{2n-k}\\
t_{2n-k+1}
\end{pmatrix} =A_k\begin{pmatrix}
t_k\\
t_{k+1}\\
t_{2n-k}\\
t_{2n-k+1}
\end{pmatrix}
\end{eqnarray*}
and if $n$ is odd, the last coordinate is
\begin{eqnarray*}
\begin{pmatrix}
v_n\\
w_n\end{pmatrix}=\frac{\sqrt{2}}{2}\begin{pmatrix}1&1\\-i&i\end{pmatrix}\begin{pmatrix}t_{n}\\t_{n+1}\end{pmatrix}=B\begin{pmatrix}t_{n}\\t_{n+1}\end{pmatrix}.\end{eqnarray*}
The coordinate transformation $A$ is block-diagonal after changing the order of the coordinates, with $4 \times 4$ blocks $A_1, A_2,\cdots,A_k$ and possibly one $2\times 2$ block $B$.  
$A$ is invertible because $A_k$'s and $B$ are invertible: 
$$A_k^{-1}=\frac{1}{2}\begin{pmatrix}
 1 &  i &  i & -1\\
 1 & -i &  i &  1\\
 1 &  i & -i &  1\\
 1 & -i & -i & -1
\end{pmatrix}, \quad B=\frac{\sqrt{2}}{2}\begin{pmatrix}1&i\\1&-i\end{pmatrix}.$$

It's also easy to check the property that $A(\tau_0(t)) = \overline{A(t)}$ and it takes $Fix(\tau_0)$ to $\mathbb R^n\times \mathbb R^n$. So Part (1) follows.
An elementary computation shows that the equation $|v|^2 = |w|^2$ is equivalent to the equation $\mathrm{Re}(\sum\limits_{j=1}^{n} t_{2j-1} \overline{t}_{2j}) = 0$, the equation $\mathrm{Re}(\left<v,w\right>) = 0$ is equivalent to the equation $\text{Im}(\sum\limits_{j=1}^{n} t_{2j-1} \overline{t}_{2j}) = 0$, and $\sum\limits_{i=1}^{2n}|t_i|^2=\sum\limits_{i=1}^n(|v_i|^2+|w_i|^2)$. So Part (2) follows.

We can understand the transformation of every block by the following computation:
$$A_i
\left(\begin{matrix}
e^{i(2n+3-4i)\theta} \\
      & e^{i(2n+1-4i)\theta} \\
      &   & e^{-i(2n+1-4i)\theta} \\
      &   &              & e^{-i(2n+3-4i)\theta}
\end{matrix}\right)
A_i^{-1}
$$
$$ = 
\begin{pmatrix}
 \cos((2n+2-4i)\theta) & -\sin((2n+2-4i)\theta) \\
 \sin((2n+2-4i)\theta) &  \cos((2n+2-4i)\theta) \end{pmatrix}
\otimes
\left(\begin{matrix}
\cos\theta & -\sin\theta \\
\sin\theta &  \cos\theta 
\end{matrix}\right)=L_i\otimes R
$$
$$
B
\left(\begin{matrix}
e^{i\theta} & \\
          & e^{-i\theta}
\end{matrix}\right)
B^{-1}
=1\otimes
\left(\begin{matrix}
\cos\theta & -\sin\theta\\
\sin\theta &  \cos\theta
\end{matrix}\right)=1\otimes R
$$

Comparing the definition of the representation $\phi$ of $U(1)$ into $U(2n)$, we obtain $A\phi'(e^{i\theta})A^{-1}=\phi(e^{i\theta})$, for $e^{i\theta}\in U(1)$. So Part (3) follows.
\end{proof}

\section{Construction of Projective Structures}\label{section:ConstructionOfProjectiveStructures}
In this section, we will construct projective and spherical structures starting from a Hitchin or quasi-Hitchin representation. More precisely, our construction will work when the representation lies in a small neighborhood of the Fuchsian locus $\mathrm{Fuch}_{2n}(S)$. We will initially work with a representation $\rho \in \mathrm{Fuch}_{2n}(S)$, and we will subsequently extend our results to representations that are close enough to $\rho$. 

The manifolds that will support the geometric structures are constructed in the following explicit way. Let $\K$ be $\R$ or $\C$, and
let $P$ be a principal $U(1)$-bundle on $S$ with Euler class $g-1$. We recall that the Euler class of a $U(1)$-bundle completely determines the isomorphism class of the bundle. We consider the representation $\phi'$ given by (\ref{form:U1 action irr}) and recall that by Remark \ref{rem:actions on the cones}, this gives $U(1)$-actions on $SF_\K'$ and $F_\K'$. We define the associated bundles 
\begin{align*}
 SU_\K & :=P\times_{\phi'} SF_\K'\,,\\
 U_\K & :=P\times_{\phi'} F_\K'\,,
\end{align*}
whose projections we denote by 
\[ p : U_\K\ \ra\ S\,,  \hspace{1cm}  \overline{p} : SU_\K\ \ra\ S\,.  \]   

\begin{rem}  \label{rem:different action}
It is important to remark that the $U(1)$-action on $F_\K'$ is not effective, the subgroup $\{\pm 1\}$ acts trivially. We have an effective action on $F_\K'$ of the quotient group $U(1)/\{\pm 1\}$, that is still isomorphic to $U(1)$. We will always consider $U_\K$ as a $U(1)$-bundle with reference to this new structure group, and we notice that, with this new group, the bundle $U_\K$ has Euler class $2g-2$.  
\end{rem}

The definition of the manifolds $SU_\K$ and $U_\K$ is explicit. In Sections \ref{sec:ComparisonDiagonal} and \ref{DescriptionOfFibers} we will describe their topology more precisely.

\subsection{General strategy}     \label{subsec:strategy}
Let $\rho:\pi_1(S) \rightarrow PSL(2n,\R)$ be a Fuchsian representation in $PSL(2n,\R)$. It can be lifted to a representation $\overline{\rho}: \pi_1(S) \rightarrow SL(2n,\R)$. 
The representation $\rho$ induces representations 
\[\hat{\rho} = \rho \circ p_*:\pi_1(U_\K)\ \rightarrow\ PSL(2n,\R)\,,  \hspace{1cm} \hat{\overline{\rho}} = \overline{\rho} \circ \overline{p}_*:\pi_1(SU_\K)\ \rightarrow\ SL(2n,\R)\,,  \]
that is trivial on the fibers of the bundles $U_\K$ and $SU_\K$.

We will then construct a $\K\P^{2n-1}$-structure on $U_\K$ with holonomy $\hat{\rho}$, and a $\SS(\K^{2n})$-structure on $SU_\K$ with holonomy $\hat{\overline{\rho}}$, by using the technique of graph of geometric structures described in Section \ref{subsec:graph of structure}. 

We will consider the flat (real or complex) vector bundle $(E_\K,\nabla_\K)$ over $S$ associated with the representation $\overline{\rho}$. This bundle has an associated spherical bundle $\SS(E_\K)$ and an associated projective bundle $\P(E_\K)$, with bundle maps
\[ E_\K\ \ra\ \SS(E_\K)\ \ra\ \P(E_\K)\,. \]
By introducing a scalar or Hermitian product on $E_\K$, we can also embed the spherical bundle into $E$:
\[\SS(E)\ \ra\ E\,.\]
Now a $\K\P^{2n-1}$-structure on $U_\K$ with holonomy $\hat{\rho}$ is given by a transverse section of the pull-back bundle $p^*\P(E_\K)$, and a $\SS(\K^{2n})$-structure on $SU_\K$ with holonomy $\hat{\overline{\rho}}$ is given by a transverse section of $\overline{p}^*\SS(E_\K)$. We divide our construction of these transverse sections in two steps.
 
The first step (in Section \ref{subsec:the bundles}) is to realize $U_\K$ as a subbundle of $\P(E_\K)$ and $SU_\K$ as a subbundle of $\SS(E_\K)$. Then we can define the section $s:U_\K\rightarrow p^*\P(E_\K)$ as the tautological section that associates to every point $v$ of $U_\K$ the same point $v$ seen as a point of $\P(E_\K)$, and similarly for $SU_\K$. This gives us the sections. 

The second step (in Section \ref{subsec:transversality}) is to verify the transversality condition of the section $s$.

\subsection{Higgs bundles}
For both steps, we will use the Higgs bundle description of $(E_\K,\nabla_\K)$. This is inspired by Baraglia's Thesis \cite{BaragliaThesis}, see also the survey paper \cite{AleSIGMA} and our previous paper \cite{AdS}. Let's briefly recall the basic definitions of Higgs bundles. We need a complex structure $\Sigma$ on $S$, and we denote by $K$ the canonical bundle of $\Sigma$, i.e. the holomorphic cotangent bundle. 
\begin{df}
An $SL(2n,\R)$-Higgs bundle over $\Sigma$ is a tuple $(E_\C,\omega,Q,\phi)$ where $E_\C$ is a holomorphic vector bundle of rank $2n$ over $\Sigma$ satisfying $\det E_\C=\mathcal O$, $\omega \in H^0(\Sigma,\det E_\C)$ is a holomorphic volume form, $Q:E_\C \ra E_\C^*$ is a holomorphic symmetric $\C$-bilinear form of volume one for $\omega$, and $\phi \in H^0(\Sigma, \mathrm{End}(E_\C)\otimes K)$ is $Q$-symmetric and satisfies $\tr \phi=0$.
\end{df}
An $SL(2n,\R)$-Higgs bundle $(E_\C,\omega,Q,\phi)$ is called \emph{stable} if every proper $\phi$-invariant holomorphic subbundle has negative degree. The work of Hitchin \cite{Hitchin87, HitchinLieGroups}, Simpson \cite{Simpson88} and Bradlow-Garcia-Prada-Gothen \cite{BGPG} assures that for a stable $SL(2n,\R)$-Higgs bundle over $\Sigma$, there exists a Hermitian metric $H$ satisfying the Hitchin equation
\[F_{\nabla_H}+[\phi,\phi^{*H}]=0,\]
where $\nabla_H$ is the Chern connection of $H$, $F_{\nabla_H}$ is the curvature of $\nabla_H$ and $\phi^{*H}$ is the Hermitian adjoint of $\phi$. Equivalently, the connection 
\[\nabla_\C=\nabla_H+\phi+\phi^{*H}\] 
is flat. Moreover, the metric $H$, seen as a smooth $\C$-anti-linear isomorphism between $E_\C$ and $E_\C^*$, commutes with $Q$, in the sense that $Q \circ H = H \circ Q$. This defines a $\C$-anti-linear involution on $E_\C$ given by 
\begin{equation} \label{eq:tau}
\tau=H\circ Q:E_\C\rightarrow E_\C\,,
\end{equation}
which is a real structure over $E_\C$. The real part of $E_\C$ is the smooth real vector bundle
\[ E_\R = \mathrm{Fix}(\tau) = \left\{v \in E_\C \midwd \tau(v) = v \right\}\,. \]
The connection $\nabla_\C$ preserves $E_\R$, and we define $\nabla_\R$ as the restriction of $\nabla_\C$ to $E_\R$.  

In terms of structure groups, the pair $(\omega,Q)$ induces a holomorphic $SO(2n,\C)$-structure on $E_\C$, and $H$ induces a smooth $SU(2n)$-structure on $E_\C$. Together, these two structures induces a smooth $SL(2n,\R)$-structure on $E_\C$ and on $E_\R$. The pairs $(E_\C,\nabla_\C)$ and $(E_\R,\nabla_\R)$ are flat (resp. complex and real) vector bundles with the same holonomy in $SL(2n,\R)$. 

In this work, we want to use Higgs bundles to describe the flat $(E_\K,\nabla_\K)$ bundles associated with the Fuchsian representation $\rho \in \mathrm{Fuch}_{2n}(S)$. For this, we denote by $\Sigma$ the Riemann surface homeomorphic to $S$ that corresponds to $\rho$.
We consider the $SL(2n,\R)$-Higgs bundle $(E_\C,\omega,Q,\phi)$ over $\Sigma$, given by
\begin{align}
E_\C=K^{\frac{2n-1}{2}}\oplus K^{\frac{2n-3}{2}}\oplus\cdots\oplus K^{\frac{1-2n}{2}}\,,  \label{eq:bundle EC}
\hspace{1cm} & 
\omega = 1 \in \det E = \mathcal{O}\,,
\\
Q=\begin{pmatrix}                  \label{eq:bil Q}
&&&&1\\
&&&1&\\
&&\iddots&&\\
1&&&&
\end{pmatrix}\,, 
\hspace{1cm} & 
\phi=\begin{pmatrix}0&&&&&\\r_1&0&&&&\\&r_2&0&&&\\&&\ddots&\ddots&&\\&&&r_{2n-1}&0\end{pmatrix}\,,
\end{align} 
where $r_i=\sqrt{\frac{i(n-i)}{2}}$ for $1\leq i\leq 2n-1$. For this Higgs bundle, the monodromy of the flat connection $\nabla_\C$ is precisely the representation $\rho$ corresponding to the Riemann surface $\Sigma$. 

The Hermitian metric $H$ is purely determined by the unique hyperbolic metric in the conformal class of $\Sigma$. Let $h^2|dz|^2=\frac{1}{2y^2}|dz|^2$ be the Hermitian hyperbolic metric on $\Sigma$, that is, it satisfies $\triangle\log h=h^2$, where $\triangle=4\partial_z\partial_{\bar z}$. Then one can check the Hermitian metric solving Hitchin equation is 
\begin{equation}  \label{eq:metric}
H=\text{diag}(h^{1-2n},h^{3-2n},\cdots,h^{2n-1})\,. 
\end{equation}
The Hermitian adjoint of $\phi$ is 
\begin{equation*}
\phi^{*H}=\begin{pmatrix}0&r_1h^2&&&\\&0&r_2h^2&&\\&&\ddots&\ddots&\\&&&0&r_{2n-1}h^2\\&&&&0\end{pmatrix}.
\end{equation*}

We can write $\tau$ explicitly: for a local section $\lambda\cdot dz^{\frac{2n+1-2k}{2}}$ of $E_\C$, we have:
\begin{align*}
\tau\left(\lambda\cdot dz^{\frac{2n+1-2k}{2}}\right)&=H\circ Q\left(\lambda\cdot dz^{\frac{2n+1-2k}{2}}\right)=H\left(\lambda\cdot dz^{\frac{2k-2n-1}{2}}\right)\\
&=\bar \lambda\cdot h^{2k-2n-1}dz^{\frac{2k-2n-1}{2}}\,.
\end{align*}
So with respect to the local frame $\left\{dz^{\frac{2n-1}{2}}, dz^{\frac{2n-3}{2}},\cdots dz^{\frac{1-2n}{2}}\right\}$ of $E$, we have:
\begin{equation}
\tau\begin{pmatrix}t_1\\t_2\\ \vdots\\t_{2n}\end{pmatrix} = \begin{pmatrix}h^{2n-1}\overline{t}_{2n}\\h^{2n-3}\overline{t}_{2n-1}\\ \vdots \\h^{1-2n}\overline{t}_1\end{pmatrix}\,.
\end{equation} 
The real part $E_\R$ is given by
\begin{equation}
E_\R=\left\{\left(h^{\frac{2n-1}{2}}t_1,h^{\frac{2n-3}{2}}t_2, \dots,h^{\frac{1-2n}{2}}t_{2n})\right)^T \midwd  t_i=\bar t_{2n+1-i} \ \text{ for }\  1\leq i\leq 2n \right\}\,.
\end{equation}

\subsection{Embedding of \texorpdfstring{$SU_\K$}{SUK} and \texorpdfstring{$U_\K$}{UK}}  \label{subsec:the bundles}

We now want to embed the bundles $SU_\K$ and $U_\K$ into $\SS(E_\K)$ and $\P(E_\K)$, respectively, as subbundles. To do this, we will use the description of $E_\K$ given by the theory of Higgs bundles.

Suppose $P'$ be the unitary frame bundle of $K^{\frac{1}{2}}$ which is a principal $U(1)$-bundle of Euler class $g-1$. Suppose $z$ is a local complex coordinate on $\Sigma$ and $h^2|dz|^2$ is the Hermitian hyperbolic metric on $\Sigma.$ So $h^{\frac{1}{2}}dz^{\frac{1}{2}}$ is a local section of $P'$, denoted by $\sigma$. Then any local section of $P'$ is of the form $\sigma\cdot e^{i\theta}.$

We define $\sigma_i := \sigma^{\otimes\frac{2n+1-2i}{2}}$, a unitary frame of $K^{\otimes\frac{2n+1-2i}{2}}$.
Let $P_1$ be the fiber bundle over $\Sigma$ consists of unitary frames of the form $(\sigma_1,\sigma_2, \cdots, \sigma_{2n})\cdot e^{i\theta}=(\sigma_1\cdot e^{i(2n-1)\theta},\sigma_2\cdot e^{i(2n-3)\theta}, \cdots, \sigma_{2n}\cdot e^{i(1-2n)\theta})$. It is clear that there is a natural transitive and effective right action of $U(1)$.

The bundles $P'$ and $P_1$ are isomorphic as principal $U(1)$-bundles via the map $\eta: P' \ra P_1$ as follows:
\[\eta(\sigma\cdot e^{i\theta}) = (\sigma_1,\sigma_2,\cdots,\sigma_{2n})\cdot e^{i\theta}. \]

\begin{lem}   \label{lem:Euler class P1}
$P_1$ is a principal $U(1)$-bundle over $S$ of Euler class $g-1$. 
\end{lem}
Note that the Euler number completely determines a principal $U(1)$-bundle on a closed orientable surface $S$. In particular, this guarantees that $P_1$ is isomorphic to the $U(1)$-bundle $P$, introduced at the beginning of the section.
Using the fact that $U(1)$ acts on $\C^{2n}$ through the representation $\phi'$, we have the associated complex vector bundle  $P_1\times_{\phi'}\C^{2n}$. 

\begin{lem} \label{lem:isomorphism}
The map
\begin{eqnarray}
\Psi: P_1\times_{\phi'}\C^{2n}&\longrightarrow& E_\C\\
\left((\sigma_1,\cdots, \sigma_{2n})\cdot e^{i\theta}, \phi'(e^{-i\theta})\cdot t\right)&\longmapsto& \sum_{i=1}^{2n}t_i\sigma_i,\quad \text{for $t=(t_1,\cdots, t_{2n})^T$}
\end{eqnarray}
is an isomorphism of $U(1)$-bundles. Moreover, $\Psi$ restricts to an isomorphism of $U(1)$-bundles from $P_1\times_{\phi'}\mathrm{Fix}(\tau_0)$ to $E_\R$.
\end{lem}
\begin{proof}
The first statement is clear. For the second, note that $\tau(\sigma_i)=\sigma_{2n+1-i}$. Suppose $t_i=\bar t_{2n+1-i}$ for $1\leq i\leq 2n$, then 
\begin{equation*}
\tau\left(\sum_{i=1}^{2n}t_i\sigma_i\right)=\sum_{i=1}^{2n}\bar t_i\tau(\sigma_i)=\sum_{i=1}^{2n}t_{2n+1-i}\sigma_{2n+1-i}=\sum_{i=1}^{2n}t_i\sigma_i\,. \qedhere
\end{equation*}
\end{proof}

The previous two lemmas give us an explicit way to embed $SU_\K$ into $\SS(E_\K)$ and $U_\K$ into $\P(E_\K)$ as subbundles. Indeed, by Lemma \ref{lem:Euler class P1}, $P_1$ is isomorphic to the principal bundle $P$ introduced at the beginning of the section. Hence $SU_\K = P_1\times_{\phi'} SF_\K'$ and $U_\K = P_1\times_{\phi'} F_\K'$. Since $SF_\K'$ is a subset of  $\SS(\C^{2n})$ or $\SS(\mathrm{Fix}(\tau_0))$ and $F_\K'$ is a subset of  $\P(\C^{2n})$ or $\P(\mathrm{Fix}(\tau_0))$, Lemma \ref{lem:isomorphism} gives embeddings 
\[SU_\K \ra \SS(E_\K)\,, \hspace{1cm} U_\K \ra \P(E_\K)\,. \]
 
To make this more explicit, we can write the equations of the subbundles $SU_\K$ and $U_\K$. We denote by $C(E_\K)$ the bundle
\[C(E_\K):=P_1\times_{\tau_{2n}} C_\K'\,.\]
In terms of the local frames $\left\{dz^{\frac{2n-1}{2}}, dz^{\frac{2n-3}{2}},\cdots dz^{\frac{1-2n}{2}}\right\}$, it has the following form:
\begin{equation}\label{DefinitionM}
C(E_\K):=\left\{\left( h^{\frac{2n-1}{2}}t_1, h^{\frac{2n-3}{2}}t_2, \dots, h^{\frac{1-2n}{2}}t_{2n}\right)^T \in E_\K \midwd \sum\limits_{i=1}^n t_{2i-1}\bar t_{2i}=0 \right\}\,.
\end{equation} 
The subbundles $SU_\K$ and $U_\K$ are the projection of $C(E_\K)$ to $\SS(E_\K)$ and $\P(E_\K)$.

\subsection{Transversality of the tautological section} \label{subsec:transversality}
We now consider the pullback bundles $\overline{p}^*\SS(E_\K)\rightarrow SU_\K$ and $p^*\P(E_\K)\rightarrow U_\K$, whose flat structure is described by the flat connections $\overline{p}^*\nabla_\K$ and $p^*\nabla_\K$. 
Note that $SU_{\K}, U_\K$ are subsets of $\SS(E_\K), \P(E_\K)$ respectively. We denote by 
\[ [s]_\SS: SU_{\K} \ra  \overline{p}^*\SS(E_\K)\,, \hspace{1cm} [s]_\P: U_{\K} \ra  p^*\P(E_\K)\,, \]
the tautological sections: for $x\in SU_\K \subset \SS(E_\K)$, $[s]_\SS(x)$ is the same point seen as a point of the pull-back bundle $\overline{p}^*\SS(E_\K)$. Similar definition for $[s]_\P$.

\begin{thm}\label{Independence}
The tautological sections $s_{\SS}$ and $s_{\P}$ are transverse.
\end{thm}
\begin{proof}
We consider the bundle 
\[p: C(E_\K)\ \ra\ S\,.\]
This bundle is a subset of $E_\K$, hence the pull-back bundle $p^*E_\K$ has a tautological section 
\[s:C(E_\K)\ \ra\ p^*E_\K\,.\]

Locally, for a coordinate neighborhood $(V, z)$ of $\Sigma$, $(z, t)\in V\times C_\K'$ parametrizes the point $\sum\limits_{i=1}^{2n}t_i\cdot \sigma_i\in C(E_\K)$, where $\sigma_i=h^{\frac{2n+1-2i}{2}}\cdot  dz^{\frac{2n+1-2i}{2}}$ as defined before.

 Correspondingly, we have the tautological section $s: C(E_\K)\rightarrow p^*E_\K$ as follows: at point $(z,t)\in C(E_\K)$, 
$$s(z, t):=\sum\limits_{i=1}^{2n}t_i\cdot p^*\sigma_i.$$
Since $s(z, \lambda\cdot t):=\lambda\cdot \sum\limits_{i=1}^{2n}t_i\cdot p^*\sigma_i$, for $\lambda\in \K$,
the section $s$ descends to tautological sections $s_{\SS}: SU_\K\rightarrow p^*\SS(E_\K)$ and $s_{\P}: U_\K\rightarrow p^*\P(E_\K)$.
By the dimension count, the transversality of the tautological sections $s_{\SS}$ and $s_{\P}$ follow directly from the transversality of $s$.

At $x_0=(z_0,t_0)\in C(E_\C)$, denote $(\frac{\partial}{\partial z})_{x_0}, (\frac{\partial}{\partial \bar z})_{x_0}$ as the tangent vector to $V\times \{t_0\}\subset V\times C_\K'$ and $v$ as the tangent vector to $\{z_0\}\times C_\K'\subset V\times C_\K'$ at $x_0$. 
Therefore, to show the transversality of $s$, it is equivalent to show that suppose there exist $a\in \C$ and a vertical tangent vector $v$ of $T_{x_0}C(E_\C)$ (that is, $p_*v=0$) such that \begin{equation}\label{LinearDependenceC}
a(p^*\nabla)_z s+\overline a(p^*\nabla)_{\overline z} s+(p^*\nabla)_vs=0,\end{equation} 
then $a$ is equal to 0 and $v$ is 0.

For the section $s(z, t)=\sum\limits_{i=1}^{2n}t_i\cdot p^*\sigma_i$ and a tangent vector $X$ of $C(E_\K)/\{0\}$, the rule for the pullback connection is
\begin{eqnarray*}
(p^*\nabla)_X s
&=&\sum_{i=1}^{2n}t_i\cdot p^*(\nabla_{p_*X}\sigma_i)+\sum_{i=1}^{2n}X(t_i)\cdot p^*\sigma_i.
\end{eqnarray*}

Therefore, 
\begin{eqnarray*}
(p^*\nabla)_{(\frac{\partial}{\partial z})_{x_0}}s=\sum_{i=1}^{2n}t_i\cdot p^*(\nabla_{(\frac{\partial}{\partial z})_{z_0}}\sigma_i),\quad
(p^*\nabla)_{(\frac{\partial}{\partial \bar z})_{x_0}}s=\sum_{i=1}^{2n}t_i\cdot p^*(\nabla_{(\frac{\partial}{\partial \bar z})_{z_0}}\sigma_i).\end{eqnarray*}
For a vertical direction $v$, we have $p_*(v)=0$, then 
\begin{equation*}
(p^*\nabla)_{v}s=\sum_{i=1}^{2n}v(t_i)\cdot p^*\sigma_i.
\end{equation*}

Equation (\ref{LinearDependenceC}) then reduces to: there exist $a\in \C$ and a vertical tangent vector $v$ at $T_{x_0}C(E_\C)\subset T_{x_0}E_\C$ such that 
\begin{equation}\label{LinearDependenceE}
a\nabla_z s+\overline a\nabla_{\overline z} s+v=0\quad \text{at some point $z_0=p(x_0)\in \Sigma$}.
\end{equation}
From now on, we will make computations in terms of the local holomorphic frame $\{dz^{\frac{2n-1}{2}}, dz^{\frac{2n-3}{2}},\cdots dz^{\frac{1-2n}{2}}\}$. 
\begin{itemize}
\item The tautological section $$s=\begin{pmatrix}h^{\frac{2n-1}{2}}t_1&h^{\frac{2n-3}{2}}t_2&\cdots&h^{\frac{1-2n}{2}}t_{2n}\end{pmatrix}^T$$ satisfies $\sum\limits_{i=1}^n t_{2i-1}\bar t_{2i}=0.$
\item The vertical tangent vector along each fiber is 
$$v=\begin{pmatrix}
h^{\frac{2n-1}{2}}v_1&h^{\frac{2n-3}{2}}v_2& \cdots&h^{\frac{1-2n}{2}}v_{2n}\end{pmatrix}^T$$
satisfying $\sum\limits_{i=1}^{n}v_{2i-1}\overline{t}_{2i}+\sum\limits_{i=1}^{n}t_{2i-1}\overline{v}_{2i}=0.$
\item The derivatives of $s$ with respect to $\frac{\partial}{\partial z}, \frac{\partial}{\partial \overline{z}}$ direction are 
\begin{eqnarray*}
&&\nabla_z s=(\partial+H^{-1}\partial H+\phi)s=-\begin{pmatrix}
\partial(h^{\frac{2n-1}{2}})t_1\\ \partial(h^{\frac{2n-3}{2}})t_2\\ \vdots \\ \partial(h^{\frac{1-2n}{2}})t_{2n}\end{pmatrix}+\begin{pmatrix}0\\r_1h^{\frac{2n-1}{2}}t_1\\ \vdots \\r_{2n-1}h^{\frac{3-2n}{2}}t_{2n-1}\end{pmatrix},\nonumber\\\\
&&\nabla_{\overline{z}} s=(\overline\partial+\phi^{*H})s=\begin{pmatrix}
\overline\partial(h^{\frac{2n-1}{2}})t_1\\  \vdots\\ \overline\partial(h^{\frac{3-2n}{2}})t_{2n-1}\\ \overline\partial(h^{\frac{1-2n}{2}})t_{2n}\end{pmatrix}+\begin{pmatrix}r_1h^{\frac{2n+1}{2}}t_2\\ \vdots \\r_{2n-1}h^{\frac{5-2n}{2}}t_{2n}\\0\end{pmatrix}.\nonumber
\end{eqnarray*}
\end{itemize}

Note that directly checking each entry of the vector-valued Equation (\ref{LinearDependenceE}) is not a wise idea since it is not easy to catch useful information. Instead, we introduce two test directions at point $x_0$:
\begin{eqnarray*}
P&=&\begin{pmatrix}
h^{\frac{1-2n}{2}}t_2&-h^{\frac{3-2n}{2}}t_1&h^{\frac{5-2n}{2}}t_4&-h^{\frac{7-2n}{2}}t_3& \cdots&-h^{\frac{2n-1}{2}}t_{2n-1}\end{pmatrix}^T,\\
Q&=&\begin{pmatrix}
h^{\frac{1-2n}{2}}t_2&h^{\frac{3-2n}{2}}t_1&h^{\frac{5-2n}{2}}t_4&h^{\frac{7-2n}{2}}t_3&\cdots&h^{\frac{2n-1}{2}}t_{2n-1}\end{pmatrix}^T.\end{eqnarray*}
The inner products of $v, \nabla_z s, \nabla_{\overline{z}}s$ with $P$ are as follows:
\begin{eqnarray}
\left<v,P\right>&=&\sum_{i=1}^nv_{2i-1}\overline{t}_{2i}-\sum_{i=1}^{n}v_{2i}\overline{t}_{2i-1}, \quad\text{hence}\quad \text{Im}\left<v,P\right>=0.\nonumber\\
\left<\nabla_z s,P\right>&=&A-\sum_{i=1}^n r_{2i-1}h|t_{2i-1}|^2+\sum_{i=1}^{n-1}r_{2i}ht_{2i}\overline{t}_{2i+2}.\nonumber\\
\left<\nabla_{\overline{z}} s,P\right>&=&B+\sum_{i=1}^nr_{2i-1}h|t_{2i}|^2-\sum_{i=1}^{n-1}r_{2i}ht_{2i+1}\overline{t}_{2i-1}.\nonumber\end{eqnarray}
Let $a_k=\frac{2n+1}{2}-k$.
Here \begin{eqnarray}
&A&=-\sum_{i=1}^n \partial(h^{a_{2i-1}})t_{2i-1}h^{-a_{2i-1}}\overline{t}_{2i}+\sum_{i=1}^n\partial(h^{a_{2i}})t_{2i}h^{-a_{2i}}\overline{t}_{2i-1}\nonumber\\
&&=\partial\log h(\sum_{i=1}^na_{2i}(t_{2i}\overline{t}_{2i-1}-t_{2i-1}\overline{t}_{2i})- \sum_{i=1}^nt_{2i-1}\bar t_{2i})\nonumber\\
&&=\partial\log h\cdot\sum_{i=1}^na_{2i}(t_{2i}\overline{t}_{2i-1}-t_{2i-1}\overline{t}_{2i})\nonumber
\end{eqnarray}

and \begin{eqnarray}B&=&\sum_{i=1}^n \overline\partial(h^{a_{2i-1}})t_{2i-1}h^{-a_{2i-1}}\overline{t}_{2i}-\sum_{i=1}^n\overline\partial(h^{a_{2i}})t_{2i}h^{-a_{2i}}\overline{t}_{2i-1}\nonumber\\
&=&\overline\partial\log h(\sum_{i=1}^na_{2i}(t_{2i-1}\bar t_{2i}-t_{2i}\overline{t}_{2i-1})+\sum_{i=1}^n t_{2i-1}\overline{t}_{2i})\nonumber\\
&=&\overline\partial\log h\cdot\sum_{i=1}^na_{2i}(t_{2i-1}\bar t_{2i}-t_{2i}\overline{t}_{2i-1})\nonumber
\end{eqnarray}

Therefore $A=\overline B$, and then $\mathrm{Im}(aA+\overline{a}B)=0$. Hence  
\begin{eqnarray}
0&=& i\ \mathrm{Im}(a\left<\nabla_z s,P\right>+\overline{a}\left<\nabla_{\overline{z}}s,P\right>+\left<v,P\right>)\nonumber\\
&=& h \cdot i\ \mathrm{Im}\left(\sum_{i=1}^n r_{2i-1}(-a|t_{2i-1}|^2+\overline{a}|t_{2i}|^2)+\sum_{i=1}^{n-1}r_{2i}(at_{2i}\overline{t}_{2i+2}-\overline{a}t_{2i+1}\overline{t}_{2i-1})\right)\nonumber
\end{eqnarray}
So we obtain the first important equation 
 \begin{equation}\label{EquationP} -i \ \mathrm{Im}(a)\cdot\sum_{i=1}^n r_{2i-1}(|t_{2i-1}|^2+|t_{2i}|^2)+i\sum_{i=1}^{n-1}r_{2i}\text{Im}(at_{2i}\overline{t}_{2i+2}-\overline{a}t_{2i+1}\overline{t}_{2i-1})=0.\end{equation}
Similarly, the inner product of $v, \nabla_z s, \nabla_{\overline{z}}s$ with $Q$ are as follows:
\begin{eqnarray}
\left<v,Q\right>&=&\sum_{i=1}^nv_{2i-1}\overline{t}_{2i}+\sum_{i=1}^{n}v_{2i}\overline{t}_{2i-1}, \quad\text{hence}\quad \mathrm{Re}\left<v,Q\right>=0.\nonumber\\
\left<\nabla_z s,Q\right>&=&C+\sum_{i=1}^n r_{2i-1}h|t_{2i-1}|^2+\sum_{i=1}^{n-1}r_{2i}ht_{2i}\overline{t}_{2i+2}.\nonumber\\
\left<\nabla_{\overline{z}} s,Q\right>&=&D+\sum_{i=1}^n r_{2i-1}h|t_{2i}|^2+\sum_{i=1}^{n-1}r_{2i}ht_{2i+1}\overline{t}_{2i-1}.\nonumber\end{eqnarray}

Here \begin{eqnarray}C&=&-\sum_{i=1}^n \partial(h^{a_{2i-1}})t_{2i-1}h^{-a_{2i-1}}\overline{t}_{2i}-\sum_{i=1}^n\partial(h^{a_{2i}})t_{2i}h^{-a_{2i}}\overline{t}_{2i-1}\nonumber\\
&=&-\partial\log h(\sum_{i=1}^na_{2i}(t_{2i-1}\overline{t}_{2i}+t_{2i}\overline{t}_{2i-1})+\sum_{i=1}^nt_{2i-1}\bar t_{2i})\nonumber\\
&=&-\partial\log h\cdot\sum_{i=1}^na_{2i}(t_{2i-1}\overline{t}_{2i}+t_{2i}\overline{t}_{2i-1})\nonumber
\end{eqnarray}

and \begin{eqnarray}D&=&\sum_{i=1}^n \overline\partial(h^{a_{2i-1}})t_{2i-1}h^{-a_{2i-1}}\overline{t}_{2i}+\sum_{i=1}^n\overline\partial(h^{a_{2i}})t_{2i}h^{-a_{2i}}\overline{t}_{2i-1}\nonumber\\
&=&\overline\partial\log h(\sum_{i=1}^na_{2i}(t_{2i-1}\bar t_{2i}+t_{2i}\overline{t}_{2i-1})+\sum_{i=1}^n t_{2i-1}\overline{t}_{2i})\nonumber\\
&=&\overline\partial\log h\cdot\sum_{i=1}^na_{2i}(t_{2i-1}\bar t_{2i}+t_{2i}\overline{t}_{2i-1})\nonumber
\end{eqnarray}
Therefore $C=-\bar D$ and then $\mathrm{Re}(aC+\overline{a}D)=0$. Hence
\begin{eqnarray}
0&=&\mathrm{Re}(a\left<\nabla_z s,Q\right>+\overline{a}\left<\nabla_{\overline{z}}s,Q\right>+\left<v,Q\right>)\nonumber\\
&=&h\cdot\mathrm{Re}\left(\sum_{i=1}^nr_{2i-1} (a|t_{2i-1}|^2+\overline{a}|t_{2i}|^2)+\sum_{i=1}^{n-1}r_{2i}(at_{2i}\overline{t}_{2i+2}+\overline{a}t_{2i+1}\overline{t}_{2i-1})\right)\nonumber
\end{eqnarray}
So we obtain the second important equation 
\begin{equation}\label{EquationQ}
\mathrm{Re}(a)\cdot \sum_{i=1}^nr_{2i-1} (|t_{2i-1}|^2+|t_{2i}|^2)+\sum_{i=1}^{n-1}r_{2i}\mathrm{Re}(at_{2i}\overline{t}_{2i+2}+\overline{a}t_{2i+1}\overline{t}_{2i-1})=0.
\end{equation}

Summing up Equations (\ref{EquationP}), (\ref{EquationQ}), we obtain
\begin{equation}\label{SameNorm}
\overline{a}\cdot\sum_{i=1}^n r_{2i-1}(|t_{2i-1}|^2+|t_{2i}|^2)+a\cdot\sum_{i=1}^{n-1}r_{2i}(t_{2i}\overline{t}_{2i+2}+\bar t_{2i+1}t_{2i-1})=0.
\end{equation}

Our goal is to show the above equality only holds if $a=0$.
Recall that $r_k=\sqrt{\frac{k(2n-k)}{2}}$ and set $r_0=r_{2n}=0$, then 
\begin{eqnarray}
&&4r_{2i-1}^2-(r_{2i}+r_{2i-2})^2\geq 4r_{2i-1}^2-2(r_{2i}^2+r_{2i-2}^2)\nonumber\\
&=& 2(2i-1)(2n-2i+1)-(2i)(2n-2i)-(2i-2)(2n-2i+2)=2.\nonumber
\end{eqnarray}
So we have $2r_{2i-1}>r_{2i}+r_{2i-2}$ and 
\begin{eqnarray}
&&\left|\sum_{i=1}^n r_{2i-1}(|t_{2i-1}|^2+|t_{2i}|^2)\right|>\frac{1}{2}\left|\sum_{i=1}^n(r_{2i}+r_{2i-2}) (|t_{2i-1}|^2+|t_{2i}|^2)\right|\nonumber\\
&&\geq \frac{1}{2}\left|\sum_{i=1}^{n-1}r_{2i}\Big(|t_{2i}|^2+|t_{2i+2}|^2+|t_{2i-1}|^2+|t_{2i+1}|^2)\right|\nonumber\\
&&\geq \left|\sum_{i=1}^{n-1}r_{2i}(t_{2i}\overline{t}_{2i+2}+\bar t_{2i+1}t_{2i-1})\right|,\nonumber
\end{eqnarray} 
where the first strict inequality follows from the fact that $t\neq 0$. 
Then Equation (\ref{SameNorm}) holds only if $a=0$ and thus $v=0$. We finish proving the transversality of $s$.
\end{proof}

The previous proposition gives a way to construct spherical and projective structures on $SU_\K$ and $U_\K$ respectively with a chosen holonomy. We get maps
\[ \overline{\mathcal{G}}_\K^F : \mathrm{Fuch}_{2n}(S)\ \ra\ \mathcal{D}_{\SS(\K^{2n})}(SU_\K)\,, \hspace{1cm} \mathcal{G}_\K^F : \mathrm{Fuch}_{2n}(S)\ \ra\ \mathcal{D}_{\K\P^{2n-1}}(U_\K)\,, \]
from the space of Fuchsian representations to the deformation spaces of geometric structures such that
\[ \mathrm{hol}(\overline{\mathcal{G}}_\K^F(\rho)) = \overline{\rho} \circ \overline{p}_*\,, \hspace{1cm} \mathrm{hol}(\mathcal{G}_\K^F(\rho)) = \rho \circ p_*\,. \]

\subsection{Nearby representations}

We will now extend the maps $\overline{\mathcal{G}}_\K^F$ and $\mathcal{G}_\K^F$ from the previous section to a small neighbourhood of $\mathrm{Fuch}_{2n}(S)$ in $\mathrm{Hit}_{2n}(S)$ or $\mathrm{QHit}_{2n}(S)$. We will see that, for representations that are close enough to $\mathrm{Fuch}_{2n}(S)$, the corresponding geometric structure can be constructed with the method of the graph of a geometric structure. From this, we will deduce Theorem \ref{thm:developing of the fiber}, about some special properties of these geometric structures.

\begin{prop}    \label{prop:nearby repr}
Let $\K = \R$ or $\C$. There is a connected open subset $\mathcal{N}_\K$ in $\mathrm{Hit}_{2n}(S)$ if $\K=\R$ or in $\mathrm{QHit}_{2n}(S)$ if $\K=\C$, and unique maps 
\[ \overline{\mathcal{G}}_\K : \mathcal{N}_\K\ \ra\ \mathcal{D}_{\SS(\K^{2n})}(SU_\K)\,, \hspace{1cm} \mathcal{G}_\K : \mathcal{N}_\K\ \ra\ \mathcal{D}_{\K\P^{2n-1}}(U_\K)\,, \]
such that
\begin{enumerate}
\item $\mathrm{hol}(\overline{\mathcal{G}}_\K(\rho)) = \overline{\rho} \circ \overline{p}_*$ and $\mathrm{hol}(\mathcal{G}_\K(\rho)) = \rho \circ p_*$. 
\item $\mathrm{Fuch}_{2n}(S) \subset \mathcal{N}_\K$.
\item The restrictions of $\overline{\mathcal{G}}_\K$ and $\mathcal{G}_\K$ to $\mathrm{Fuch}_{2n}(S)$ are $\overline{\mathcal{G}}_\K^F$ and $\mathcal{G}_\K^F$.
\end{enumerate}
\end{prop}  
\begin{proof}
Hitchin and quasi-Hitchin representations are absolutely irreducible, hence by Lemma \ref{lemma:Holonomy principle}, the map $\mathrm{hol}$ is a local homeomorphism. For every Fuchsian representation $\rho$ we can find a small connected open neighborhood $U_\rho$ of $\overline{\mathcal{G}}_\K^F(\rho)$ and $\mathcal{G}_\K^F(\rho)$ such that $\mathrm{hol}$ has a unique local inverse, and we will define $\overline{\mathcal{G}}_\K$ and $\mathcal{G}_\K$ on $\mathrm{hol}(U_\rho)$ to agree with the local inverse. The set $\mathcal{N}_\K$ can be defined as the union of all the open sets $\mathrm{hol}(U_\rho)$ for $\rho \in \mathrm{Fuch}_{2n}(S)$. 
\end{proof}

Now, we will construct a (possibly smaller) connected open neighborhood $\mathcal{O}_\K \subset \mathcal{N}_\K$, where we can see that the geometric structures $\overline{\mathcal{G}}_\K(\rho)$ and $\mathcal{G}_\K(\rho)$, for $\rho \in \mathcal{O}_\K$, can be constructed using the method of the graph of a geometric structure.

Fix a $\rho_0 \in \mathrm{Fuch}_{2n}(S)$, and choose a hyperbolic metric $h$ on $S$ that induces the Fuchsian representation $\rho_0$. With this hyperbolic metric, $S$ becomes a Riemann surface $\Sigma$.  

Let $(E_\C,\omega,Q,\phi)$ be the Higgs bundle on $\Sigma$ corresponding to the Fuchsian representation $\rho_0$, given by (\ref{eq:bundle EC}),(\ref{eq:bil Q}). 
In coordinates, we write the metric $h$ as $h^2|dz|^2$, and this induces a Hermitian metric $H$ on $E_\C$, as in (\ref{eq:metric}).  Using $H$ and the bilinear form $Q$ from (\ref{eq:bil Q}), we can define a real structure $\tau$ as in (\ref{eq:tau}), and a real subbundle
\[E_\R = \{\ v \in E_\C \ \mid \ \tau(v) = v \ \}\,.\]

Using the Hermitian metric $H$, we can define some new subbundles $C(E_\K)$ as in \eqref{DefinitionM}. The projective and spherical quotients of $C(E_\K)$ are homeomorphic to the manifolds $U_\K$ and $SU_\K$ respectively, and the pull-back bundles $p^* \P(E_\K)$ and $\overline{p}^* \SS(E_\K)$ have tautological sections $[s]_\P$  and $[s]_\SS$ as above. Notice that the construction of $C(E_\K)$ and the sections sections $[s]_\P$  and $[s]_\SS$ is carried out on a Higgs bundle for a Fuchsian representation.

We will now consider the (infinite dimensional) space $\mathcal{D}_\K$ of all smooth, absolutely irreducible, flat connections on the bundle $E_\K$ that preserve the volume form $\omega$. This space carries a free action of the gauge group of $E_\K$, and the quotient map induced by this action is
\[\pi:\mathcal{D}_\K\ \lra \ \mathrm{Hom}^{a.i.}(\pi_1(S),SL(2n,\K))/SL(2n,\K)\,.\] 
Since $\pi$ is the quotient map by a group action, it is an open map. 
 
Now we consider the flat connection $\nabla_\C \in \mathcal{D}_\C$ that solves Hitchin's equations for the Fuchsian Higgs bundle $(E_\C,\omega,Q,\phi)$. We denote by $\nabla_\R \in \mathcal{D}_\R$ its restriction to the real part $E_\R$. 

We proved in Theorem \ref{Independence} that $[s]_\P$  and $[s]_\SS$ are transverse sections for $\nabla_\K$. That proof used the theory of Higgs bundles for describing the Fuchsian flat connection $\nabla_\K$. Since the manifolds $U_\K$ and $SU_\K$ are compact, we can find an open connected neighborhood $U_{\nabla_\K}$ of $\nabla_\K$ in $\mathcal{D}_\K$ such that $[s]_\P$  and $[s]_\SS$ are transverse for all connections in $U_{\nabla_\K}$. Since the projection $\pi$ is open, the image $V_{\rho_0} = \pi(U_{\nabla_\K})$ is an open connected neighborhood of $\rho_0$. For every $\rho \in V_{\rho_0} \cap \mathcal{N}_\K$, the geometric structures determined by $[s]_\P$  and $[s]_\SS$ must coincide with $\overline{\mathcal{G}}_\K(\rho)$ and $\mathcal{G}_\K(\rho)$, by the uniqueness, and they are obtained from the graph of a geometric structure.

We can now define the subset $\mathcal{O}_\K$ as the union over $\rho_0 \in \mathrm{Fuch}_{2n}(S)$ of the $V_{\rho_0}$, intersected with $\mathcal{N}_\K$. From this we can deduce the following special property of the geometric structures constructed in a neighborhood of the Fuchsian representations: the developing image of every fiber is a standard embedding of the Stiefel manifold.

\begin{thm}   \label{thm:developing of the fiber}
For $\K = \R$ or $\C$, fix a representation $\rho\in \mathcal{O}_\mathbb{K}$. Denote by $\widetilde{SU_\mathbb{K}}$ and $\widetilde{U_\mathbb{K}}$ the universal coverings of $SU_\mathbb{K}$ and $U_\mathbb{K}$, and by 
\begin{align*}
D_{[s]_\SS}: \widetilde{SU_\mathbb{K}} &\ \lra\ \SS(\K^{2n})\,, \\
D_{[s]_\P}: \widetilde{U_\mathbb{K}} &\ \lra\ \K\P^{2n-1} 
\end{align*}
the developing maps corresponding to the geometric structures $\overline{\mathcal{G}}_\K(\rho)$ and $\mathcal{G}_\K(\rho)$. 
For every point $x\in S$, let $SF = \overline{p}^{-1}(x)$ and $F = p^{-1}(x)$ denote the fibers above $x$ in $SU_\mathbb{K}$ and $U_\mathbb{K}$ respectively. Let $\widetilde{SF}$ and $\widetilde{F}$ denote one of their lifts to the universal covering. 
Then their images via the developing map
\begin{align*}
D_{[s]_\SS}\left(\widetilde{SF}\right) &\subset \SS(\K^{2n})\,, \\
D_{[s]_\P}\left(\widetilde{F}\right) &\subset \R\P^{2n-1}
\end{align*}
are standard spherical or projective embeddings of the Stiefel manifolds. 
\end{thm}
\begin{proof}
Recall from the second paragraph of Section \ref{section:ConstructionOfProjectiveStructures} that the holonomy of the geometric structure $\mathcal{G}_\K(\rho)$ is $\hat{\rho} = \rho \circ p_*$, where $p:U_\K \ra S$ is the fiber bundle map and $\hat{\rho}$ is induced by a flat connection on the bundle $E_\K \ra S$, pulled back to $p^*E_\K \ra U_\K$. 

Let $\iota : F \ra U_\K$ be the inclusion of the fiber over $x\in S$. Since the connection on $p^*E_\K$ is a pull-back connection, when we restrict it to the bundle $\iota^*p^*E_\K \ra F$ we obtain a trivial connection: the bundle is a product bundle
\[\iota^*p^*E_\K \simeq \K^{2n} \times F \ra F\,, \] 
endowed with the product connection. Hence, the connection identifies all the fibers with a copy of $\K^{2n}$. 

The developing map $D_{[s]_\P}$ was defined as the $\rho$-equivariant map associated with the tautological section $s$ by the flat connection. When we restrict $D_{[s]_\P}$ to one fiber, we can compute it explicitly using the description of $\iota^*p^*E_\K$ as a product bundle. Using (\ref{DefinitionM}) and Lemma \ref{lemma:C-linear transformation}, we can see that the image $D_{[s]_\P}\left(\widetilde{F}\right)$ is a standard projective embedding of the projective Stiefel manifold. 

The proof for the geometric structure $\overline{\mathcal{G}}_\K(\rho)$ is the same.
\end{proof}

\section{Domains of discontinuity and developing Image}  \label{dods}
In Section \ref{section:ConstructionOfProjectiveStructures}, we constructed geometric structures $\overline{\mathcal{G}}_\K(\rho)$ and $\mathcal{G}_\K(\rho)$ associated with a Hitchin or quasi-Hitchin representation $\rho$ close enough to the Fuchsian locus.  These are structures on the manifolds $SU_\K$ and $U_\K$, manifolds whose topology is explicitly known. 

The construction given in Section \ref{section:ConstructionOfProjectiveStructures} was completely independent from the construction of geometric structures using domains of discontinuity briefly described in Section \ref{sec:dod}. In that Section, the geometric structures are on the manifolds $M_\K$ and $SM_\K$, whose topology is difficult to understand.  

In this section, we compare the structures constructed with the two methods. We will prove that, when $\rho\in \mathcal{O}_\K$, the structures $\overline{\mathcal{G}}_\K(\rho)$ and $\mathcal{G}_\K(\rho)$ are isomorphic to the structures constructed using domains of discontinuity. In particular, this proves that the manifold $SM_\K$ is homeomorphic to $SU_\K$, and $M_\K$ is homeomorphic to $U_\K$. In this way, we can determine the topology of the manifolds $SM_\K$ and $M_\K$.

\subsection{Comparison theorem}

\begin{thm} \label{thm:comparison}
For $\K = \R$ or $\C$, suppose that $\rho \in \mathcal{O}_\K$. Then the structures $\overline{\mathcal{G}}_\K(\rho)$ and $\mathcal{G}_\K(\rho)$ are isomorphic to the structures constructed in Section \ref{subsec:general_dod} and Theorem \ref{thm:Guichard-Wienhard-dod} using domains of discontinuity. In particular, $M_\K$ is homeomorphic to $U_\K$ and $SM_\K$ is homeomorphic to $SU_\K.$ 
\end{thm}
\begin{rem}In a paper in preparation, Guichard and Wienhard already proved that $SM_\R$ is homeomorphic to the total space of a fiber bundle over $S$ with fiber the Stiefel manifold $SF_\R$. No similar results were known for $SM_\C$ or $M_\C$. 
\end{rem}

\begin{cor} \label{cor:developing of the fiber}
For every $\rho \in \mathcal{O}_\K$, the geometric structures described in Section \ref{subsec:general_dod} and Theorem \ref{thm:Guichard-Wienhard-dod} using domains of discontinuity have the property stated in Theorem \ref{thm:developing of the fiber}:
the developing image of a fiber is a standard embedding of the Stiefel manifold.
\end{cor}

We conjecture that the corollary is also valid for all $\rho$ in $\mathrm{Hit}_{2n}(S)$ and $\mathrm{QHit}_{2n}(S)$, but for the moment we don't have the right tools to prove it in general.  

\begin{rem}
In the case $n=2$, $\K=\R$, results very similar to Theorem \ref{thm:comparison} and Corollary \ref{cor:developing of the fiber} were obtained by Baraglia \cite{BaragliaThesis}. 
\end{rem}

\begin{proof}[Proof of Theorem \ref{thm:comparison}]
We will first prove the theorem under the additional assumption that $\rho$ is Fuchsian. The general theorem for $\rho \in \mathcal{O}_\K$ will then follow from this special case: Proposition \ref{prop:nearby repr} and Theorem \ref{thm:Guichard-Wienhard-dod} give two sections of the map $\mathrm{hol}$. Since, by Lemma \ref{lemma:Holonomy principle}, the map $\mathrm{hol}$ is a local homeo, if the two sections coincide at a point, they coincide on the whole $\mathcal{O}_\K$, because $\mathcal{O}_\K$ is connected.  

Now we will assume that $\rho$ is Fuchsian. It is enough to show that $SU_\K\cong SM_\K$ since the other case $U_\K\cong M_\K$ follows immediately. The proof consists of four steps.

\subsubsection*{Step 1:}
We first study the Higgs bundles for a Fuchsian representation in $SL(2,\R)$. Our aim is to describe a trivialization given by parallel frames. The corresponding Higgs bundle is $\left(E^2_\C=K^{\frac{1}{2}}\oplus K^{-\frac{1}{2}},\begin{pmatrix}0&0\\{\sqrt{\frac{1}{2}}}&0\end{pmatrix}\right)$. Denote the corresponding flat connection by $\nabla^2$. 
We can identify the universal covering $\widetilde \Sigma$ of $\Sigma$ with the upper half plane 
\[\HH^2 = \{\ z = x+iy\ \mid\ y > 0 \ \}\,.\] 
Lift the flat bundle $(E^2_\C\rightarrow \Sigma, \nabla^2)$ to $\left(\widetilde E^2_\C\rightarrow \HH^2, \widetilde\nabla^2\right)$. With respect to the global holomorphic frame $\left\{dz^{\frac{1}{2}}, dz^{-\frac{1}{2}}\right\}$ of $\widetilde E_\C^2$, the Hermitian metric solving Hitchin equation is $\begin{pmatrix}h^{-1}&0\\0&h\end{pmatrix}$ for $h=\frac{1}{\sqrt{2}y}$, and the real covariant constant sections with respect to $\widetilde\nabla^2$ are of the form $\begin{pmatrix}
(a\overline z+b)e^{-\frac{i\pi}{4}}h\\(az+b)e^{\frac{i\pi}{4}}\end{pmatrix}$ for $a, b\in\mathbb R$. Let $(a,b)=(1,0), (0,1)$, we have a basis 
\begin{eqnarray}
&e_1(z)=h^{\frac{1}{2}}(\overline z\lambda^{-1}X+z\lambda Y),\quad e_2(z)=h^{\frac{1}{2}}(\lambda^{-1} X+\lambda Y)\end{eqnarray} 
for $\lambda=e^{\frac{i\pi}{4}}, X=\begin{pmatrix}h^{\frac{1}{2}}\\0\end{pmatrix}, Y=\begin{pmatrix}0\\h^{-\frac{1}{2}}\end{pmatrix}$. Then
\begin{equation}
X=-\frac{h^{\frac{1}{2}}\lambda^{-1}(e_1(z)-ze_2(z))}{\sqrt{2}},\quad Y=-\frac{h^{\frac{1}{2}}\lambda(e_1(z)-\overline{z}e_2(z))}{\sqrt{2}}.
\end{equation}
The parallel transport operator $(T^2)_{z_0}^z: (\widetilde E^2_\C)_z\rightarrow (\widetilde E_\C^2)_{z_0}$ with respect to $\widetilde\nabla^2$ is a linear isomorphism that takes the frame $\{e_1(z), e_2(z)\}$ to the frame $\{e_1(z_0),e_2(z_0)\}$ respectively. This trivialization for a Fuchsian Higgs bundle for $SL(2,\R)$ was described by Baraglia \cite[Sec. 3.3]{BaragliaThesis}. He then used it to understand a Fuchsian Higgs bundle for $SL(3,\R)$ and $SL(4,\R)$.

\subsubsection*{Step 2:}
We now apply the previous construction to a Fuchsian Higgs bundle for $SL(2n,\R)$. The symmetric product of the Higgs bundle $\left(K^{\frac{1}{2}}\oplus K^{-\frac{1}{2}},\begin{pmatrix}0&0\\{\sqrt{\frac{1}{2}}}&0\end{pmatrix}\right)$ can be described as 
\[E_\C=\text{Sym}^{2n-1}\left(K^{\frac{1}{2}}\oplus K^{-\frac{1}{2}}\right)=K^{\frac{2n-1}{2}}\oplus K^{\frac{2n-3}{2}}\oplus\cdots\oplus K^{\frac{1-2n}{2}}\] 
and the induced Higgs field
\begin{equation*}
\phi=\begin{pmatrix}
0&&&&\\
r_1&0&&&\\
&r_2&0&&\\
&&\ddots&\ddots&\\
&&&r_{2n-1}&0\\
\end{pmatrix},
\end{equation*} where $r_i=\sqrt{\frac{i(2n-i)}{2}}$ for $1\leq i\leq 2n-1$, with respect to the basis 
$$\{\sigma_1,\sigma_2,\cdots,\sigma_{2n}\}$$ for
$\sigma_k=h^{\frac{2n+1-2k}{2}}dz^{\frac{2n+1-2k}{2}}$ is a frame of $K^{\frac{2n+1-2k}{2}}$ whose lift to $\mathbb H^2$, also denoted $\sigma_k$. Note that $\sigma_k$ can be identified with 
$ \binom{2n-1}{k-1}^{\frac{1}{2}}X^{2n-k}Y^{k-1}$. Denote by $\nabla$ the corresponding flat connection. Again, we lift the bundle $(E_\C\rightarrow \Sigma, \nabla)$ to $\left(\widetilde E_\C\rightarrow \HH^2, \widetilde\nabla\right)$ and denote by $T_{z_0}^z$ the associated parallel transport operator. Lift the principal bundle $P_1\rightarrow \Sigma$ to $\widetilde P_1\rightarrow \HH^2$. A point in $\widetilde P_1$ is $ \{\sigma_1,\cdots, \sigma_{2n}\}_{z}\cdot e^{i\theta}$ parametrized by $(z, \theta)$ and $\tilde P_1$ is a trivial principal $U(1)$-bundle. Then $\widetilde E_\K=\widetilde{P}_1\times_{U(1)}\K^{2n}$ (note here $\R^{2n}$ means $\text{Fix}(\tau_0)\subset \C^{2n}$). A point in $\widetilde E_\K$ can be denoted by $(z,t)$ for $t\in\K^{2n}$.

Lift the fiber bundle $p: C(E_\K)\rightarrow \Sigma$ to $\tilde p: \overline{C(E_\K)}\cong \widetilde P_1\times_{U(1)}C_\K'\rightarrow \HH^2$, we have the pullback bundle $\tilde p^*(\widetilde E_\K)\rightarrow \overline{C(E_\K)}$ with the pullback connection $\tilde p^*\widetilde\nabla$. The tautological section $s: C(E_\K)\rightarrow p^*(E_\K)$ lifts to a tautological section $\tilde s:\overline{C(E_\K)}\rightarrow p^*(\widetilde E_\K)$. Since the holonomy of the connection $\tilde p^*\widetilde\nabla$ only depends on the homotopy classes of paths in $\mathbb H^2$, the tautological section $\tilde s$ gives rise to the $\pi_1(S)$-equivariant developing map 
$$D: \overline{C(E_\K)}\cong\widetilde P_1\times_{U(1)} C'_\K\ni (z,t)\longmapsto T_{(z_0,t_0)}^{(z,t)}(\tilde s(z,t))\in p^*(\widetilde E_\K)_{(z_0,t_0)}.$$ Note that $T_{(z_0,t_0)}^{(z,t)}(\tilde s(z,t))=p^*(T_{z_0}^z(z,t))$. By identifying $p^*(\widetilde E_\K)_{(z_0,t_0)}$ with $(\widetilde E_\K)_{z_0}$, we have $D(z,t)=T_{z_0}^z(z,t)\in (\widetilde E_\K)_{z_0}$.

For a point $(z,t)=\sum_{k=1}^{2n}t_k\sigma_k\in \overline{C(E_{\mathbb K})}$, the image under the developing map $D$ is 
\begin{eqnarray}
&&D\left(\sum_{k=1}^{2n}t_k\sigma_k\right)=\sum_{k=1}^{2n}t_k\cdot T_{z_0}^z(\sigma_k)\nonumber\\
&=&\sum_{k=1}^{2n}t_k\cdot \binom{2n-1}{k-1}^{\frac{1}{2}}T_{z_0}^z(X^{2n-k}Y^{k-1})\nonumber\\
&=&\sum_{k=1}^{2n}t_k\cdot \binom{2n-1}{k-1}^{\frac{1}{2}}\Big((T^2)_{z_0}^z(X)\Big)^{2n-k}\Big((T^2)_{z_0}^z(Y)\Big)^{k-1}\nonumber\\
&=&-(2\sqrt{2}y)^{\frac{1-2n}{2}}\sum_{k=1}^{2n}\Bigg[\binom{2n-1}{k-1}^{\frac{1}{2}}t_k(\lambda^{-1}(e_1(z_0)-ze_2(z_0)))^{2n-k}\nonumber\\
&& \hspace{8.5cm} \cdot (\lambda(e_1(z_0)-\bar ze_2(z_0)))^{k-1}\Bigg]\,. \label{PolynomialExpression}    
\end{eqnarray}

The vector space $(\widetilde E_\K)_{z_0}$ is isomorphic to the space of homogeneous polynomials of $e_1(z_0), e_2(z_0)$ of degree $2n-1$. The developing map $D$ descends to $D_\SS: \overline{SU_\K}\cong \widetilde P_1\times_{U(1)}SF_\K\ni[(z,t)]\mapsto [T_{z_0}^z(z,t)]\in \SS(\widetilde E_\K)_{z_0}$.

\subsubsection*{Step 3:} 
We will first discuss the special case when $\K=\R$ and $n=2$. In this case, we can directly show that $D_\SS$ is a diffeomorphism from $\overline{SU_\K}$ onto the domain $S\Omega_\R$. Another argument for this fact will be given in the proof of Theorem \ref{thm:Fibration}.

It's known that $SK_\R$ contains degree 3 homogeneous polynomials containing a real root of multiplicity $\geq 2$. So $S\Omega_\R$ is disconnected and has two connected components $\Omega_1,\Omega_2$, where $\Omega_1$ consists of polynomials with three distinct real roots and $\Omega_2$ consists of polynomials with a pair of conjugate strictly complex roots and a single real root. Note that a real root here means it is valued in $\mathbb R\cup\{\infty\}.$ Observe that a polynomial in $\Omega_1, \Omega_2$ is uniquely determined by its roots up to sign.

Recall $SF_\R^2=\{\ (t_1, t_2, \bar t_2, \bar t_1)\ \mid\ t_1\bar t_2=0\ \}/\R^{+}$ and hence $SF_\R^2$ is disconnected and has two connected components 
\[F_1=\{\ (e^{-3i\theta}, 0, 0, e^{3i\theta})\ \mid\ \theta\in \left[0,\tfrac{2\pi}{3}\right)\ \}\] 
and 
\[F_2= \{\ (0,e^{-i\theta},e^{i\theta}, 0)\ \mid\ \theta\in[0,2\pi)\ \}\,.\] 
Note that the action of $U(1)$ preserves each component.

First, we will show that $D_\SS$ is a diffeomorphism from $\widetilde P_1\times_{U(1)}F_1$ onto $\Omega_1$. 
For a point $e^{-3i\theta}\sigma_1+e^{3i\theta}\sigma_4\in \widetilde P_1\times_{U(1)}F_1$, 
\begin{eqnarray*}
D_\SS(e^{-3i\theta}\sigma_1+e^{3i\theta}\sigma_4)=-e^{-3i\theta}(\lambda^{-1}(e_1(z_0)-ze_2(z_0)))^3-e^{3i\theta}(\lambda(e_1(z_0)-\bar ze_2(z_0)))^3.
\end{eqnarray*}
We will use the identification between projective real roots with $\mathbb R\cup \{\infty\}$:  $[a,1]$ with $a$, $[b, 0]$ with $\infty$. The roots of the polynomial $a_1, a_2, a_3$ are respectively
\[\frac{e^{-i\theta}\lambda^{-1}z+e^{i\theta}\lambda\bar z}{e^{-i\theta}\lambda^{-1}+e^{i\theta}\lambda}, \frac{e^{-i\frac{\pi}{3}}e^{-i\theta}\lambda^{-1}z+e^{i\frac{\pi}{3}}e^{i\theta}\lambda\bar z}{e^{-i\frac{\pi}{3}}e^{-i\theta}\lambda^{-1}+e^{i\frac{\pi}{3}}e^{i\theta}\lambda}, \frac{e^{-i\frac{2\pi}{3}}e^{-i\theta}\lambda^{-1}z+e^{i\frac{2\pi}{3}}e^{i\theta}\lambda\bar z}{e^{-i\frac{2\pi}{3}}e^{-i\theta}\lambda^{-1}+e^{i\frac{2\pi}{3}}e^{i\theta}\lambda}.\]
So the polynomial has three projective real roots. Suppose $a_1,a_2,a_3\in \R$ or $a_1=\infty, a_2, a_3\in \mathbb R$ or there are at least two of $a_1, a_2, a_3$ are $\infty$. 

For $z=re^{i\phi}$ and let $\theta'=\theta+\frac{\pi}{4}\in [\frac{\pi}{4}, \frac{11\pi}{12})$, we will find $(r, \phi, \theta')$ solving the system
\begin{eqnarray}
&&r\cdot \frac{\cos(\theta'-\phi)}{\cos(\theta')}=a_1,\label{System:roots}\\
&&r\cdot \frac{\cos(\theta'+\frac{\pi}{3}-\phi)}{\cos(\theta'+\frac{\pi}{3})}=a_2,\nonumber\\
&&r\cdot \frac{\cos(\theta'+\frac{2\pi}{3}-\phi)}{\cos(\theta'+\frac{2\pi}{3})}=a_3,\nonumber
\end{eqnarray} satisfying $r\geq 0, \phi\in (0,\pi)$ and $\theta'\in [\frac{\pi}{4},\frac{11\pi}{12}).$ Note that the polynomial given by $(r,\phi,\theta'+\frac{\pi}{3})$ is exactly the negative of the polynomial given by $(r,\phi, \theta')$. So we can reduce to consider $\theta'\in [\frac{\pi}{4}, \frac{7\pi}{12})$.\\
(i) Suppose there are at least one of $a_1,a_2,a_3$ are $\infty$, then  $\theta'=\frac{\pi}{2}$ and the last two equations of the system (\ref{System:roots}) are 
\begin{eqnarray*}
r(\cos\phi-\frac{1}{\sqrt{3}}\sin\phi)=a_2,\\
r(\cos\phi+\frac{1}{\sqrt{3}}\sin\phi)=a_3,
\end{eqnarray*} which is equivalent to $r\cos\phi=\frac{a_2+a_3}{2}, r\sin\phi=\frac{\sqrt{3}(a_3-a_2)}{2}.$ So $a_3> a_2$ if and only if $r\sin\phi> 0$, that is, $z\in \mathbb H^2$.  

So if $z\in \mathbb H^2$, the image of $D_{\mathbb S}$ lies in $\Omega_1$. \\
(ii) Suppose $a_1, a_2, a_3\in \R$. Let $\xi=\tan\theta'$. Then the system (\ref{System:roots}) becomes 
\begin{eqnarray*}
&&r(\cos\phi+\xi\sin\phi)=a_1,\\
&&r((1-\sqrt{3}\xi)\cos\phi+(\xi+\sqrt{3})\sin\phi)=a_2(1-\sqrt{3}\xi),\\
&&r((1+\sqrt{3}\xi)\cos\phi+(\xi-\sqrt{3})\sin\phi)=a_3(1+\sqrt{3}\xi).
\end{eqnarray*}
Summing over the last two equations, we obtain
\begin{equation*}
2r(\cos\phi+\xi\sin\phi)=a_2(1-\sqrt{3}\xi)+a_3(1+\sqrt{3}\xi).
\end{equation*}
By comparing with the first equation, we obtain
\begin{equation*}
2a_1=a_2(1-\sqrt{3}\xi)+a_3(1+\sqrt{3}\xi)\Longrightarrow \xi=\tan\theta'=\frac{2a_1-(a_2+a_3)}{\sqrt{3}(a_3-a_2)}.
\end{equation*} 
So there exists a unique $\theta'\in [\frac{\pi}{4}, \frac{7\pi}{6}).$
So $1-\sqrt{3}\xi=\frac{2(a_3-a_1)}{a_3-a_2}$ and thus $a_3<a_2.$
Then we obtain 
\begin{eqnarray*}r\sin\phi&=&\frac{a_2(1-\sqrt{3}\xi)-a_1(1-\sqrt{3})\xi}{(\xi+\sqrt{3})-\xi(1-\sqrt{3}\xi)}=\frac{2(a_2-a_1)(a_3-a_1)}{\sqrt{3}(1+\xi^2)(a_3-a_2)}\\
r\cos\phi&=&a_1-\xi r\sin\phi .
\end{eqnarray*}
So if $z\in \mathbb H^2$, equivalently, $r\sin\phi>0$, the roots $a_1,a_2,a_3$ are distinct and thus the image of $D_{\mathbb S}$ lies in $\Omega_1$. Conversely, if $a_3>a_1>a_1$, we obtain $z\in \mathbb H^2$.

From the above calculation, for any polynomial $P$ in $\Omega_1$, then there uniquely exists a tuple $(\theta', r, \phi)\in [\frac{\pi}{4},\frac{7\pi}{12})\times (0,+\infty)\times (0,\pi)$ solving the system (\ref{System:roots}). So the map $D_{\SS}$ is a bijection. Therefore, we finish proving that $D_\SS$ is a diffeomorphism from $\widetilde P_1\times_{U(1)}F_1$ onto $\Omega_1$.

Next we will show that $D_\SS$ is a diffeomorphism from $\widetilde P_1\times_{U(1)}F_2$ onto $\Omega_2$. For a point $e^{-i\theta}\sigma_2+e^{i\theta}\sigma_3\in \widetilde P_1\times_{U(1)}F_2$. 
\begin{eqnarray*}
&&D_\SS(e^{-i\theta}\sigma_2+e^{i\theta}\sigma_3)=-(e_1(z_0)-ze_2(z_0))(e_1(z_0)-\bar ze_2(z_0))\\
&&\hspace{3.5cm}\cdot\Big(e^{-i\theta}(\lambda^{-1}(e_1(z_0)-ze_2(z_0)))+e^{i\theta}(\lambda(e_1(z_0)-\bar ze_2(z_0)))\Big)\,.\end{eqnarray*} The roots of the polynomial are
\[z, ~\bar z,~ \frac{e^{-i\theta}\lambda^{-1}z+\lambda\bar z e^{i\theta}}{e^{-i\theta}\lambda^{-1}+\lambda e^{i\theta}}\] for $z\in \mathbb H^2$. Write $z=re^{i\phi}$ for $r\in \mathbb R_{\geq 0}, \phi\in (0,\pi)$. Suppose the projective real root is $a$, then we have 
\[r\cdot \frac{\cos(\frac{\pi}{4}+\theta-\phi)}{\cos(\frac{\pi}{4}+\theta)}=a,\]
which is equivalent to $\theta=\arctan(\frac{ a-r\cos\phi}{r\sin\phi})-\frac{\pi}{4}$ or $\theta=\arctan(\frac{ a-r\cos\phi}{r\sin\phi})+\frac{3\pi}{4}.$ Thus we finish proving that $D_\SS$ is a diffeomorphism from $\widetilde P_1\times_{U(1)}F_2$ onto $\Omega_2$.

\subsubsection*{Step 4:}
Now, we see the general case, when $\K=\R, n\geq 3$ or $\K=\C, n\geq 2$. In this case, we will also show that $D_\SS$ is a diffeomorphism onto the domain $S\Omega_\K$. The proof relies on Lemma \ref{DevelopMap} and Proposition \ref{Equivariance}, proved in the following.

By Lemma \ref{DevelopMap},  we obtain $D_\SS((\overline{SU_\K})_{z=i})\cap SK_{\mathbb K}=\emptyset$. By the $SL(2,\mathbb R)$-equivariance of the developing map $D_{\mathbb S}$ in Proposition \ref{Equivariance} and the $SL(2,\mathbb R)$-invariance of the subset $SK_{\mathbb K}$, we obtain that $D_\SS(\overline{SU_\K})\cap SK_{\mathbb K}=\emptyset$. Hence $D_\SS(\overline{SU_\K})\subset S\Omega_\K$. Now we are going to show that they are in fact equal. 

Since the developing map $D_\SS: \overline{SU_\K}\rightarrow S\Omega_\K$ is $\pi_1(S)$-equivariant, therefore it descends to a map, denoted $\underline{D}_\SS: SU_\K\rightarrow SM_\K$, which is again a local diffeomorphism. Both $SU_\K$ and $SM_\K$ are compact, implying that $\underline{D}_\SS$ is proper. Therefore $\underline{D}_\SS$ is a covering map and so is $D_\SS$. 

For $\K=\C$ and $n\geq 2$, $SK_\C=\SS^{4n-1}\setminus S\Omega_\C$ is of codimension $2n-1\geq 3$, hence $\pi_1(S\Omega_\C)=1$. From Section \ref{DescriptionOfFibers}, $\pi_1(\overline{SU_\C})=\pi_1(SF_\C^n)=\pi_1(SF_\R^{2n})=1.$

For $\K=\R$ and $n\geq 4$, $SK_\R=\SS^{2n-1}\setminus S\Omega_\R$ is of codimension $n-1\geq 3$, hence $\pi_1(S\Omega_\R)=1$. From Section \ref{DescriptionOfFibers}, if $n\geq 4$, $\pi_1(\overline{SU_\R})=\pi_1(SF_\R^n)=1$.  

For $\K=\R$ and $n=3$, we have $\pi_1(S\Omega_\R)=\mathbb Z_2$ following from the proof of Theorem 11.3 in \cite{GW}.  From Section \ref{DescriptionOfFibers},  $\pi_1(\overline{SU_\R})=\pi_1(SF_\R^3)=\mathbb Z_2.$

By comparing the fundamental groups of $\overline{SU_\K}$ and $S\Omega_\K$, we obtain $D_\SS$ is indeed a diffeomorphism from $\overline{SU_\K}$ onto $S\Omega_\K$. Hence $\underline{D}_\SS$ is a diffeomorphism from $SU_\K$ onto $SM_\K$.
\end{proof}

\begin{lem}\label{DevelopMap}
$D_\SS((\overline{SU_\K})_{z=i})\cap SK_{\mathbb K}=\emptyset$. 
\end{lem}
\begin{proof}
The proof relies on Lemma \ref{NonIntersect}, proved in the next section.
Up to a scalar of $\mathbb R^+$, for $z=i$, the polynomial $P$ in Equation (\ref{PolynomialExpression}) is 
\begin{eqnarray*}P&=&\sum_{k=1}^{2n}\binom{2n-1}{k}^{\frac{1}{2}}t_ke^{\frac{i(2k-2n-1)\pi}{4}}(e_1(z_0)-ie_2(z_0))^{2n-k}(e_1(z_0)+i e_2(z_0))^{k-1}\\
&=&\sum_{k=0}^{2n-1}\binom{2n-1}{k}^{\frac{1}{2}}t_{2n-k}e^{\frac{i(2n-2k-1)\pi}{4}}(e_1(z_0)-ie_2(z_0))^{k}(e_1(z_0)+ie_2(z_0))^{2n-k-1}.\\
\end{eqnarray*}
We use the notation $Z=\frac{1}{2}(e_1(z_0)-ie_2(z_0)), W=\frac{1}{2}(e_1(z_0)+ie_2(z_0)).$
Then up to a scalar of $\mathbb R^+$, $$P=\sum_{k=0}^{2n-1}\binom{2n-1}{k}^{\frac{1}{2}}t_{2n-k}e^{\frac{i(2n-2k-1)\pi}{4}}Z^kW^{2n-1-k}.$$

Set $s_k=t_{2n-k}e^{\frac{i(2n-2k-1)\pi}{4}}$. Then $s_k$'s satisfy $\sum_{k=0}^{n-1}s_{2k}\overline{s_{2k+1}}=i\cdot \sum_{k=0}^{n-1}t_{2k}\overline{t_{2k+1}}=0$. Set  $p_k=\binom{2n-1}{k}^{\frac{1}{2}}s_k$, then the polynomial can be written as
\[P=\sum_{k=0}^{2n-1}p_kZ^kW^{2n-1-k}\,.\]

In order to relate this notation with the notation in Section  \ref{sec:Geometricdescription}, the $X, Y$ of Section \ref{sec:Geometricdescription} correspond to $e_1(z_0), e_2(z_0)$ here, and the $Z, W$ of Section \ref{sec:Geometricdescription} correspond to the the $Z, W$ defined here. 

Thus the equation for $s_k$ implies that $P\in C^{\lambda}_{\mathbb K}$ for $\lambda$ as in Lemma \ref{lem:case2lambdatransverse}: 
\[0=\sum_{k=0}^{n-1}s_{2k}\overline{s_{2k+1}}=\sum_{k=0}^{n-1}\binom{2n-1}{2k}^{-\frac{1}{2}}\binom{2n-1}{2k+1}^{-\frac{1}{2}}p_{2k}\overline{p_{2k+1}}=q_\lambda(P,P)\,.\]
The statement then follows from Lemma \ref{NonIntersect}.
\end{proof}

\subsection{\texorpdfstring{$SL(2,\mathbb R)$}{SL(2,R)}-equivariance of the developing maps}
The group $SL(2,\mathbb R)$ acts on $\mathbb H^2$ by the M\"obius transformation. It induces a natural left action of $SL(2,\mathbb R)$ on the canonical line bundle $K$ on $\mathbb H^2$: for a differential form $w\in K_z$, $\gamma\in SL(2,\mathbb R)$, 
\[\gamma\cdot w:=(\gamma^{-1})^*(w)\in K_{\gamma(z)}\,.\]
Then it induces a natural left action of $SL(2,\mathbb R)$ on $K^{\frac{1}{2}}, K^{-\frac{1}{2}}$, the frame bundle $\widetilde P_1$ and  $\widetilde P_1\times_{U(1)}\K^{2n}\cong \widetilde E_\K$.
So the fiber of $\overline{SU_\K}$ at $i$ is mapped to the fiber of $\overline{SU_\K}$ at $\gamma\cdot i$ under the action of $\gamma$.

\begin{lem}\label{LeftAction}
For $\gamma=\begin{pmatrix}a&b\\c&d\end{pmatrix}\in SL(2,\mathbb R)$ and the frame $e_{z'}=h(z)^{\frac{1}{2}}(dz)_{z'}^{\frac{1}{2}}$ of $K^{\frac{1}{2}}$, we have
\begin{equation*}
(\gamma^{-1})^*(e_{z'})=e_{\gamma(z')}\cdot \frac{cz'+d}{|cz'+d|}.
\end{equation*}
\end{lem}
\begin{proof}
We recall the fact $h(z)=\frac{1}{\sqrt{2}y}, \gamma'(z)=(cz+d)^{-2}, \text{Im}(\gamma(z))=\frac{y}{|cz+d|^2}.$ Then 
\begin{align*}
(\gamma^{-1})^*\Big(h(z')^{\frac{1}{2}}(dz)_{z'}^{\frac{1}{2}}\Big)&=h(z')^{\frac{1}{2}}\cdot (\gamma'(z))^{-\frac{1}{2}}_{z'}\cdot (dz)^{\frac{1}{2}}_{\gamma(z')}\\
&=\left(\sqrt{2}y'\right)^{-\frac{1}{2}}\cdot (cz'+d)\cdot (dz)^{\frac{1}{2}}_{\gamma(z')}\\
&=\left(\sqrt{2}\cdot \frac{y'}{|cz'+d|^2}\right)^{-\frac{1}{2}}\cdot \frac{cz'+d}{|cz'+d|}\cdot (dz)^{\frac{1}{2}}_{\gamma(z')}\\
&=\left(\sqrt{2}\cdot \text{Im}(\gamma(z'))\right)^{-\frac{1}{2}}\cdot \frac{cz'+d}{|cz'+d|}\cdot (dz)^{\frac{1}{2}}_{\gamma(z')}. \qedhere
\end{align*}
\end{proof}

\begin{prop}\label{Equivariance}
By identifying $(\widetilde E_\K)_{z_0}$ with $\K^{2n}$, the developing maps $D: \overline{C(E_\K)}\rightarrow \K^{2n}$,  $D_{\SS}: \overline{SU_\K}\rightarrow \SS(\K^{2n})$ and $D_{\P}: \overline{U_\K}\rightarrow \K\P^{2n-1}$ are $SL(2,\mathbb R)$-equivariant, that is, for $\gamma\in SL(2,\R)$,
\[D(\gamma \cdot p)=\gamma\cdot D(p), \quad\quad D_{\SS}(\gamma \cdot p)=\gamma\cdot D_{\SS}(p), \quad\quad D_{\P}(\gamma \cdot p)=\gamma\cdot D_{\P}(p) \,.\]
\end{prop}
\begin{proof} It is enough to show the equivariance of the map $D$. 
First, \begin{eqnarray*}(T^2)_{z_0}^z(X)&=&(T^2)_{z_0}^{z}(\lambda^{-1}\sqrt{2}^{-1}h(z)^{\frac{1}{2}}(e_1(z)-ze_2(z))\\
&=&\lambda^{-1}\sqrt{2}^{-1}h(z)^{\frac{1}{2}}(e_1(z_0)-e_2(z_0))\\
&=&\begin{pmatrix}e_1(z_0)&e_2(z_0)\end{pmatrix}\cdot \lambda^{-1}\sqrt{2}^{-1}h(z)^{\frac{1}{2}}\begin{pmatrix}1\\-z\end{pmatrix}.
\end{eqnarray*}
Let $\gamma=\begin{pmatrix}a&b\\c&d\end{pmatrix}$. By Lemma \ref{LeftAction}, $\gamma\cdot X=X\cdot e^{i\theta_{\gamma, z}}$ and $\gamma\cdot Y=Y\cdot e^{-i\theta_{\gamma, z}}$, then \begin{eqnarray*}
(T^2)_{z_0}^{\gamma(z)}(\gamma\cdot X)&=&(T^2)_{z_0}^{\gamma(z)}(X\cdot e^{i\theta_{\gamma, z}})\\
&=&(T^2)_{z_0}^{\gamma(z)}(\lambda^{-1}\sqrt{2}^{-1}h(\gamma(z))^{\frac{1}{2}}e^{i\theta_{\gamma, z}}(e_1(\gamma(z))-\gamma(z)e_2(\gamma(z)))\\
&=&\lambda^{-1}\sqrt{2}^{-1}h(\gamma(z))^{\frac{1}{2}}\frac{cz+d}{|cz+d|}(e_1(z_0)-\gamma(z)e_2(z_0))\\
&=&\begin{pmatrix}e_1(z_0)&e_2(z_0)\end{pmatrix}\cdot (cz+d)\lambda^{-1}\sqrt{2}^{-1}h(z)^{\frac{1}{2}}\begin{pmatrix}1\\-\gamma(z)\end{pmatrix}\\
&=&\begin{pmatrix}e_1(z_0)&e_2(z_0)\end{pmatrix}\cdot \lambda^{-1}\sqrt{2}^{-1}h(z)^{\frac{1}{2}}\begin{pmatrix}cz+d\\-az-b\end{pmatrix}\\
&=&\begin{pmatrix}e_1(z_0)&e_2(z_0)\end{pmatrix}\cdot (\gamma^{-1})^{t}\cdot \Big(\lambda^{-1}\sqrt{2}^{-1}h(z)^{\frac{1}{2}}\begin{pmatrix}1\\-z\end{pmatrix}\Big).
\end{eqnarray*}

We carry out similar computation for $Y$ and obtain
\begin{eqnarray*}
&&\begin{pmatrix}(T^2)_{z_0}^z(X)&(T^2)_{z_0}^z(Y)\end{pmatrix}=\begin{pmatrix}e_1(z_0)&e_2(z_0)\end{pmatrix}\cdot M(z),\\
&&\begin{pmatrix}(T^2)_{z_0}^{\gamma(z)}(e^{i\theta_{\gamma, z}}X)&(T^2)_{z_0}^{\gamma(z)}(e^{-i\theta_{\gamma,z}}Y)\end{pmatrix}=\begin{pmatrix}e_1(z_0)&e_2(z_0)\end{pmatrix}\cdot (\gamma^t)^{-1}\cdot M(z),\end{eqnarray*}
for $M(z)=\sqrt{2}^{-1}h(z)^{\frac{1}{2}}\cdot \begin{pmatrix}\lambda^{-1}&\lambda \\-\lambda^{-1}z&-\lambda \bar z\end{pmatrix}$.

Suppose $(U,V)$ is a frame of $(\widetilde E_\K)_{z_0}$ satisfying 
$$\begin{pmatrix}U&V\end{pmatrix}=\begin{pmatrix}e_1(z_0)&e_2(z_0)\end{pmatrix}\cdot A, \quad A\in SL(2,\C),$$ denote by $\tau:SL(2,\C)\rightarrow SL(2n,\C)$ the induced representation such that 
\begin{eqnarray*}
&&\begin{pmatrix}U^{2n-1}&U^{2n-2}V&\cdots&V^{2n-1}\end{pmatrix}\\
&&=\begin{pmatrix}e_1(z_0)^{2n-1}&e_1(z_0)^{2n-2}e_2(z_0)&\cdots&e_2(z_0)^{2n-1}\end{pmatrix}\cdot \tau(A).
\end{eqnarray*}
Consider the map $D: \widetilde P_1\times_{U(1)}\K^{2n}\cong \widetilde E_\K\ni (z, \vec u)\rightarrow T_{z_0}^z(z,\vec{u})\in(\widetilde E_\K)_{z_0}$. 

Let $\vec{u}=\begin{pmatrix}u_1\\u_2\\\vdots\\u_{2n}\end{pmatrix}$ and $\vec{u}'=\begin{pmatrix}\sqrt{C_{2n-1}^0}u_1\\\sqrt{C_{2n-1}^1}u_2\\\vdots\\\sqrt{C_{2n-1}^{2n-1}}u_{2n}\end{pmatrix}$. So 
\begin{eqnarray*}
D((z, \vec{u}))
&=&T_{z_0}^z(\sum_{i=1}^{2n}u_i\sqrt{C_{2n-1}^{i-1}}X^{2n-i}Y^{i-1})\\
&=&\sum_{i=1}^{2n}u_i\sqrt{C_{2n-1}^{i-1}}(T^2)_{z_0}^z(X)^{2n-i}(T^2)_{z_0}^z(Y)^{i-1}\\
&=&\begin{pmatrix}e_1(z_0)^{2n-1}&e_1(z_0)^{2n-2}e_2(z_0)&\cdots&e_2(z_0)^{2n-1}\end{pmatrix}\cdot \tau(M(z))\cdot \vec{u}',
\end{eqnarray*}
and since the $SL(2,\mathbb R)$ action on $\widetilde P\times_{U(1)}\K^{2n}$ is $\gamma\cdot (z,\vec u)=(\gamma(z), e^{\theta_{\gamma,z}}\cdot\vec{u}),$ \begin{eqnarray*}
&&D(\gamma\cdot (z, \vec{u}))=D([(\gamma(z), e^{\theta_{\gamma, z}}\cdot \vec{u})])\\
&=&T_{z_0}^{\gamma(z)}(\sum_{i=1}^{2n}u_i\cdot \sqrt{C_{2n-1}^{i-1}}\cdot (e^{i\theta_{\gamma, z}}X)^{2n-i}\cdot (e^{-i\theta_{\gamma, z}}Y)^{i-1})\\
&=&\sum_{i=1}^{2n}u_i\sqrt{C_{2n-1}^{i-1}}(T^2)_{z_0}^{\gamma(z)}(e^{i\theta_{\gamma, z'}}X)^{2n-i}(T^2)_{z_0}^{\gamma(z')} (e^{-i\theta_{\gamma, z'}}Y)^{i-1}\\
&=&\begin{pmatrix}e_1(z_0)^{2n-1}&e_1(z_0)^{2n-2}e_2(z_0)&\cdots&e_2(z_0)^{2n-1}\end{pmatrix}\cdot \tau((\gamma^t)^{-1}\cdot M(z))\cdot \vec{u}'\\
&=&\gamma\cdot D(z, \vec{u}).
\end{eqnarray*}

The last equality follows from the following: for a polynomial $P$ of $e_1(z_0), e_2(z_0)$,
\begin{equation*}
\gamma\cdot P(e_1(z_0), e_2(z_0))=P\left(\gamma^{-1}\cdot \begin{pmatrix}e_1(z_0)\\e_2(z_0)\end{pmatrix}\right)=P\left((e_1(z_0),e_2(z_0)\right)\cdot (\gamma^t)^{-1}). \qedhere
\end{equation*}
\end{proof}

\section{Geometric description of the fibration}
\label{sec:Geometricdescription}

In this section we are using again the notation introduced in Section \ref{subsec:fuchsian_dod}, in particular, $V_\mathbb{K}$ will be the space $\mathbb{K}^{(2n-1)}[X,Y]$ of homogeneous polynomials in two variables $X, Y$ of degree $2n-1$.

We will now consider two new variables $W,Z$ defined as 
\[W = \tfrac{1}{2}(X + i Y), \hspace{2cm} Z = \tfrac{1}{2}(X - i Y)\,.\]
In reverse, we have $X=Z+W$ and $Y=i(Z-W)$. We can now identify $V_\C$ with $\mathbb{C}^{(2n-1)}[Z,W]$, and $V_\mathbb{R}$ with the subspace 
\[ \{\ P\in \mathbb{C}^{(2n-1)}[Z,W]\ \mid \ P(Z,W)=\overline{P}(W,Z) \ \}\,. \]

The real form on $V_\mathbb{C}$ of which $V_\mathbb{R}$ is the real locus is the form: 
\[\tau: V_\C \ni \sum_{k=0}^{2n-1}p_{k}Z^{k}W^{2n-1-k}\ \longmapsto\ \sum_{k=0}^{2n-1}\overline{p_{2n-1-k}}Z^{k}W^{2n-1-k} \in V_\C\,.\]

Consider the action of $SO(2)$ on $\mathbb{C}^{(2n-1)}[Z,W]$ such that for $\theta \in \mathbb{R}$,

\[ R_\theta\cdot Z=Ze^{i \theta},\quad R_\theta\cdot W=We^{-i \theta}\,.\]

 Recall that $R_\theta$ is the rotation of angle $\theta$ defined in \eqref{eq:R theta}.
This action preserves $V_\mathbb{R}$, and corresponds in the variables $X,Y$ to the action of $SO(2)$ on $V_\mathbb{C}$ or $V_\mathbb{R}$ defined by :
\[ R_\theta\cdot X=R_\theta\cdot (Z+W)=\cos(\theta)X+\sin(\theta)Y=X\circ R_\theta^{-1}\,,\]
\[ R_\theta\cdot Y=R_\theta\cdot (Z+W)=\cos(\theta)Y-\sin(\theta)X=Y\circ R_\theta^{-1}\,.\]

This action of $SO(2)$ on $V_\mathbb{K}$ is the restriction to $SO(2)$ of the action defined in (\ref{eq:action on polynomials}) on $SL(2, \mathbb R)$, i.e. to the irreducible representation $\bar{\iota}:SL(2, \mathbb R)\to SL(2n, \mathbb{K})$.

\subsection{Definition of an invariant quadratic form.}

Let $\lambda=(\lambda_0,\cdots,\lambda_{n-1})\in \mathbb{R}_{>0}^n$. Define on $V_\mathbb{C}$ a symmetric $\mathbb{R}$-bilinear form $q_\lambda$, taking values in $\mathbb{C}$, by 
\begin{equation}
q_\lambda(P,Q)=\frac{1}{2}\sum_{k=0}^{n-1}\lambda_k p_{2k}\overline{q_{2k+1}}+\lambda_k q_{2k}\overline{p_{2k+1}}\,,
\end{equation}
where 
\[P=\sum_{k=0}^{2n-1}p_kZ^{k}W^{2n-1-k},\hspace{2cm} Q=\sum_{k=0}^{2n-1}q_kZ^{k}W^{2n-1-k}\,,\]
and for $0\leq k\leq 2n-1$, $p_k,q_k\in \mathbb{C}$.

We define $C^\lambda_\mathbb{K}=\{\ P\in V_\mathbb{K}\ \mid\ q_\lambda(P,P)=0\ \}$. For any $\lambda \in \mathbb{R}_{>0}^n$, this set can be identified with $C'_\mathbb{K}$ via some linear isomorphism $V_\mathbb{K}\to \mathbb{K}^{2n} $.

\begin{lem}
\label{lem:Invariance}
Let $\theta\in \mathbb{R}$, and $P, Q\in V_\mathbb{C}$, then 
$$q_\lambda(R_\theta\cdot P,R_\theta\cdot Q)=e^{2i\theta}q_\lambda(P,Q).$$
In particular, $C^\lambda_{\mathbb{K}}$ is $SO(2)$-invariant.
\end{lem}

\begin{proof}
Let $P,Q\in V_\mathbb{C}$ be two polynomials, and let us write them as 
$$P=\sum_{k=0}^{2n-1}p_kZ^{k}W^{2n-1-k},\quad Q=\sum_{k=0}^{2n-1}q_kZ^{k}W^{2n-1-k},$$
where for $0\leq k\leq 2n-1$, $p_k,q_k\in \mathbb{C}$. The polynomials $R_\theta\cdot P$ and $R_\theta\cdot Q$ hence can be written as 
$$P=\sum_{k=0}^{2n-1}p_k e^{(2n-1-2k)i\theta}Z^{k}W^{2n-1-k},\quad Q=\sum_{k=0}^{2n-1}q_k e^{(2n-1-2k)i\theta}Z^{k}W^{2n-1-k}.$$

Therefore one can compute $q_\lambda(R_\theta\cdot P,R_\theta\cdot Q)$:
\begin{equation}
\begin{split}
q_\lambda(R_\theta\cdot P,R_\theta\cdot Q) &  =\frac{1}{2}\sum_{k=0}^{n-1}\lambda_k p_{2k}\overline{q_{2k+1}} e^{(2n-1-4k-(2n-1-4k-2))i\theta} \\ 
& +\frac{1}{2}\sum_{k=0}^{n-1}\lambda_k q_{2k}\overline{p_{2k+1}} e^{(2n-1-4k-(2n-1-4k-2))i\theta},
\end{split}
\end{equation}
which implies the statement.
\end{proof}

\subsection{Action of a diagonal element.}
The action of $SL(2,\R)$ on the algebra $\C[X,Y]$ defined in (\ref{eq:action on polynomials}) restricts, for every $n$, to an action on the homogeneous component $\C^{(n)}[X,Y]$. This restriction gives an irreducible representation
\[SL(2, \mathbb R) \ \lra \ SL(n,\C)\,.\]
The differential of this map at the identity is a Lie algebra representation
\[\mathfrak{sl}(2,\mathbb{R}) \ \lra \ \mathfrak{sl}(n,\mathbb{C})\,. \]
We identify $\mathfrak{sl}(n,\mathbb{C})$ with the Lie algebra of traceless endomorphisms $\mathrm{End}_0(\C^{(n)}[X,Y])$. When we put all these representations together, we get a linear action of $\mathfrak{sl}(2,\mathbb{R})$ on the algebra of polynomials $\C[X,Y]$. For an element $\mathbf{r} \in \mathfrak{sl}(2,\mathbb{R})$, and $P\in \C[X,Y]$, we will denote by $\mathbf{r}\cdot P$ the action by this representation. This action does not preserve the product of polynomials, it is, instead a derivation:

\begin{lem}
For all $\mathbf{r} \in \mathfrak{sl}(2,\mathbb{R})$, the map
\[ \C[X,Y] \ni P \ \longmapsto \ \mathbf{r} \cdot P \in \C[X,Y]\]
is a derivation, i.e. 
\[\mathbf{r} \cdot (PQ) = (\mathbf{r} \cdot P)Q + P(\mathbf{r} \cdot Q)\,.\]
\end{lem}
\begin{proof}
The action of $SL(2,\R)$ on $\C[X,Y]$ preserves the product: for all $g \in SL(2,\R)$, we have $g(PQ) = g(P)g(Q)$. Now let's consider the elements $e^{t\mathbf{r}} \in SL(2,\R)$. We have 
\[e^{t\mathbf{r}}(PQ) = e^{t\mathbf{r}}(P)e^{t\mathbf{r}}(Q)\,.\]
Now
\[\mathbf{r} \cdot (PQ) = \frac{d}{dt}e^{t\mathbf{r}}(PQ) = \frac{d}{dt}\left(e^{t\mathbf{r}}(P)e^{t\mathbf{r}}(Q)\right) = (\mathbf{r} \cdot P)Q + P(\mathbf{r} \cdot Q)\,. \]
\end{proof}

This representation is compatible with the action of $SL(2, \mathbb R)$ in the sense that for $g\in SL(2, \mathbb R)$ and $\mathbf{h}\in \mathfrak{sl}_2(\mathbb{R})$, the action of $Ad_g(\mathbf{h})$ is the conjugate of the action of $\mathbf{h}$ by the action of $g$. Moreover $\R[X,Y]$ is preserved by this representation.

For $t\in \mathbb{R}$, we define
\[g_t=\begin{pmatrix} e^{-t}& 0 \\ 0 &e^t \end{pmatrix}\in SL(2, \mathbb R)\,.\]
\[\mathbf{g}_0 = \begin{pmatrix} -1& 0 \\ 0 &1 \end{pmatrix}\in \mathfrak{sl}(2, \mathbb R)\,.\]
Note that $\mathbf{g}_0$ is the differential of $t\mapsto g_t$ at $t=0$. Let $P$ be an element of $V_{\mathbb C}$, then $\mathbf{g}_0\cdot P=\frac{d}{dt}_{|t=0} g_t \cdot P$.  

\begin{lem}
Let $P\in V_\mathbb{C}$, then $\mathbf{g}_0\cdot P= \frac{\partial P}{\partial X}X-\frac{\partial P}{\partial Y}Y$. Moreover, after the change of variables, $\mathbf{g}_0\cdot P= \frac{\partial P}{\partial Z}W+\frac{\partial P}{\partial W}Z$.
\end{lem}
\begin{proof}
The map which to $P$ associate $\frac{\partial P}{\partial X}X-\frac{\partial P}{\partial Y}Y$ and $\frac{\partial P}{\partial Z}W+\frac{\partial P}{\partial W}Z$ are derivations. Hence it is sufficient to check that they coincide with $P\mapsto \mathbf{g}_0\cdot P$ on a generating set of the algebra of homogenous polynomials.

One can check it for the first fomula : $ \mathbf{g}_0\cdot X=X=\frac{\partial X}{\partial X}X-\frac{\partial X}{\partial Y}Y$ and $ \mathbf{g}_0\cdot Y=-Y=\frac{\partial Y}{\partial X}X-\frac{\partial Y}{\partial Y}Y$. 

As for the second formula, we have
 \[\frac{\partial (Z+W)}{\partial Z}W+\frac{\partial (Z+W)}{\partial W}Z=Z+W\,,\]
 \[\frac{\partial i(Z-W)}{\partial Z}W+\frac{\partial i(Z-W)}{\partial W}Z=i(W-Z)\,,\]
 and the proof is finished.
\end{proof}

Recall that $C^\lambda_\mathbb{K}=\{\ P\in V_\mathbb{K}\ \mid\ q_\lambda(P,P)=0\ \}$. Given a point $P \in C^\lambda_\mathbb{K}$, the tangent space to $C^\lambda_\mathbb{K}$ at $P$ is given by the linear equation
\[ T_P C^\lambda_\mathbb{K} = \{\ Q\in V_\mathbb{K}\ \mid\ q_\lambda(P,Q)=0\ \}\,. \]

\begin{df}
We will say that an element $\lambda=(\lambda_0,\cdots,\lambda_{n-1})\in \mathbb{R}_{>0}^n$ such that $\lambda_i=\lambda_{n-i}$ for $1\leq i<n$ is \emph{transverse to} $\mathbf{g}_0$ if for all $P\in V_\mathbb{C}\setminus \{0\}$, then $\mathrm{Re}(q_\lambda(P, \mathbf{g}_0\cdot P))> 0$
\end{df}

We now are going to show that a special choice of $\lambda$ is indeed transverse to $\mathbf{g}_0$. 

\begin{lem}   \label{lem:real to complex}
Let $\lambda=(\lambda_0,\cdots,\lambda_{n-1})\in \mathbb{R}_{>0}^n$ be such that $\lambda_i=\lambda_{n-i}$ for $1\leq i<n$ . If $\mathrm{Re}(q_\lambda(P, \mathbf{g}_0\cdot P))> 0$ for any $P\in V_\mathbb{R}\setminus \{0\}$, then this also holds for any $P\in V_\mathbb{C}\setminus \{0\}$.
\end{lem}

\begin{proof}
Let $P\in V_\mathbb{C}\setminus 0$. This polynomial can be written as $P=A+iB$, with $A,B\in V_\mathbb{R}$, $(A,B)\neq (0,0)$. Thus
$$q_\lambda(P,\mathbf{g}_0\cdot P)=q_\lambda(A,\mathbf{g}_0\cdot A)+q_\lambda(iB,\mathbf{g}_0\cdot iB)+q_\lambda(A,\mathbf{g}_0\cdot iB)+q_\lambda(iB,\mathbf{g}_0\cdot A).$$

But $\mathrm{Re}\left(q_\lambda(A,\mathbf{g}_0\cdot A)\right)>0$, $\mathrm{Re}\left(q_\lambda(iB,i\mathbf{g}_0\cdot B)\right)>0$, and
 $$\mathrm{Re}\left(q_\lambda(A,i\mathbf{g}_0\cdot B)\right)=\mathrm{Re}\left(q_\lambda(iB,\mathbf{g}_0\cdot A)\right)=0.$$
Indeed one can write 
$$A=\sum_{k=0}^{2n-1}a_kZ^{k}W^{2n-1-k},\quad \mathbf{g}_0\cdot B=\sum_{k=0}^{2n-1}b_kZ^{k}W^{2n-1-k},$$
where for  $0\leq k\leq 2n-1$, $a_k,b_k\in \mathbb{C}$ and $a_k=\overline{a_{2n-1-k}}$, $b_k=\overline{b_{2n-1-k}}$. Then
 $$\mathrm{Re}\left(q_\lambda(A,i\mathbf{g}_0\cdot B)\right)=\frac{1}{2}\sum_{k=0}^{n-1}(-i)\lambda_ka_{2k} \overline{b_{2k+1}}+i \lambda_k b_{2k} \overline{a_{2k+1}}=0.$$
And the same holds for $\mathrm{Re}\left(q_\lambda(B,i\mathbf{g}_0\cdot A)\right)$.
Hence $\mathrm{Re}(q_\lambda(P,\mathbf{g}_0\cdot P))>0$.
\end{proof}

Assume we are in the special case where for $0\leq k\leq n-1$, $\lambda_k=\lambda_{n-1-k}$. We will write a $P\in V_\mathbb{R}$ as 
\[P=\sum_{k=0}^{2n-1}p_k Z^{k}W^{2n-1-k}\,,\]
where for $1\leq k\leq 2n-1$, $p_k\in \mathbb{C}$ and $p_k=\overline{p_{2n-1-k}}$. Let us compute $ \mathbf{g}_0\cdot P$:
\[ \mathbf{g}_0\cdot P=\sum_{k=0}^{2n-1}((k+1)p_{k+1}+p_{k-1}(2n-k))Z^{k}W^{2n-1-k}\,.\]
Hence
\begin{equation} \label{eq:real q lambda}
q_\lambda(P, \mathbf{g}_0\cdot P)=\sum_{k=0}^{n-1} \lambda_k ((2n-2k-1)p_{2k}\overline{p_{2k}}+(2k+2)p_{2k}\overline{p_{2k+2}})\,.
\end{equation}

\begin{lem} \label{lem:case2lambdatransverse}
Suppose $\lambda$ is defined in the following way: for $0\leq k\leq n-1$, 
\[\lambda_k=\binom{2n-1}{2k}^{-\frac{1}{2}}\binom{2n-1}{2k+1}^{-\frac{1}{2}}\,.\] 
Then $\lambda$ is transverse to $\mathbf{g}_0$.
\end{lem}
\begin{proof}
Chose $b_k$ and $c_k$ for $0\leq k<n$ as follows: \[b_k^2=(2n-2k-1)(2n-2k)\cdot \frac{2k}{2k+1}, \quad c_k^2=(2n-2k-2)(2n-2k-1)\cdot \frac{2k+2}{2k+1}.\] 

Theses parameters satisfy two inequalities. Firstly, 

\begin{eqnarray*}
&&2(2n-1-2k)^2-(b_k^2+c_k^2)\\
&=&\frac{2n-2k-1}{2k+1}[2(2n-2k-1)(2k+1)-(2n-2k)(2k)-(2n-2k-2)(2k+2)]\\
&=&\frac{2(2n-2k-1)}{2k+1}>0.
\end{eqnarray*} Hence $2n-2k-1>\sqrt{\frac{b_k^2+c_k^2}{2}}\geq \frac{b_k+c_k}{2}$, which proves:

 \begin{equation}\label{Equation1}2n-2k-1\geq \frac{1}{2}(b_k+c_k).
 \end{equation}
 
Secondly, since \[\frac{\lambda_{k+1}^2}{\lambda_k^2}=\frac{\binom{2n-1}{2k}\binom{2n-1}{2k+1}}{\binom{2n-1}{2k+2}\binom{2n-1}{2n+3}}=\frac{(2k+1)(2k+2)^2(2k+3)}{(2n-2k-2)(2n-2k-3)^2(2n-2k-4)},\]
\begin{eqnarray*}\frac{(2k+2)^4}{b_{k+1}^2c_k^2}&=&\frac{(2k+2)^4}{[(2n-2k-3)(2n-2k-2)\cdot \frac{2k+2}{2k+3}][(2n-2k-2)(2n-2k-1)\cdot \frac{2k+2}{2k+1}]}\\
&=&\frac{(2k+1)(2k+2)^2(2k+3)}{(2n-2k-1)(2n-2k-2)^2(2n-2k-3)}.
\end{eqnarray*}

Hence one has :

\begin{equation}\label{Equation2}\frac{\lambda_{k+1}^2}{\lambda_k^2}> \frac{(2k+2)^4}{b_{k+1}^2c_k^2}.\end{equation}

By Lemma \ref{lem:real to complex} we can assume that $P\in V_\R$. Using \eqref{eq:real q lambda} and \eqref{Equation1}, one gets:
\begin{eqnarray*}
&&\mathrm{Re}(q_{\lambda}(P, \mathbf{g}_0\cdot P))\\
&\geq& \sum_{k=0}^{n-1}\frac{1}{2}\lambda_k(b_k+c_k)p_{2k}\overline{p_{2k}}-\lambda_k(2k+2)|p_{2k}||p_{2k+2}|\\
&\geq&\sum_{k=0}^{n-2}\frac{1}{2}(\lambda_{k+1}b_{k+1}|p_{2k+2}|^2+\lambda_kc_k|p_{2k}|^2)-\lambda_k(2k+2)|p_{2k}||p_{2k+2}|\\
&\geq&\sum_{k=0}^{n-2}(\sqrt{\lambda_k\lambda_{k+1}b_{k+1}c_k}-\lambda_k(2k+2))|p_{2k}||p_{2k+2}|\\
&=&\sum_{k=0}^{n-2}(\sqrt{\frac{\lambda_{k+1}}{\lambda_k}}-\frac{2k+2}{\sqrt{b_{k+1}c_k}})\sqrt{b_{k+1}c_k}\lambda_k|p_{2k}||p_{2k+2}|.
\end{eqnarray*}
Since $b_k, c_k$ satisfy  \eqref{Equation2}, $\sqrt{\frac{\lambda_{k+1}}{\lambda_k}}-\frac{2k+2}{\sqrt{b_{k+1}c_k}}> 0$ and therefore $\mathrm{Re}(q_{\lambda}(P, \mathbf{g}_0\cdot P))\geq0$; equality holds if and only if $p_{2k}p_{2k+2}=0$ and $\lambda_{k+1}b_{k+1}|p_{2k+2}|=\lambda_kc_k|p_{2k}|$ for $0\leq k\leq n-2$, together with the condition $p_k=\overline{p_{2n-1-k}}$ for $0\leq k\leq 2n-1$, we have $P=0$. 
\end{proof}

\subsection{Construction of the fibration.}

Let $F^\lambda_\mathbb{K}=\mathbb{P}(C^\lambda_\mathbb{K})$ and $SF^\lambda_\mathbb{K}=\mathbb{S}(C^\lambda_\mathbb{K})$ be the projectivization in $\mathbb{P}(V_\mathbb{K})$ of $C^\lambda_\mathbb{K}$ and the associated set in the sphere $\mathbb{S}(V_\mathbb{K})$. 

\medskip

Let $\tilde{U}^\lambda_{\mathbb{K}}$ be the bundle associated to the $SO(2)$-principal bundle $T^1\mathbb{H}^2$ defined by the action of $SO(2)$ on $F^\lambda_\K$. One can describe it as is the quotient of $PSL(2, \mathbb R)\times F^\lambda_\mathbb{K}$ by the action of $SO(2)$ defined by $r \cdot (g,f)=(g\circ r,r^{-1}\cdot f)$ for $r\in SO(2)$ and $(g,f)\in PSL(2, \mathbb R)\times F^\lambda_\mathbb{K}$. 

Let similarly $S\tilde{U}^\lambda_{\mathbb{K}}$ bundle associated to the $SO(2)$-principal bundle $T^1\mathbb{H}^2$ defined by the action of $SO(2)$ on $SF^\lambda_\K$.

\smallskip

The bunldes $\tilde{U}^\lambda_{\mathbb{K}}$ and $S\tilde{U}^\lambda_{\mathbb{K}}$ are diffeomorphic to $\mathcal{S}_2\times F^\lambda_\mathbb{K}$ and $\mathcal{S}_2\times SF^\lambda_\mathbb{K}$ respectively, with $\mathcal{S}_2\simeq \mathbb{H}^2$ the space of symmetric elements of $PSL(2, \mathbb R)$.

\medskip

Let us define the map $\phi_0: PSL(2, \mathbb R)\times F^\lambda_\mathbb{K} \to \mathbb{P}(V_\mathbb{K})$ which to an element $(g,f)\in PSL(2, \mathbb R)\times F^\lambda_\mathbb{K}$ associates $g\cdot f$. This defines a map on the quotient $\phi:\tilde{U}^\lambda_{\mathbb{K}}\to \mathbb{P}(V_\mathbb{K})$. We can define in the same way a map $S\phi:S\tilde{U}^\lambda_{\mathbb{K}}\to \mathbb{S}(V_\mathbb{K})$.

\medskip

Let $\rho$ be a Fuchsian representation into $PSL_{2n}(\mathbb{R})$. It can be written $\iota\circ \rho_0$ for some Fuchsian representation $\rho_0$ into $PSL(2, \mathbb R)$. Recall that $\Omega_\mathbb{K}$, $K_\mathbb{K}$, $S\Omega_\mathbb{K}$ and $SK_\mathbb{K}$ were defined in Section \ref{sec:dod}. Let $U^\lambda_{\mathbb{K}}$ be the quotient of $\tilde{U}^\lambda_{\mathbb{K}}$ by the action defined by $\rho_0$, and $SU^\lambda_{\mathbb{K}}$ be the quotient of $S\tilde{U}^\lambda_{\mathbb{K}}$.

\begin{thm}\label{thm:Fibration}
Suppose that $\lambda$ is transverse to $\mathbf{g}_0$. Then the map $\phi$ induces a diffeomorphism between $\tilde{U}^\lambda_{\mathbb{K}}$  and the domain of discontinuity $\Omega_\mathbb{K}$. The map $S\phi$ induces a diffeomorphism between $S\tilde{U}^\lambda_{\mathbb{K}}$  and $S\Omega_\mathbb{K}$.

These diffeomorphisms are $PSL(2, \mathbb R)$-equivariant. In particular the quotient $M_\mathbb{K}$ of $\Omega_\mathbb{K}$ by the action of $\rho$ is homeomorphic to $U^\lambda_{\mathbb{K}}$ which is a fibration over $S_g$, whose fiber is the Stiefel manifold $F_\mathbb{K}^n$ described in Section \ref{sec:manifolds}. Similarly the quotient $SM_\mathbb{K}$ of $S\Omega_\mathbb{K}$ by the action of $\rho$ is homeomorphic to $SU^\lambda_{\mathbb{K}}$.
\end{thm}

Two intermediate lemmas are necessary to prove this theorem. 

\begin{lem}\label{NonIntersect}
Suppose that $\lambda$ is transverse to $\mathbf{g}_0$. Then 
\[K_\mathbb{K}\cap F^\lambda_\mathbb{K}=\emptyset, \hspace{2cm} SK_\mathbb{K}\cap SF^\lambda_\mathbb{K}=\emptyset,.\]
\end{lem}

\begin{proof}

Consider an element $P\in C^\lambda_\mathbb{K}\setminus \{0\}$. Let $f$ be the function $t\in \mathbb{R}\mapsto \mathrm{Re}(q_\lambda(g_t\cdot P,g_t\cdot P))$. For $t\in \mathbb{R}$, the derivative 
\[f'(t)=2\mathrm{Re}\left(q_\lambda(g_t\cdot P,\mathbf{g}_0\cdot(g_t\cdot P))\right)\]
is strictly positive because of transversality. Hence $f$ is strictly increasing. Since $f(0)=0$ then for all $t<0$, $f(t)<0$. 

However on $\mathbb{P}(V_\mathbb{K})$, the line $[g_t\cdot P]$ converges when $t$ tends to $-\infty$ towards the line $[X^aY^{2n-1-a}]$ for some $0\leq a\leq 2n-1$. One can deduce that $\mathrm{Re}(q_\lambda)$ is non-positive on $[X^aY^{2n-1-a}]$. By transversality of $\lambda$, we have:
$$\mathrm{Re}(q_\lambda(X^aY^{2n-1-a}, \mathbf{g}_0\cdot X^aY^{2n-1-a}))>0.$$ 

Since one has $\mathbf{g}_0\cdot X^aY^{2n-1-a}=(2a-2n+1)X^aY^{2n-1-a}$, the quadratic form $\mathrm{Re}(q_\lambda)$ has the same sign as $\frac{1}{2a-2n+1}$ on the line $[X^aY^{2n-1-a}]$. Since this form has to be non-positive, therefore $2a-2n+1\leq 0$, and hence $a<n$.

But $a\geq n$ if and only if $P$ admits $[X,Y]=[0,1]$ as a root of multiplicity $\geq n$. Moreover since $F^\lambda_\mathbb{K}=\mathbb{P}(C^\lambda_\mathbb{K})$ and $SF^\lambda_\mathbb{K}=\mathbb{S}(C^\lambda_\mathbb{K})$ are $SO(2)$-invariant, then $K_\mathbb{K}\cap F^\lambda_\mathbb{K}=\emptyset$ and  $SK_\mathbb{K}\cap SF^\lambda_\mathbb{K}=\emptyset$.
\end{proof}

From now on, in order to prove Theorem \ref{thm:Fibration} we consider only the projective case. The proof in the spherical case is similar.

\begin{lem}
The map $\phi$ is a local diffeomorphism.
\end{lem}

\begin{proof}
Consider some $(\text{Id},[f_0])\in \tilde{U}^\lambda_\mathbb{K}\simeq \mathcal{S}_2\times F^\lambda_\mathbb{K}$ and $(\mathfrak{g},f)\in T_{(\text{Id},[f_0])}\tilde{U}^\lambda_\mathbb{K}$ such that $T_{(\text{Id},[f_0])}\phi(\mathfrak{g},f)=0$. Since $\mathfrak{g}$ is symmetric, there exists $r\in PSO(2)$ and $\ell\in \mathbb{R}^+$ such that $\mathfrak{g}=\ell \times r\mathbf{g}_0r^{-1}$. Then
\[T_{(\text{Id},[f_0])}\phi(\mathfrak{g},f)= \ell \times (r\mathbf{g}_0r^{-1}\cdot f_0)+f\,.\]

Hence since $f\in T_{[f_0]} F^\lambda_\mathbb{K}$, then $\ell r\mathbf{g}_0(r^{-1}\cdot f_0)$ also belongs to $T_{[f_0]} F^\lambda_\mathbb{K}$. However, by transversality, $q_\lambda(\mathbf{g}_0\cdot (r^{-1}\cdot f_0),r^{-1}\cdot f_0)\neq 0$ since its real part is strictly positive. Therefore $\mathbf{g}_0(r^{-1}\cdot f_0) \notin T_{r^{-1}\cdot [f_0]} F^\lambda_\mathbb{K}$.

Since $C^\lambda_\mathbb{K}$ is $SO(2)$-invariant, then 
$$r\mathbf{g}_0(r^{-1}\cdot f_0)\notin T_{[f_0]} F^\lambda_\mathbb{K}.$$
 
Hence $\ell=0$, and therefore $f=0$. Consequently $T_{(\text{Id},[f_0])}\phi$ is injective between spaces of the same dimension, that is $2n-1$ if $\mathbb{K}=\mathbb{R}$ and $4n-2$ if $\mathbb{K}=\mathbb{C}$. Hence $\phi$ is a local diffeomorphism in a neighborhood of any point of $M^\mathbb{K}$ of the form $(\text{Id},[f_0])$.

Let $g\in PSL(2, \mathbb R)$ and $[f_0]\in F^\lambda_\mathbb{K}$, $g^{-1}\cdot \phi(g,[f_0])=\phi(\text{Id},[f_0])$, Hence $\phi$ is a local diffeomorphism on the whole space $\tilde{U}^\lambda_\mathbb{K}$.
\end{proof}

We can now show Theorem \ref{thm:Fibration}.

\begin{proof}[Proof of Theorem \ref{thm:Fibration}]
Suppose $[g_1,f_1],[g_2,f_2]\in \tilde{U}_\mathbb{K}^\lambda$ satisfy $\phi( [g_1,f_1])=\phi([g_2,f_2])$. Then let us denote $h=g_2^{-1}g_1$, so that $h\cdot f_1=f_2$. Up to choosing adequate representatives of $[g_1,f_1]$ and $[g_2,f_2]$, one can assume that $h=g_t$ for some $t\geq 0$. But in this case if $t\neq 0$, the transversality of $\lambda$ implies that $0=\mathrm{Re}(q_\lambda(f_2,f_2))>\mathrm{Re}(q_\lambda(f_1,f_1))=0$. Therefore $t=0$, and consequently $g_1=g_2$ and $f_1=f_2$, that is, $\phi$ is injective.

\medskip

The map $\phi$ is a local diffeomorphism. Hence $\phi$ is an open map.

Since $\rho$ comes from a Fuchsian representation, $U^\lambda_\mathbb{K}$ is compact. Hence its image by $\phi$ in the compact $M_\mathbb{K}$ is compact, and in particular closed. The image of $\tilde{U}_\mathbb{K}^\lambda$ by $\phi$ is $PSL(2, \mathbb R)$-invariant, therefore it is closed in $\mathbb{P}(V_\mathbb{K})\setminus K_\mathbb{K}$. Hence $\phi$ is a closed map.

For $n\geq 3$ and $\mathbb{K}=\mathbb{R}$,  or $n\geq 2$ and $\mathbb{K}=\mathbb{C}$, the space $\Omega_\mathbb{K}$ is connected, so $\phi$ which is open and closed must be subjective. For $n=2$, and $\mathbb{K}=\mathbb{R}$, $F^\lambda_\K$ has two connected component, and $\Omega_\mathbb{K}$ has also two connected components. Since $\phi$ is injective open and closed, it must be surjective. In both cases $\phi$ is surjective.

Hence $\phi$ is a $PSL(2, \mathbb R)$-invariant diffeomorphism. Moreover $U_\mathbb{K}^\lambda$ is a fibration over $\mathcal{S}_2\simeq\mathbb{H}^2$ whose fibers are diffeomorphic to $F^\lambda_\mathbb{K}$. The space $F^\lambda_\mathbb{K}$ is a standard projective embedding of the projective Stiefel manifold.
\end{proof}

\section{Comparison with the diagonal representation}  \label{sec:ComparisonDiagonal}

Let $\rho_0:\pi_1(S)\rightarrow SL(2,\R)$ be a Fuchsian representation, and consider the diagonal embedding of $SL(2,\R)$ into $SL(2n,\R)$, with $(n\geq 2)$. This induces the following representation, that we will call a \emph{diagonal representation}:
\[\rho=\mathrm{diag}(\rho_0,\cdots, \rho_0):\pi_1(S)\ \rightarrow\ SL(2n,\R)\,.\] 
In this section, we apply to $\rho$ the construction of domains of discontinuity given by Guichard and Wienhard \cite{GW} and outlined in Section \ref{subsec:general_dod}. We explicitly describe domains $\Omega_\K\subset \K\P^{2n-1},  S\Omega_\K\subset\SS(\K^{2n})$ such that $\rho$ acts on $\Omega_\K,  S\Omega_\K$ properly discontinuously, freely and co-compactly. We then describe the quotient manifolds 
\[ W_\K := \rho(\pi_1(S))\backslash \Omega_\K, \hspace{2cm} SW_\K := \rho(\pi_1(S))\backslash S\Omega_\K\,.  \]
as fiber bundles over $S$. We will then show that these manifolds are often homeomorphic to the manifolds $SM_\K$, $M_\K$ obtained from a Hitchin or quasi-Hitchin representation.

\subsection{The diagonal representation}

The vector space $\K^{2n}$ is isomorphic to the tensor product $\K^n \otimes \K^2$. We see the elements of this tensor products as $n \times 2$ matrices, and multiply them using the usual matrix multiplication. We denote the isomorphism by
\[\eta: \K^n \otimes \K^2 \ni(v,w)=\begin{pmatrix}v_1&w_1\\v_2&w_2\\\vdots&\vdots\\v_n&w_n\end{pmatrix}\ \longmapsto\ \begin{pmatrix}v_1\\w_1\\v_2\\w_2\\\vdots\\w_n\end{pmatrix}\in \K^{2n}\,.\]

There is a natural left action of $SL(n,\K)\times SL(2,\K)$ on $\K^n \otimes \K^2$ defined by:
\begin{eqnarray}\label{Action}
\quad(SL(n,\K)\times SL(2,\K))\times \K^n \otimes \K^2 \ni (U,V)\cdot A\ \longmapsto\ U A V^{-1}\in\K^n \otimes \K^2\,.
\end{eqnarray}  

Using this action, the diagonal representation can be identified with the representation
\[ \rho: \pi_1(S) \ni \gamma \ \lra \ (\mathrm{Id}, \rho_0(\gamma)) \in SL(n,\K)\times SL(2,\K) \subset SL(2n,\K)\,.\] 

For this, we need to use the fact that an $SL(2,\R)$-representation $\rho_0$ is conjugate to its dual representation $\rho_0^*$. We will use the notation $\rho = (\mathrm{Id}, \rho_0)$.

\begin{prop}  \label{prop:topology of W}
Let $\rho=(\mathrm{Id},\rho_0)$ be a diagonal representation. Then there exist domains $\Omega_\K\subset \K\P^{2n-1}$ and $ S\Omega_\K\subset\SS(\K^{2n})$ such that $\rho$ acts on $\Omega_\K$ and $ S\Omega_\K$ properly discontinuously, freely and co-compactly. 

Moreover, the quotient manifolds $SW_\K$ and $W_\K$ are homeomorphic to fiber bundles over $S=\rho_0(\pi_1(S))\backslash \HH^2$ with fiber $SF_\K$ and $F_\K$ respectively, structure group $U(1)$ with action given by the geodesic flow $\delta$ as in (\ref{eq:geodesic flow}) and Euler class $g-1$. 
\end{prop}
\begin{rem}
As in Remark \ref{rem:different action}, the $U(1)$-action on $F_\K$ is not effective, the subgroup $\{\pm 1\}$ acts trivially. We have an effective action on $F_\K$ of the quotient group $U(1)/\{\pm 1\}$, that is still isomorphic to $U(1)$. We will always consider $W_\K$ as a $U(1)$-bundle with reference to this new structure group, and we notice that, with this new group, the bundle $W_\K$ has Euler class $2g-2$.  
\end{rem}
\begin{proof}
Denote by $\mathrm{Sym}^{2\times 2}(\R)$ the space of real symmetric $2\times 2$ matrices. Denote by $\SS \mathrm{Sym}^{2\times 2}(\R)$ the spherical quotient, i.e. the quotient by multiplication by $\R_{>0}$. We identify $\HH^2$ with the open subset of $\SS \mathrm{Sym}^{2\times 2}(\R)$ consisting of equivalence classes of positive definite symmetric $2\times 2$ matrices. In this model, the isometries of $\HH^2$ can be seen via the following left action of $SL(2,\R)$ on $\mathrm{Sym}^{2\times 2}(\R)$ that preserves $\HH^2$: 
\[ SL(2,\R) \times \mathrm{Sym}^{2\times 2}(\R) \ni (g,B) \ \longmapsto \ g\cdot B=(g^{-1})^T B g^{-1} \in \mathrm{Sym}^{2\times 2}(\R)\,. \]

We now want to define a projection to $\HH^2$. Consider the map $h$ defined by:
\begin{eqnarray*}
h: ~\K^n \otimes \K^2 \ni A\ \longmapsto\  \mathrm{Re}(\overline{A}^T A)\in \mathrm{Sym}^{2\times 2}(\R)\,.
\end{eqnarray*}

The element $h(A)$ is always a positive semi-definite symmetric matrix, and it is positive definite if and only if $A$ has full rank. Moreover, the projection $h$ is equivariant with respect to the action of $SL(2,\R)$: for $A\in\K^n \otimes \K^2$ and $g\in SL(2,\R)$, 
\begin{eqnarray*}
h((\mathrm{Id}, g)\cdot A)=h(Ag^{-1})=\mathrm{Re} (\overline{Ag^{-1}}^T(Ag^{-1}))&=&(g^{-1})^T\mathrm{Re}(\overline{A}^TA)g^{-1}\\
&=&(g^{-1})^Th(A)g^{-1}=g\cdot h(A).
\end{eqnarray*}

The map $h$ descends to the projectivized maps:
\begin{eqnarray*}
h_\P: ~\P(\K^n \otimes \K^2)\ \longrightarrow & \SS \mathrm{Sym}^{2\times 2}(\R),\quad h_\SS: ~\SS(\K^n \otimes \K^2)\longrightarrow&\SS \mathrm{Sym}^{2\times 2}(\R)\,.
\end{eqnarray*} 
 We define the following subsets of $\P(\K^n \otimes \K^2)$ and $\SS(\K^n \otimes \K^2)$:
\[\Omega_\K=\{\ [A]_{\P}\in \P(\K^n \otimes \K^2)\ \mid\ A \text{ is of rank 2}\ \}\,,\]
\[ S\Omega_\K=\{\ [A]_\SS\in \SS(\K^n \otimes \K^2)\ \mid\ A \text{ is of rank 2}\ \}\,.\] 
We then have the restricted maps $h_\P|_{\Omega_\K}: \Omega_\K\rightarrow \HH^2$ and $h_\SS|_{S\Omega_\K}: S\Omega_\K \rightarrow \HH^2$ respectively.

Since the action of $SL(2,\R)$ on $\HH^2$ is transitive, the restricted projections $h_\SS|_{S\Omega_\K}$ and $h_\P|_{\Omega_\K}$ are surjective onto $\HH^2$. 

Next we that $h_\SS|_{S\Omega_\K}$ and $h_\P|_{\Omega_\K}$ are fiber bundles over $\HH^2$ with fiber $SF_\K, F_\K$ respectively. By the transitivity of the action of $SL(2,\R)$, it is enough to understand the pre-image of the identity matrix $[\mathrm{Id}]$ in $\HH^2$. The equation for  $A\in \K^n \otimes \K^2$ to be in this fiber is 
\[\mathrm{Re}\left(\overline{A}^T A \right)=\mathrm{Id}\,,\]
equivalently, writing $A=(v,w)$,
\begin{equation}
\mathrm{Re}(\left< v, w\right>)=0,\quad \left< v, v \right>=\left< w, w\right>\,.
\end{equation}
Comparing with the definition of $SF_\K$ in Section \ref{sec:stiefel}, we see that $h_\SS^{-1}(\mathrm{Id})$ coincides with $SF_\K$ and moreover $h_\SS^{-1}(g\cdot \mathrm{Id})$ coincides with $(\mathrm{Id},g)\cdot SF_\K$ for $g\in SL(2,\R)$. Similarly, $h_\P^{-1}(g\cdot \mathrm{Id})$ coincides with $(\mathrm{Id},g)\cdot F_\K$, for $g\in SL(2,\R)$. 

We consider $SL(2,\R)$ as a principal $U(1)$-bundle over $\HH^2$. The associated fiber bundle over $\HH^2$ with fiber $SF_\K$ is $SL(2,\R)\times_{\delta} SF_\K$. We define the following map
\begin{eqnarray*}
\Psi: SL(2,\R)\times_{\delta} SF_\K \ni [(g,A)]\ \longmapsto\ (\mathrm{Id},g)\cdot A\in S\Omega_\K.
\end{eqnarray*}
The map $\Psi$ is well-defined since for any $U\in SO(2)$, $(gU, AU)$ also maps to $Ag^{-1}$. Since $\Psi$ takes the fiber at $g\cdot \mathrm{Id}$ of $SL(2,\R)\times_{\delta}SF_\K$ to the fiber at $g\cdot \mathrm{Id}$ of $ S\Omega_\K$, it is a bundle isomorphism between fiber bundles over $\HH^2$ with fiber $SF_\K$.
Passing to the quotient space by the action of $\pi_1(S)$, the fibration $h_\SS$ descends to a fibration $h_\SS:SW_\K\rightarrow S$ with fiber $SF_\K$ and $h_\P$ descends to a fibration $h_\P:W_\K\rightarrow S$ with fiber $F_\K$.
\end{proof}

\subsection{Comparison with a Hitchin representations}

We would like to compare the topology of the manifolds $SW_\K$ and $W_\K$ with the topology of the manifolds $SM_\K$ and $M_\K$ constructed from a Fuchsian representation. As fiber bundles over $S$, they have the same fibers and same Euler class, but they have different structure groups: the structure group of $SW_\K$ and $W_\K$ is given by the geodesic flow $\delta$ from (\ref{eq:geodesic flow}), while the structure group of $SU_\K$ and $U_\K$ is given by the representation $\phi$ from (\ref{eq:phi}).

\begin{prop}  \label{prop:diagonal-Hitchin}
As topological bundles over $S$, we have the following isomorphisms.
\begin{enumerate}
\item For $n\geq 3$, $M_\R\cong W_\R$;
\item For $n\geq 2$, $M_\C\cong W_\C$;
\item For $n\geq 3$ and the genus $g$ of $S$ is odd, $SM_\R\cong SW_\R$;
\item For $n\geq 2$ and the genus $g$ of $S$ is odd, $SM_\C\cong SW_\C$.
\end{enumerate} 
\end{prop}
\begin{proof}
We saw in Theorem \ref{thm:comparison} that $SM_\K$ and $M_\K$ are homeomorphic to $SU_\K$ and $U_\K$ respectively. By definition, 
\[SU_\K\cong P'\times_{\phi'}SF_\K', \quad U_{\K}\cong P\times_{\phi'}F_\K'\,,\]
where $P', P$ are principal $U(1)$-bundle on $S$, with Euler number $g-1$ and $2g-2$ respectively. By Lemma \ref{lemma:C-linear transformation}, using the identification between $SF_\K$ and $F_\K$, we have 
\[SU_\K\cong P'\times_{\phi}SF_\K, \quad U_{\K}\cong P\times_{\phi}F_\K\,,\]
Similarly, we saw in Proposition \ref{prop:topology of W} that 
\[SV_\K\cong P'\times_{\delta}SF_\K, \quad V_{\K}\cong P\times_{\delta}F_\K\,.\]
The proposition then follows from the following  
Lemma \ref{prop:topology of the bundles}. 
\end{proof}

\begin{lem}\label{prop:topology of the bundles}
As topological bundles over $S$, we have the following isomorphisms.
\begin{enumerate}
\item For every $n\geq 3$, $P\times_{\phi}F_\R\cong P\times_{\delta}F_\R$; 
\item For every $n\geq 2$, $P\times_{\phi}F_\C\cong P\times_{\delta}F_\C$;
\item For every $n\geq 3$ and the genus $g$ of $S$ is odd, $P'\times_{\phi}SF_\R\cong P'\times_{\delta}SF_\R$; 
\item For every $n\geq 2$ and the genus $g$ of $S$ is odd, $P'\times_{\phi}SF_\C\cong P'\times_{\delta}SF_\C^n$.
\end{enumerate}
\end{lem}
\begin{proof}
Recall that $\phi$ and $\delta$ are representations from $U(1)$ to $SO(n)\times SO(2)$ defined by (\ref{eq:geodesic flow}) and (\ref{eq:phi}). The projections on the two factors are $(L_\theta,R_\theta)$ for $\phi$, and $(\mathrm{Id},R_\theta)$ for $\delta$. Then, we have 
Since $\phi=(\phi_1,\phi_2)$, we can see that
\[P\times_{\phi}(SO(n) \times SO(2))= (P\times_{L_\theta}SO(n)) \times_S (P\times_{R_\theta}SO(2))\,.\]
\[P\times_{\delta}(SO(n) \times SO(2))= (P\times_{\mathrm{Id}}SO(n)) \times_S (P\times_{R_\theta}SO(2))\,.\]

We will just need to prove that, under our hypothesis, the bundle $P\times_{L_\theta}SO(n)$ is trivial as a principal $SO(n)$-bundle. For $n \geq 3$, principal $SO(n)$-bundles over a closed Riemann surface are classified by their second  Stiefel-Whitney class $w_2 \in H^2(S,\Z_2)=\Z_2$, which is the reduction modulo $2$ of the Euler class. So $w_2(P)=0$ since its Euler class as a $U(1)$-bundle is $2g-2$. Hence $P\times_{L_\theta}SO(n)$ is trivial as a principal $SO(n)$-bundle. Hence 
\[P\times_{\phi}(SO(n) \times SO(2)) \cong P\times_{\delta}(SO(n)\times SO(2) \]
Part 1 then follows by changing fiber to $F_\R.$

Similarly in the case of $F_\C$, using the extension by $SU(n)\times SO(2)$ instead of $SO(n)\times SO(2)$ and the fact that the bundle $P\times_{L_\theta}SU(n)$ is trivial as a principal $SU(n)$-bundle if $n\geq 2$, we can prove Part 2.

When $g-1$ is even, the above argument also works for the principal $U(1)$-bundle $P'$ of Euler class $g-1$ and we obtain the result in Part 3 and 4.
\end{proof}

\begin{rem} The second Stiefel-Whitney class is not enough to classify $SO(2)$-bundles, hence our lemma does not apply to the case $n=2$, $\K=\R$. In this case, the bundle $M_\R$ is the disjoint union of two circle bundles one of them isomorphic to the unit tangent bundle of $S$, and the other isomorphic to the unit circle bundle of the tricanonical bundle $K^3$.
\end{rem}

We saw that the bundles $SU_\K\ra S$ are non-trivial when considered as bundles with the structure group $SO(n)\times SO(2)$, or $SU(n) \times SO(2)$, because they have Euler class $g-1$ with reference to the second factor. A very natural question is whether these bundles become trivial when the structure group is extended to a larger group. This would imply that the total  space of the bundle is a product of $S$ by the fiber. We can prove this in a special case.
 
\begin{thm} If the genus $g$ of $S$ is odd, and $n=3$, then $SU_\R\cong SV_\R \cong S \times SF_ \R$.
\end{thm}  
\begin{proof}
We have an explicit identification between $SO(3)$ and $SF_\R$ where the first column of the matrix gives the vector $v$, and the second column gives the vector $w$.

The group $SO(3)\times SO(2)$ embeds into $SO(3)\times SO(3)$ by the homomorphism $i$ with image $SO(3)\times \text{diag}(SO(2),1)$ and the action of $SO(3)\times SO(2)$ on $SF_\R$ extends to an action of $SO(3)\times SO(3)$ on $SF_\R\cong SO(3)$ as 
\begin{equation} \label{SO(3)SO(3)action}
(SO(3) \times SO(3)) \times SO(3) \ni (A,B)\cdot C\longmapsto ACB^{-1} \in SO(3)\,. 
\end{equation}  

Recall $P'$ is a $U(1)$-bundle of Euler class $g-1$. 
Decompose the homomorphism $i\circ\phi:U(1)\rightarrow SO(3)\times SO(3)$ into $(\psi_1,\psi_2)$, the corresponding principal $SO(3)\times SO(3)$-bundle 
\begin{eqnarray*}
P'\times_{i\circ\phi}(SO(3)\times SO(3))&=&(P'\times_{\psi_1}SO(3))\times_S (P'\times_{\psi_2}SO(3))\\
&=&(S\times SO(3))\times_S(S\times SO(3))
\end{eqnarray*} 
is trivial since $P\times_{\psi_i}SO(3) (i=1,2)$ is  trivial as in Lemma \ref{prop:topology of the bundles}. Hence the bundle $SU_\R$ is also trivial by changing fiber to $SF_\R$ by the action of $SO(3)\times SO(3)$ on $SF_\R$. 
\end{proof}

\section{Description of \texorpdfstring{$SM_\K$}{SMK} and \texorpdfstring{$M_\K$}{MK}}\label{DescriptionOfFibers}

\subsection{Description of \texorpdfstring{$SF_\R$}{SFR} and \texorpdfstring{$F_\R$}{FR}}

We have natural maps from the spaces $SF_\R$ and $F_\R$ to the oriented Grassmannian $\mathrm{Gr}^+(2,\R^n)$:
\[SF_\R \ni [(v,w)]_\SS\ \rightarrow\ \mathrm{Span}(v,w) \in \mathrm{Gr}^+(2,\R^n)\,,\]
\[F_\R \ni [(v,w)]_\P\ \rightarrow\ \mathrm{Span}(v,w) \in \mathrm{Gr}^+(2,\R^n)\,.\]
In this way we can see the spaces $SF_\R, F_\R$ as circle bundles over the oriented Grassmannian: $SF_\R$ is the unit circle bundle associated to the tautological vector bundle over $ \mathrm{Gr}^+(2,\R^n)$ and $F_\R$ is the projectivized tautological vector bundle.

Let's see more explicitly the topology of these spaces in some special cases. 
\begin{itemize}
  \item For $n=1$, $SF_\R = F_\R = \mathrm{Gr}^+(2,\R) = \emptyset$. 
  \item For $n=2$, $SF_\R$ and $F_\R$ are both the disjoint union of two circles, and $\mathrm{Gr}^+(2,\R^2)$ is two points.
  \item For $n=3$, $SF_\R = \R\P^3$ and $F_\R$ is the Lens space $L(4,1)$ (see \cite{Ko}). From this we also see that $\pi_1(SF_\R) = \Z_2$ and $\pi_1(F_\R)= \Z_4$. In this case, $\mathrm{Gr}^+(2,\R^3) = \SS^2$.
  \item For $n \geq 4$, $\pi_1(SF_\R) = 0$, this follows from the long exact sequence of the fibration of $T^1 \SS^{n-1}$. Hence, in this case, $SF_\R$ is the universal covering of $F_\R$, and $\pi_1(F_\R^n)=\Z_2$. For $n\geq 3$, $\pi_1(\mathrm{Gr}^+(2,\R^n)) = 0$. 
  \item  For $n=4$, $SF_\R = \SS^2 \times \SS^3$, $F_\R = \SS^2 \times \R\P^3$ because $\SS^3 = SU(2)$ and $\RP^3 = SO(3)$ are Lie groups. $\mathrm{Gr}^+(2,\R^4) = \SS^2 \times \SS^2$.  
  \item For $n=8$, $SF_\R = \SS^6 \times \SS^7$ and $F_\R = \SS^6 \times \R\P^7$, because $\SS^7$ and $\R\P^7$ are parallelizable.   
  \item For all $n\neq 2,4,8$, the unit tangent bundles are non-trivial.
\end{itemize}

\subsection{Description of \texorpdfstring{$SF_\C$}{SFC} and \texorpdfstring{$F_\C$}{FC}}

As before, we can also see $SF_\C$ as the unit tautological circle bundle over the oriented Grassmannian $\mathrm{Gr}^+(2,\R^{2n})$:
\[SF_\C \ni [(v,w)]_\SS\ \rightarrow\ \mathrm{Span}_\R(v,w) \in \mathrm{Gr}^+(2,\R^{2n})\,. \]

From the following lemma, the space $F_\C$ is a sphere bundle over $\C\P^{n-1}$.
\begin{lem}   \label{lem:topology of FCn}
The space $F_\C$ is diffeomorphic to the total space of the sphere bundle 
\[\SS^{2n-2}\ \ra\ \SS(\Oo \oplus T)\ \ra\ \C\P^{n-1}\]
where $\Oo$ is a trivial real line bundle over $\C\P^{n-1}$ and $T = T \C\P^{n-1}$ is the tangent bundle.
\end{lem}
\begin{proof}
The map $SF_\C \ra \SS^{2n-1}$ is $U(1)$-equivariant, hence it induces a map between the quotients:
\[F_\C \ni [(v,w)]_\P\ \ra\ [v]_\P \in \C\P^{n-1}\,. \]
This map is a submersion between compact spaces, hence by Ehresmann's theorem it is a fiber bundle over $\C\P^{n-1}$ with fiber $\SS^{2n-2}$. We have the following commutative diagram
\[
\begin{tabular}{ccc}
$SF_\C$      & $\ra$ & $\SS^{2n-1}$ \\
$\downarrow$ &       &  $\downarrow$ \\
$F_\C$       & $\ra$ & $\C\P^{n-1}$,
\end{tabular}
\]
where the map $\SS^{2n-1} \ra \C\P^{n-1}$ is the Hopf fibration. This diagram shows that the map $F_\C  \ra \C\P^{n-1}$ is a bundle whose pull-back by the Hopf fibration is $T^1 \SS^{2n-1}$. 
The fibers of the Hopf fibration give a $1$-dimensional orientable foliation of the sphere $\SS^{2n-1}$. By choosing unit vectors parallel to the fibers, and hyperplanes orthogonal to these vectors, we define a decomposition of the tangent bundle to the sphere as $T \SS^{2n-1} = \Oo \oplus \Oo^\perp$, where the line bundle $\Oo$ is the kernel of the differential of the Hopf fibration. This decomposition can be seen explicitly in the following way. 
\begin{eqnarray*}
T \SS^{2n-1} & = & \left\{(v,w) \in \C^n\times\C^n \mid  \left|v\right|^2=1, \mathrm{Re}(\left<v,w\right>)=0  \right\} \\
\Oo          & = & \left\{(v,w) \in \C^n\times\C^n \mid  \left|v\right|^2=1, w \in \R (iv)  \right\} \\
\Oo^\perp    & = & \left\{(v,w) \in \C^n\times\C^n \mid  \left|v\right|^2=1, \left<v,w\right>=0  \right\}  
\end{eqnarray*}

We see from the formula that $\Oo^\perp / U(1)$ is the tangent bundle $T\C\P^{n-1}$, and $\Oo^\perp$ with the pull-back of the tangent bundle $T \C\P^{n-1}$. Hence the bundle $F_\C \ra \C\P^{n-1}$ is isomorphic to the sphere bundle $\SS(\Oo \oplus T \C\P^{n-1})$. 
\end{proof}

\begin{cor}
When $n=2$, the space $F_\C$ is diffeomorphic to $\SS^2 \times \SS^2$.
\end{cor}
\begin{proof}
$\C\P^1$ can be embedded in $\R^3$ with a trivial normal bundle, hence $\Oo \oplus T$ is the restriction to $\C\P^1$ of the tangent bundle of $\R^3$, which is trivial.
\end{proof}

For $n=1$, $F_\C$ is two points. For $n=2$, we saw that $F_\C$ is the trivial bundle $\SS^2 \times \SS^2$. For $n=3$ and $n\geq 5$, the $\SS^{2n-2}$-bundle $F_\C  \ra \C\P^{n-1}$ is necessarily non-trivial, because its pull-back by the Hopf fibration is $T^1 \SS^{2n-1}$, which is non-trivial. For $n=4$ we don't know whether the $\SS^{6}$-bundle $F_\C  \ra \C\P^{n-1}$ is trivial or not.

\subsection{Circle bundles over a product}
The space $F_\C$ can also be seen as a tautological space over a parameter space similar to a Grassmannian. To do this we notice that $U(1)$  also acts on the real Grassmannian $\mathrm{Gr}^+(2,\R^{2n})$ by
\[ U(1) \times \mathrm{Gr}^+(2,\R^{2n})\ni (e^{i\theta}, \mathrm{Span}_\R(v,w))\ \ra\ \mathrm{Span}_\R(e^{i\theta}v,e^{i\theta}w) \in \mathrm{Gr}^+(2,\R^{2n})\,, \]
where the multiplication by $e^{i\theta}$ is defined by an identification $\R^{2n} \simeq \C^n$. We need to remark that the map is well defined because the subspace $\mathrm{Span}_\R(e^{i\theta}v,e^{i\theta}w)$ does not depend on the choice of a basis $v,w$ of $\mathrm{Span}_\R(v,w)$. 

The quotient of the oriented Grassmannian by this action has a geometric interpretation as the \emph{Grassmannian of projective oriented circles}, the  parameter space of (possibly degenerate) oriented circles in $\C\P^{n-1}$. A \emph{circle} in $\C\P^1$ is the image of $\R\P^1$ by a \emph{M\"obius transformation} (an element of $PGL(2,\C)$). We define a \emph{circle} in $\C\P^{n-1}$ as a circle contained in some complex projective line in $\C\P^{n-1}$. The space of circles is not compact, because a sequence of circles with smaller and smaller radius does not converge to a circle, but it has a subsequence converging to a point. So we will consider points of $\C\P^{n-1}$ as \emph{degenerate circles}. \emph{Oriented circles} are circles with a privileged direction of travel. \emph{Oriented degenerate circles} are points with a sign ($+$ or $-$). This makes sense, because a ``very small'' circle has an interior and an exterior, and the natural orientation of $\C\P^1$, when seen from the interior, makes the direction of travel positive or negative. It is easy to see that the space of (possibly degenerate) oriented circles in $\C\P^{n-1}$ is parametrized by 
\[\mathrm{Circ}^+(\CP^{n-1}) = \mathrm{Gr}^+(2,\R^{2n})/U(1)\]
since the image in $\C\P^{n-1}$ if every real $2$-plane is a circle, or a point if the plane happens to be a complex line.

The space $\mathrm{Circ}^+(\CP^{n-1})$ is not a manifold, but in the next proposition we will see that it is a manifold with conical singularities (see Baas \cite{Baas}, where he reports Sullivan's definition of manifold with conical singularities).

\begin{prop} \label{prop:circle bundle over a product}
The open subset of $\mathrm{Circ}^+(\CP^{n-1})$ consisting of non-degenerate oriented circles is a manifold. The subset of degenerate oriented circles is homeomorphic to the disjoint union of two copies of $\CP^{n-1}$, and every degenerate circle has a  neighborhood in $\mathrm{Circ}^+(\CP^{n-1})$ homeomorphic to $\C^{n-1} \times C(\C\P^{n-2})$, where $C(\C\P^{n-2})$ is the cone over $\C\P^{n-2}$.   
\end{prop}
\begin{proof}
The $U(1)$-action is proper because the group is compact. Moreover the action is free on the open subset of the real $2$-planes which are not complex, hence the quotient of this subset if a manifold.

To understand the second statement, we use the description of the oriented Grassmannian by charts: for every complex subspace $V$ of complex dimension $n-1$, consider the open subset $\mathcal{U}_V$ of all the real $2$-planes transverse to $V$. The set $\mathcal{U}_V$ has two connected components $\mathcal{U}_V^1$, $\mathcal{U}_V^2$, containing the same $2$-planes with different orientations. A choice of a transverse complex line $W \in \mathcal{U}_V^1$ gives a diffeomorphism from the space of $\R$-linear applications $\mathrm{Hom}_\R(W,V)$ to $\mathcal{U}_V^1$, sending every $\R$-linear map to its graph. We will write $A \in \mathrm{Hom}_\R(W,V)$ as $A = A_\ell + A_{a\ell}$, where $A_{\ell}(v) = \tfrac{1}{2}(A(v) - i A(iv)) $ is a $\C$-linear map and $A_{a\ell}(v) = \tfrac{1}{2}(A(v) + i A(iv))  $ is a $\C$-anti linear map. Choosing an element $e_1 \in W$, we can give a more explicit chart using the identification:
\[\mathrm{Hom}_\R(W,V) \ni A\ \ra\ (A_{\ell}(e_1), A_{a\ell}(e_1)) \in \C^{n-1} \times \C^{n-1}\,. \]     
In these coordinates, it is easy to write the $U(1)$-action, it acts by:
\[U(1) \times (\C^{n-1} \times \C^{n-1}) (e^{i\theta}, (v,w))\ \ra\ (v,e^{i\theta}w) \in (\C^{n-1} \times \C^{n-1})\,. \]
Degenerate circles are the points where $A_{a\ell} = 0$, and we can see that they have a neighborhood homeomorphic to $\mathcal{U}_V^1 / U(1) = \C^{n-1} \times C(\C\P^{n-2}).$
\end{proof}

The proof of the previous theorem also shows how to describe $\mathrm{Circ}^+(\CP^{n-1})$ with charts analog to the charts for ordinary Grassmannians, except that the charts take values in the singular space $\C^{n-1} \times C(\C\P^{n-2})$.

The Grassmannian of projective circles has a tautological space:
\[\{ (p,c) \in \CP^{n-1} \times \mathrm{Circ}^+(\CP^{n-1})  \mid  p \in c \}\ \ra\ \mathrm{Circ}^+(\CP^{n-1})\,. \]
When restricted to the open subset of non-degenerate circles, this map is a circle bundle. Over a degenerate circle there is only one point. 

The space $F_\C$ has a map associating to a point $[(v,w)]_\P \in F_\C$ the circle image of $\mathrm{Span}_\R(v,w)$ in $\mathrm{Circ}^+(\CP^{n-1})$. With this map $F_\C \ra \mathrm{Circ}^+(\CP^{n-1})$, $F_\C$ is precisely the tautological space described above.

In general, we don't know whether the spaces $SU_\R, U_\R, SU_\C, U_\C$ are homeomorphic to products, but we can prove a weaker statement that they all have a codimension $1$ quotient which is a product.

\begin{thm}
There are maps 
\[p_{SU_\R}:SU_\R\ \ra\ S \times \mathrm{Gr}^+(2,\R^{n})\,,\]
\[p_{U_\R}:U_\R\ \ra\ S \times \mathrm{Gr}^+(2,\R^{n})\,,\]
\[p_{SU_\C}:SU_\C\ \ra\ S \times \mathrm{Gr}^+(2,\R^{2n})\,,\]
\[p_{U_\C}:U_\C\ \ra\ S \times \mathrm{Circ}^+(\CP^{n-1})\,,\]
where the maps $p_{SU_\R}, p_{U_\R}, p_{SU_\C}$ are circle bundles over $S \times \mathrm{Gr}^+(2,\R^{n})$ such that for every point $x$ in the Grassmannian, the inverse image of $S\times \{x\}$ is the unit tangent bundle of the surface, and for every point point $z\in S$, the inverse image of $z$ times the Grassmannian is the space $SF_\R, F_\R$ or $SF_\C$. The map $p_{U_\C}$ is a circle bundle only over the open dense subset of non-degenerate circles, while over the degenerate circles the fiber is just one point. 
\end{thm}
\begin{proof}
We have seen in Proposition \ref{prop:topology of the bundles} that as bundles with structure group $O(n)\times SO(2)$ or $U(n) \times SO(2)$,  $SU_\K, U_\K$ are isomorphic to $SV_\K, V_\K$. The latter bundles have a well defined map to the oriented Grassmannian, or to the Grassmannian of projective oriented circles. These maps can be defined in every bundle chart, and since the structure group $SO(2)$ does not change the map to the relevant Grassmannian, the map is well defined globally. 
\end{proof}

\end{document}